\documentclass[reqno,10pt,a4paper]{amsart}
\usepackage{amsmath,amssymb,amsthm,graphicx,mathrsfs,url}
\usepackage[usenames,dvipsnames]{xcolor}
\usepackage[colorlinks=true,linkcolor=Red,citecolor=Green]{hyperref}
\usepackage{amsxtra}
\usepackage{epstopdf}
\usepackage{esint}
\usepackage{subcaption}
\usepackage{tikz-cd}
\usepackage{tikz,tikz-cd}
\usepackage{scrextend}
\usepackage{times}
\usepackage{csquotes}
\def\arXiv#1{\href{http://arxiv.org/abs/#1}{arXiv:#1}}

\def\?[#1]{\textbf{[#1]}\marginpar{\Large{\textbf{??}}}}

\def\smallsection#1{\smallskip\noindent\textbf{#1}.}

\newcommand{\RR}{{\mathbb R}}

\usepackage[left=1.15in,right=1.15in,top=1.15in,bottom=1.15in]{geometry}

\newtheorem{prop}{Proposition}[section]	
\newtheorem{theo}[prop]{Theorem}

\newtheorem{lemm}[prop]{Lemma}

\numberwithin{equation}{section}

\theoremstyle{definition}

\newtheorem{defi}[prop]{Definition}
\newtheorem{Assumption}[prop]{Assumption}
\newtheorem{rem}[prop]{Remark}
\newtheorem{ex}[prop]{Example}

\let\Im=\Imag

\let\Re=\Real

\newcommand*{\MA}[1]{{\color{magenta}#1}}

\DefineNamedColor{named}{JungleGreen} {cmyk}{0.99,0,0.52,0}

\def\indic{\operatorname{1\hskip-2.75pt\relax l}}

\title[Bilinear- and nonlinear Schr\"odinger equations]{Computing solutions of Schr\"odinger equations on unbounded domains \\- On the brink of numerical algorithms}
\author{Simon Becker}
\email{simon.becker@damtp.cam.ac.uk}
\address{University of Cambridge,
DAMTP, Wilberforce Rd, Cambridge CB3 0WA, UK}
\author{Anders C. Hansen}
\email{ach70@cam.ac.uk}
\address{University of Cambridge,
DAMTP, Wilberforce Rd, Cambridge CB3 0WA, UK}
\begin{document}
\begin{abstract}
We address the open problem of determining which classes of time-dependent \emph{linear Schr\"odinger equations} and focusing and defocusing cubic and quintic \emph{non-linear Schr\"odinger equations (NLS)}  on unbounded domains that can be computed by an algorithm. We demonstrate how such an algorithm in general does not exist, yielding a substantial classification theory of which problems in quantum mechanics that can be computed.  Moreover, we establish classifications on which problems that can be computed with a uniform bound on the runtime, as a function of the desired $\epsilon$-accuracy of the approximation. This include linear and nonlinear Schr\"odinger equations for which we provide positive and negative results and conditions on both the initial state and the potentials such that there exist computational (recursive) a priori bounds that allow reduction of the IVP on an unbounded domain to an IVP on a bounded domain, yielding an algorithm that can produce an $\epsilon$-approximation. In addition, we show how no algorithm can decide, and in fact not verify nor falsify, if the focusing NLS will blow up in finite time or not, yet, for the defocusing NLS, solutions can be computed given mild assumptions on the initial state and the potentials. Finally, we show that solutions to discrete NLS equations (focusing and defocusing) on an unbounded domain can always be computed with uniform bounds on the runtime of the algorithm. The algorithms presented are not just of theoretical interest, but efficient and easy to implement in applications. Our results have implications beyond computational quantum mechanics and are a part of the Solvability Complexity Index (SCI) hierarchy and Smale's program on the foundations of computational mathematics. 
For example our results provide classifications of which mathematical problems may be solved by computer assisted proofs. 
\end{abstract}
\maketitle

\section{Introduction}
Since the pioneering work of B. Engquist and A. Majda \cite{Engquist_CPAM, EM} on absorbing boundary conditions (ABC), the problem of computing approximations to solutions of PDEs on unbounded domains has been notoriously challenging. The Schr\"odinger equation (linear or non-linear) is not an exception.  In particular, despite more than 90 years of the Schr\"odinger equations, the following question is open:
\vspace{1mm}
\begin{displayquote}
	\normalsize
	{\it 
	For which classes of Schr\"odinger equations, linear and non-linear (see \eqref{eq:bilSchr} or \eqref{eq:GP}) on unbounded domains, will there exist an algorithm $\Gamma$, taking point samples of the initial state $\varphi_0$ and the potential function $V$, such that for any $\epsilon > 0$ we have that $\Gamma(\varphi_0, V,\epsilon)$ is no more than $\epsilon$ away from the true solution in the $L^2(\mathbb{R}^d)$ sense? Moreover, for which classes will $\Gamma(\varphi_0, V,\epsilon)$ have uniformly bounded runtime for all inputs in the class?}
\end{displayquote}
\vspace{1mm}
This question has been open since mathematicians initiated research in computational PDEs on unbounded domains in the 1970s. Although there is a vast literature and a myriad of different techniques, the foundations of computational quantum mechanical PDEs on unbounded domains are not known. The situation is similar to the problem of computing spectra of general operators and Schr\"odinger operators on unbounded domains, where W. Arveson pointed out in the 1990s \cite{Arveson_role_of94} that: {\it "Unfortunately, there is a dearth of literature on this basic problem, and so far as we have been able to tell, there are no proven techniques"}. Indeed, there is a vast literature providing invaluable insight into specific spectral computational cases, yet the general computational spectral problem remained unsolved for a substantial time \cite{hansen2011solvability, ben2015can, colb1}.

\subsection{Short summary of the main results}\label{sec:novelty}
We establish impossibility results demonstrating how no algorithm exists for computing solutions to large classes of linear Schr\"odinger PDE problems, despite that the spectra of the operators can easily be computed. Moreover, we provide sufficient conditions for the existence of algorithms for many classes of Schr\"odinger PDE, both linear and non-linear. 
\begin{itemize}
\item[(i)]
Our impossibility results demonstrate that the answer to the above question becomes a vast classification theory on its own. In particular, we show how it is impossible to compute solutions to the linear Schr\"odinger equation on unbounded domains even locally, given smoothness and computability assumptions on the initial state. Similarly, we obtain impossibility results on nonlinear Schr\"odinger equations and show that it is in general impossible to numerically compute whether a solution blows up in finite time or exists forever. However, in the linear case, there are classes of problems on unbounded domains where the initial state blows up in any weighted Sobolev norm, yet there do exists algorithms that can compute solutions to these problems. Such results demonstrate how intricate and potentially surprising such a classification theory is. 

\item[(ii)]
The positive results initiate a program for classifying the different classes of Schr\"odinger PDEs that can be computed and imply affirmative answers to the basic question above for classes of problems on unbounded domains that before were unknown how to handle. These results are based on new techniques that apply to both linear and nonlinear Schr\"odinger equations with time-dependent potentials. The main ideas can be summarised as follows: By assuming sufficient smoothness and decay on the initial state, bounds on its generalised Sobolev norm, weak growth conditions at infinity of the potential as well as mild regularity, we show how one can recursively compute from $\epsilon$ an $R > 0$, such that the solution on the cube $\mathcal{C}_R$ centred at zero with length $R$ and Dirichlet boundary conditions, is $\epsilon$ away from the true solution (on the unbounded domain) in the $L^2$ sense. We then use discretisation techniques for the computational problem on $\mathcal{C}_R$, where all parameters needed to get $\epsilon$-accuracy can be recursively determined from $\epsilon$.
In order to tackle initial states with little smoothness and decay and rough potentials, we then apply several approximation techniques in terms of global basis approximation and mollifying of the potentials. This allows us to deal with non-smooth potentials and almost non-smooth $L^2$ initials states $\varphi_0$ that only satisfy $\|S^{\nu}\varphi_0\|_{L^2} \leq C$, for some $C > 0$, where $S=-\Delta+\vert x \vert^2$ is the harmonic oscillator \eqref{eq:the_S} and $\nu > 0$ can be arbitrarily small. 

\item[(iii)]
The results are a direct continuation of Smale's program \cite{Smale2, Smale_Acta_Numerica, BSS_Machine, Smale_book} on the foundations of computational mathematics initiated in the 1980s. Smale asked several fundamental questions on the existence of algorithms for basic problems in numerical analysis and computational mathematics, for example if there exist alternatives to Newton's method that would always converge for polynomial root finding. This led to deep and surprising results in form of impossibility results developed by McMullen and upper bounds by Doyle \& McMullen \cite{McMullen1, mcmullen1988braiding, Doyle_McMullen}. These problems are special cases in the Solvability Complexity Index (SCI) hierarchy \cite{ben2015can, hansen2011solvability,colb1} that also is the basis for the solutions to the computational spectral problem that characterises the boundaries of computational quantum mechanics in terms of spectral computation. Our results follow in this tradition, and many of the techniques used stem from the SCI hierarchy framework. 
\end{itemize}

\begin{rem}[{\bf Sufficient and necessary conditions for the existence of algorithms}]
In Theorem \ref{Th:main_thrm_lower_bound} we demonstrate how it is in general impossible to compute solutions to the linear and nonlinear Schr\"odinger equations. In order to characterise the boundaries of what computers can achieve in quantum mechanics, this immediately sparks the question: which conditions are sufficient and necessary for the existence of algorithms? The question on sufficient conditions makes perfect sense, and large part of this paper is devoted to this question. However, the question on necessary conditions may not accurately address the issue of determining the boundaries of what computers can achieve. Indeed, one may be tempted to think that a necessary condition would be bounds on smoothness and/or decay of the initial state. As we show in Theorem \ref{Th:main_thrm_lower_bound}, and discussed above in (i), this is not the case. Blow up in any weighted Sobolev norm does not hinder the existence of algorithms for certain classes of problems. Thus, necessary conditions will be very general (such as maybe measurability or continuity of the initial state). A different point of view to establish the boundaries is to determine the classes for which there will exist an algorithm, and the classes for which there will not exist algorithms. As our results suggest, this is a highly intricate task. 
\end{rem}

\subsection{Classical approaches do not in general answer the above foundations question}
Just as the long tradition in computational spectral theory provided important knowledge in specific cases, the classical techniques yield invaluable insight into many core questions on computational issues and numerical analysis problems regarding Schr\"odinger PDEs on unbounded domains, however, do not in general answer the above foundational question. 
Classical approaches to solving Schr\"odinger equations on unbounded domains typically fall into the following three categories:
\begin{itemize}
\item[(i)] {\it Truncation of the unbounded domain to a bounded domain}. This includes the classical techniques of Engquist and Majda on ABC, the many follow up variations \cite{EM,AES,LZZ18,T98,YZ,ABK11,Sz,AABES}, as well as the influential work of S. Jin, P. Markowich and C. Sparber \cite{JMS11} (see specifically Remark 4.2). C. Lasser and C. Lubich follow a similar approach in their extensive survey \cite{LL20} (see specifically Section 7.3). 
\item[(ii)] {\it Galerkin discretisation}. This technique projects the solution to a finite set of basis functions and is explored in Lubich's foundational monograph  \cite{L08a} (Section III.1), where Hermite functions are used, see also \cite{KLY19}, and in Lasser and Lubich's work  \cite{LL20} mentioned above. 
\item[(iii)] {\it Series expansion}. Series expansions methods have been pioneered by A. Iserles, K. Kropielnicka, and P. Singh \cite{IKS18,IKS18a,IKS18b} using for example Magnus series. This approach is closely related to the Dyson series, and can also handle time-dependent potentials using variations of the Strang splitting scheme for the nonlinear Schr\"odinger equation, see also Lubich \cite{L08}. 
\end{itemize}

Considering (i), it is only in very specific cases where one knows how to set the bounded domain with the corresponding boundary conditions so that the solution to the problem on the bounded domain is $\epsilon$ away from the solution to the problem on the unbounded domain. This is even reflected in the original results of Engquist and Majda \cite{Engquist_CPAM} that only guarantee that given $\epsilon > 0$ there exists a set of boundary conditions with the desired properties. In particular, there is a function $\epsilon \mapsto f(\epsilon)$ that takes $\epsilon$ to the parameters needed to describe the appropriate bounded domain and the boundary conditions, however, this function may not be recursive. That is, there may not exist an algorithm that can compute $f$ (in fact our Theorem \ref{Th:main_thrm_lower_bound} immediately implies that in general no algorithm exists for computing $f$). This problem is universal for all techniques of the form "truncate the infinite domain", and as a result such boundary conditions are typically heuristically set. Hence, the vast literature on artificial boundary conditions does not answer the above basic question. 

The approach in (ii) is fundamentally different, providing actual global error bounds on the computed solution compared to the true solution on the unbounded domain. However, the current techniques require assumptions also on the behaviour of the true solution and strong smoothness assumptions. Hence, the results will only give answers to the above question for certain specific classes of problems where certain properties of the true solution are known. In the cases \cite{LL20} where it is used for computing in the semi-classical regime (see \eqref{eq:semicl0} in \S \ref{sec:semicllimit}), the error analysis is as the semiclassical parameter $\mu \rightarrow 0$. Thus, for fixed $\mu$ the error does not tend to zero. In the semiclassical regime such a priori properties of the solution can occur in a priori bounds on auxiliary functions such as the Herman-Kluk prefactor, cf. \cite{LS17}.

The series expansion approach in (iii) relies on a spacial discretisation, converting the problem to a system of ODEs, for which the solution have a a small error compared to the original solution. Although, for any $\epsilon$, there may exist such a spacial discretisation yielding an error of size $\epsilon$, the function $f(\epsilon)$ taking $\epsilon$ to the parameters describing the discretisation may not be recursive. Hence, in general, no algorithm will exists computing $f$ (this follows by our Theorem \ref{Th:main_thrm_lower_bound} similar to the issue in (i) above). Moreover, the series expansion approach requires sufficiently high regularity on the potential, depending on the order of the scheme, to converge \cite{IKS18,IKS18a,IKS18b}.  The situation for the nonlinear Schr\"odinger equation is similar and a priori regularity of the solution is needed to obtain convergent splitting methods \cite{L08}. 


\smallsection{Acknowledgements}
We would like to thank A. Iserles, C. Lubich and O. Nevanlinna for interesting discussions and helpful comments on our manuscript. S.B. acknowledges support by the EPSRC grant EP/L016516/1 for the University of Cambridge CDT, the CCA.

\section{The Schr\"odinger equation}

\subsection{The linear Schr\"odinger equation} 

We consider the situation of a single particle described by a self-adjoint Schr\"odinger operator $H_0=-\Delta+V: D(H_0) \subset L^2(\mathbb R^d) \rightarrow L^2(\mathbb R^d)$ with static \emph{pinning potential} $V$. Apart from the static pinning potential, we also allow the presence of an additional \emph{control potential} $V_{\operatorname{con}}$ with time-dependent control function $u \in W^{1,1}_{\operatorname{pcw}}(0,T)$ (piecewise $W^{1,1}$). Thus, writing $V_{\operatorname{TD}}(t):= u(t) V_{\operatorname{con}} $ for the time-dependent potential, we cover in this article time-dependent linear Schr\"odinger equations of the form
\begin{equation}
\begin{split}
\label{eq:bilSchr}
i \partial_t \psi(x,t) &= (H_0 + V_{\operatorname{TD}}(x,t))\psi(x,t), \quad (x,t) \in \mathbb R^d \times (0,T) \\
\psi(\bullet,0)&=\varphi_0.
\end{split}
\end{equation}
The Schr\"odinger equation \eqref{eq:bilSchr} with linear control potential appears naturally in the study of static physical systems, described by Schr\"odinger operators $H_0$, under the influence of a time-dependent electric field. This includes the study of the Stark effect that is the response of an atom or molecule to an external homogeneous constant electric field. 
In the so-called \emph{dipole approximation}, the time-dependent control potential becomes $V_{\operatorname{TD}}= \langle p,E(t) \rangle$ where $p$ is the dipole moment of the object and $E$ the external (time-dependent) electric field acting on it.

Mathematical properties of equations of the form \eqref{eq:bilSchr} have been studied in various contexts and we may only mention those that are most relevant for our analysis:
In this article, we build upon a method, introduced in \cite{BKP}, to prove existence of solutions to \eqref{eq:bilSchr} in certain generalized Sobolev spaces. These spaces are essential in our study of global numerical algorithms that provide solutions to \eqref{eq:bilSchr} on unbounded domains. In addition to standard Sobolev spaces, they ensure a fixed spatial decay rate as well. 
Let us conclude by mentioning that it was shown in \cite{B} that the analysis of \cite{BKP} can be extended to the (nonlinear) Hartree equation. For recent results and further reference on numerical methods for linear time-dependent Schr\"odinger equations, we refer to \cite{IKS18,IKS18a,IKS18b}.

\subsubsection{Algorithms sample the potential and the initial state}

In order for an algorithm to access information about the differential equation it must sample the initial state as well as the potential point-wise. Hence, we are in need of the following definition of a function with \emph{controlled local bounded variation} (CLBBV).

\begin{defi}[Initial state with controlled local boundedness and bounded variation and (CLBBV)]
\label{def:CLBV}
Given an initial state $\varphi_0 \in  \operatorname{BV}_{\mathrm{loc}}(\mathbb R^d)$ we say that $\varphi_0$ has \emph{controlled local boundedness and bounded variation} (CLBBV) by $\omega: \mathbb{N} \rightarrow \mathbb{N}$ if for every $R \in \mathbb{N}$ then $K = \omega(R)$ is such that 
\[
\|\varphi_0\big\vert_{\mathcal{C}_R(0)}\|_{L^{\infty}}, \mathrm{TV}(\varphi_0\big\vert_{\mathcal{C}_R(0)}) \leq K,
\]
where $\mathcal{C}_R(0)$ is the closed cube of length $R$ centered at zero.\footnote{We emphasize that bounded total variation already implies a possibly weak $L^{\infty}$ estimate by $\|\varphi_0\big\vert_{\mathcal{C}_R(0)}\|_{L^{\infty}} \le \vert \varphi_0(0) \vert + \mathrm{TV}(\varphi_0\big\vert_{\mathcal{C}_R(0)})$}
\end{defi}

It is a rather obvious assumption that the functions to be sampled must be of local bounded variations. Indeed, from a numerical perspective, to guarantee successful pointwise sampling of a function one will need bounds on the local total variations. Functions that have unbounded local variations will either have discontinuities that will be hard to control or arbitrary wild oscillations. Both of these issues will cause numerical instabilities in the sampling procedure.    

For potentials without singularities, we introduce the following simple definition to capture basic regularity.

\begin{defi}[Potential with controlled $W^{\varepsilon,p}$ norm]
\label{def:PCS}
For $\varepsilon>0$ and $p \in [1,\infty]$, we say that a potential $V \in W^{\varepsilon,p}_{\operatorname{loc}}(\RR^d)$ has controlled local $W^{\varepsilon,p}$-norm by $\Phi: \mathbb{Q}_{+} \rightarrow \mathbb{Q}_{+}$ if 
$
\Vert V \Vert_{W^{\varepsilon,p}(B(0,r))} \le \Phi(r).
$
\end{defi}

Any potential with singularities will be denoted by $W_{\operatorname{sing}}$ where we assume that the singularities $\{x_j\}_{j=1}^{\infty} \subset \mathbb R^d$ have no accumulation point. In order to sample a potential with singularities we need the concept of controlled singularity-blowup.    

\begin{defi}[$L^p$ potential with controlled blow-up]
Given 
\[
W_{\operatorname{sing}} \in L^p(\mathbb R^d) \cap W^{q,\infty}_{\mathrm{loc}}( \mathbb{R}^d \setminus \cup_j\{x_j\}),
\]
 with singularities $\{x_j\}_{j=1}^{\infty} \subset \mathbb R^d$, we say that $W_{\operatorname{sing}}$ has controlled singularity-blowup by $f: (0,\epsilon_0) \times \mathbb{R}_+ \rightarrow \mathbb{Q}_+ \times \mathbb{Q}_+$ if for every $(\epsilon, R) \in (0,\epsilon_0) \times \mathbb{R}_+$ then $f(\epsilon, R) = (\delta, K)$, such that for characteristic functions $\indic_M$ of set $M$  $\|W_{\operatorname{sing}}\indic_{\mathcal{A}} - W_{\operatorname{sing}}\indic_{\mathcal{C}_R(0)}\|_{L^p} \leq \epsilon$ where $\mathcal{A} = \mathcal{A}(\delta,R) = (\bigcup_{x_j \in \mathcal{C}_R(0)} \mathcal{C}_\delta(x_j))^c \cap \mathcal{C}_R(0)$, where $\mathcal{C}_{\delta}(x_j)$ is the closed cube centred at $x_j$ with length $\delta$. Moreover, $\|W_{\operatorname{sing}}\big\vert_{\mathcal{A}}\|_{W^{q,\infty}(\mathcal A)} \leq K$.  
\end{defi}

Since away from singularities, the singular potentials is locally bounded, it is natural to make the following definition

\begin{defi}[Potential with controlled local smoothness]
Given a potential $V \in W^{p,\infty}_{\mathrm{loc}}(\mathbb R^d)$ we say that $V$ has controlled local smoothness by $g: \mathbb{R}_+ \rightarrow \mathbb{N}$ if for every $R \in \mathbb{R}_+$ then $K = g(R)$ is such that $\|V\big\vert_{\mathcal{C}_R(0)}\|_{W^{p,\infty}(\mathcal{C}_R(0))} \leq K$.
\end{defi}

\begin{rem}[Input to the algorithms] 
We assume that 
\[
\left\{(\varphi_0(x_k), V_{\operatorname{TD}}(x_k,t_j)) \, \vert \, \{x_k\}_{k\in\mathbb{N}}, \{t_j\}_{j \in \mathbb{N}} \text{ are dense in $\mathbb{R}^d$ and $[0,T]$}, x_k, t_j \text{ have coordinates} \in \mathbb{Q}\right\},
\]
are accessible to the algorithm. In the Blum-Shub-Smale (BSS) model \cite{BSS_Machine}, this means that the point samples of the initial state $\varphi_0$ and the potential $V_{\operatorname{TD}}$ are accessed through an oracle node. In the Turing model \cite{Turing_Machine} (assuming rational values) this means that the point samples are accessed through an oracle tape. 
\end{rem}

\subsection{The non-linear Schr\"odinger eq. (NLS): focusing/defocusing} 

We then show, as for the linear Schr\"odinger equations, that by proving well-posedness of certain NLS on generalized Sobolev spaces, the solutions to these equations are globally computable by restricting the dynamics to a sufficiently large bounded spatial domain. 

The nonlinear Schr\"odinger equations we consider are
\begin{equation}
\begin{split}
\label{eq:GP}
i \partial_t \psi(x,t) &= H_0 \psi(x,t) + V_{\operatorname{TD}}(x,t)\psi(x,t)+  \nu F_{\sigma}(\psi(x,t)), \quad (x,t) \in \mathbb R \times (0,T) \\
\psi(\bullet,0)&=\varphi_0
\end{split}
\end{equation} 
with scattering length $\nu =1$ and nonlinearity $F_{\sigma}(\psi(x,t)) = \vert \psi(x,t) \vert^{\sigma-1} \psi(x,t)$ where we consider $\sigma=3$ (cubic NLS) and $\sigma=5$ (quintic NLS). The choice $\nu=1$ yields a \emph{defocussing} nonlinearity and $\nu=-1$ a \emph{focussing} one. 
For our positive global computability result, we can only consider a defocussing nonlinearity, i.e. $\nu=1$. This choice is necessary as the focussing nonlinearity in \eqref{eq:GP} can lead to the existence of blow up solutions. The numerical analysis of blow-ups is addressed separately in this paper.  The spaces on which we define our numerical methods to study \eqref{eq:GP} are smaller than the spaces on which \eqref{eq:GP} is naturally well-posed \cite{D16}.

\subsubsection{Blow-up criteria and focussing NLS} 
\label{subsec:INLS}
While the solution to the quintic NLS in \eqref{eq:GP} exists for all times if the nonlinearity is defocussing, this is no longer the case if \eqref{eq:GP} has a focussing quintic nonlinearity. In greater generality, we study whether it is possible to numerically decide whether a solution to a NLS will blow up in finite time or not?- We show that this is impossible in great generality.

The -at least from a physics perspective- most prominent example of a NLS with non-trivial blow-up dichotomy is the focussing ($\nu=-1$) cubic NLS 
\begin{equation}
\begin{split}
\label{eq:focusing}
i \partial_t \psi(x,t) + \Delta \psi(x,t) &= \nu \vert \psi(x,t) \vert^{2} \psi(x,t), \quad (t,x) \in \mathbb R \times \mathbb R^3, \\
\psi(0,x) &= \varphi_0(x).
\end{split}
\end{equation}
Choose any fixed $C > 0$ and $g$ as in Def. \ref{def:CLBV}. Then, we define, for $\rho,\nu \ge 1$,
the set of initial data as
\[\Omega_{\mathrm{BU}(1)} = \{\varphi_0 \in H^{\rho}_{\nu}(\RR) \, \vert \,   \left\lVert  \varphi_0 \right\rVert_{H^{\rho}_{\nu} (\mathbb R)} \leq C \text{ and }\varphi_0 \text{ has CLBV by }g \}.\]
We consider the computational problem  
 \begin{equation}\label{eq:dec_prob_1}
 \Xi_{\mathrm{BU}(1)}: \Omega_{\mathrm{BU}(1)} \ni \varphi_0 \mapsto
 \begin{cases}
 \mathrm{Yes}  &  \text{if \eqref{eq:focusing} blows up in finite time},\\
  \mathrm{No}  &  \text{if \eqref{eq:focusing} does not blows up in finite time}
 \end{cases}
 \in \mathcal{M},
\end{equation}
where  
$
\mathcal{M} = \{\mathrm{Yes}, \mathrm{No}\} = \{1,0\}.
$
Next, we consider the focussing ($\nu=-1$) mass-critical NLS with  $\sigma=1+4/d,$ in particular,
\begin{equation}
\begin{split}
\label{eq:focusing2}
i \partial_t \psi(x,t) + \Delta \psi(x,t) &= \nu \vert \psi(x,t) \vert^{\sigma-1} \psi(x,t), \quad (t,x) \in \mathbb R \times \mathbb R^d, \\
\psi(0,x) &= \varphi_0(x).
\end{split}
\end{equation}
Choose any fixed $C > 0$ and $g$ as in Def. \ref{def:CLBV}. Then, we define, for $\rho,\nu \ge 1$,
the set of initial data as
\[
\Omega_{\mathrm{BU}(2)} = \{\varphi_0 \in H^{\rho}_{\nu} (\mathbb R^d) \, \vert \,   \left\lVert  \varphi_0 \right\rVert_{H^{\rho}_{\nu}(\mathbb R^d)} \leq C \text{ and }\varphi_0 \text{ has CLBV by }g\}.
\]
We consider the computational problem  
\begin{equation}\label{eq:dec_prob_2}
 \Xi_{\mathrm{BU}(2)}: \Omega_{\mathrm{BU}(2)} \ni \varphi_0 \mapsto
 \begin{cases}
 \mathrm{Yes}  &  \text{if \eqref{eq:focusing2} blows up in finite time},\\
  \mathrm{No}  &  \text{if \eqref{eq:focusing2} does not blows up in finite time}
 \end{cases}
 \in \mathcal{M}.
\end{equation}
Our main result on the computability of blow-ups is then Theorem \ref{Th:main_thrm_blow}.

\subsection{The semiclassical limit} 
\label{sec:semicllimit}
The semiclassical formulation of the linear Schr\"odinger equation, with semiclassical parameter $\mu>0$, is 
\begin{equation}
\label{eq:semicl0}
i \mu \partial_t \psi(x,t) = -\mu^2 \Delta \psi(x,t) + V\psi(x,t) + V_{\operatorname{TD}}(t) \psi(x,t).
\end{equation}
By rescaling potentials, it suffices to analyze
\begin{equation}
\label{eq:semicl}
i  \partial_t \psi(x,t) = -\mu^2 \Delta \psi(x,t) + V\psi(x,t) + V_{\operatorname{TD}}(t) \psi(x,t).
\end{equation}
Our algorithms can handle these cases as well, although we do not specifically analyse how the runtime of the algorithm is effected by small $\mu$. 
Our Theorem \ref{theo2}, however, shows that for potentials without singular part, the size of the domain on which the evolution of \eqref{eq:semicl} is supported up to a specified error is controlled only by the semiclassical norm $\Vert \varphi_0 \Vert_{H^{2,\operatorname{sem}}_2}$ (defined in \eqref{eq:semi_classical}) of the initial state $\varphi_0$. Indeed, since the semiclassical norm is monotonically increasing in $\mu>0$, this domain is uniformly bounded in the semiclassical limit. 
In particular, the $R>0$ describing the size of the cube $\mathcal{C}_R$ described \S \ref{sec:novelty} in the outline of the algorithm in Step Ib in \S \ref{sec:OoA} is uniformly bounded in $\mu>0.$
If the singular part of the potential is non-zero, then the size of the domain depends in addition on the constants for the relative boundedness of the singular part of the potential and can potentially grow as $\mu \downarrow 0.$


\section{Main results}

The main results are split in two: the linear Schr\"odinger and the non-linear Schr\"odinger equations. For the NLS we separately analyze the computability of the solution to the defocussing equation and the computability of blow-up solutions for the focussing one. To explain our results, we introduce (generalized) Sobolev spaces for parameters $\rho,\eta  \ge 0$ 
\begin{equation}
\begin{split}
\label{eq:gensob}
&H^\rho_\eta(\mathbb R^d) :=\left\{ f \in H^\rho(\mathbb R^d); \langle \bullet \rangle^{\eta} f \in L^2(\mathbb R^d) \right\}\text{ with norms }\\
&\left\lVert f \right\rVert^2_{H^\rho_{\eta}(\mathbb R^d)}:=\left\lVert \langle\bullet \rangle^{\rho} \widehat{f} \right\rVert_{L^2(\mathbb R^d)}^2+  \left\lVert  \langle\bullet \rangle^{\eta} f \right\rVert_{L^2(\mathbb R^d)}^2
\end{split}
\end{equation} 
where $\langle x \rangle :=\left(1+\vert x \vert^2\right)^{1/2}.$ 
The semiclassical Sobolev norms are defined for a semiclassical parameter $\mu$ as 
\begin{equation}\label{eq:semi_classical}
\left\lVert f \right\rVert^2_{H^{\rho,\operatorname{sem}}_{\eta}(\mathbb R^d)}:= \left\lVert \langle \mu \bullet \rangle^{\rho} \widehat{f} \right\rVert_{L^2(\mathbb R^d)}^2+  \left\lVert  \langle\bullet \rangle^{\eta} f \right\rVert_{L^2(\mathbb R^d)}^2.
\end{equation}
Moreover, for the self-adjoint positive-definite operator 
\begin{equation}\label{eq:the_S}
S=-\Delta+\vert x \vert^2
\end{equation}
 on $L^2(\RR^d)$, we define the canonical norm on the domain $D(S^{\nu})$\footnote{Let $(\lambda_n, \varphi_n)$ be the tuple of eigenvalues and eigenfunctions (counting multiplicity), then $\varphi \in L^2(\RR^d)$ is in $D(S^{\nu})$ if and only if $\sum_{n}  \lambda_n^{2\nu} \vert \langle \varphi, \varphi_n \rangle \vert^2<\infty$.}  for any 
 $\nu \ge 0$ by
\[
\Vert f \Vert_{S^{\nu}}:= \Vert S^{\nu} f \Vert_{L^2}.
\]

The space of piecewise $W^{1,1}$ functions on some interval $(0,T)$ is denoted by $W^{1,1}_{\operatorname{pcw}}(0,T)$ and consists of all functions $u$ that there exists a finite partition $I_i:=(t_i,t_{i+1})$ of $(0,T)$ such that 
$
 \Vert u \Vert_{W^{1,1}_{\operatorname{pcw}}(0,T)} = \sum_{i=1}^{n} \Vert u \Vert_{W^{1,1}(I_i)}.
 $

\vspace{1mm}
\smallsection{Universality of results and model of computation}
All our lower bounds and impossibility results are universal, independent of the computational model i.e. the Blum-Shub-Smale (BSS) model, the Turing model, the Von Neumann (Princeton) model etc.. However, all our positive results are done in in the BSS model. As this is a paper written for numerical analysts and analysts working in mathematical physics we have deliberately chosen the BSS model as this is the standard model used by numerical analysts, i.e. one assumes basic arithmetic operations and comparisons with real numbers. It is evident from our methods that with minor changes, our approach will hold in the Turing model as well, however, as this is not a paper in logic nor computer science our focus is on the BSS model in order to serve the targeted numerical analysis and mathematical analysis community. 

\begin{rem}[Runtime of an algorithm]
We define the runtime of an algorithm to be the total number of arithmetic operations and comparisons done to execute the algorithm. This is equivalent to the definition of the runtime of a BSS machine. 
\end{rem}

 \subsection{Computing solutions to the linear Schr\"odinger equation} We establish both upper and lower bounds on the existence of algorithms for computing the solutions to the linear Schr\"odinger equation. The first result establishes that it is in general impossible to compute even a local solution regardless of the potential. 

\begin{theo}[Impossibility results and paradoxes]\label{Th:main_thrm_lower_bound}
Let $C > 0$ and $\rho \geq 0$. Choose any $T > 0$ and any open set $O \subset \mathbb{R}$ and define domains
\begin{equation}\label{eq:free_domain}
\begin{split}
\Omega^1_{\mathrm{free}} &= \{\varphi \in D(S^{\rho}); \, \vert \,   \left\lVert   \varphi \right\rVert_{L^2(\mathbb R)} \leq C, \varphi \text{ is computable} \},\\
\Omega^2_{\mathrm{free}} &= \{\varphi \in D(S^{\rho}); \, \vert \,   \left\lVert   \varphi \right\rVert_{H^\rho_0(\mathbb R)} \leq C, \varphi \text{ is computable} \},
\end{split}
\end{equation}
where $S^{\rho}$ is defined in \eqref{eq:the_S},
as well as the mapping
\begin{equation}\label{eq:free_Xi}
\begin{split}
\Xi^1_{\mathrm{free}} &: \Omega^1_{\mathrm{free}} \ni \varphi_0 \mapsto u(\bullet,T)\big\vert_{O} \in L^2(O),\\
\Xi^2_{\mathrm{free}} &: \Omega^2_{\mathrm{free}} \ni \varphi_0 \mapsto u(\bullet,T) \in L^2(\mathbb{R}),
\end{split}
\end{equation}
where $u$ is the solution to the free Schr\"odinger equation $(i \partial_t + \Delta)u = 0, \ u(\bullet,0) = \varphi_0$.
Then we have the following:
\begin{itemize}
\item[(I)]
{\bf (Not even a local solution can be computed on $\Omega^1_{\mathrm{free}}$).} There does not exist any algorithm $\Gamma$ such that $\|\Gamma(\epsilon, \varphi_0) - \Xi^1_{\mathrm{free}}(\varphi_0)\|_{L^2(O)} \leq \epsilon$ for all sufficiently small $\epsilon >0$ and $\varphi_0 \in \Omega^1_{\mathrm{free}}$. 
\item[(II)] {\bf (No global solution can be computed on $\Omega^2_{\mathrm{free}}$).} 
There does not exist any algorithm $\Gamma$ such that $\|\Gamma(\epsilon, \varphi_0) - \Xi^2_{\mathrm{free}}(\varphi_0)\|_{L^2(\mathbb{R})} \leq \epsilon$ for all sufficiently small $\epsilon >0$ and $\varphi_0 \in \Omega^2_{\mathrm{free}}$.
\item[(III)] {\bf (Certain cases can be computed despite no bounds in $\left\lVert \cdot \right\rVert_{H^\rho_{\eta}}$).} There exists a domain $\Omega_{\mathrm{free}} \subset L^2(\mathbb{R})$ where $\sup\{\|\varphi \|_{H^\rho_{\eta}} \, \vert \, \varphi \in \Omega_{\mathrm{free}}\} = \infty$ for all $(\rho, \eta) \in \mathbb{R}^2_+ \setminus \{(0,0)\}$, and an algorithm $\Gamma$ such that $\|\Gamma(\epsilon, \varphi_0) - \Xi_{\mathrm{free}}(\varphi_0)\|_{L^2(\mathbb{R})} \leq \epsilon$, where $\Xi_{\mathrm{free}}: \Omega_{\mathrm{free}} \ni \varphi_0 \mapsto u(\bullet,T) \in L^2(\mathbb{R})$.
\end{itemize} 
\end{theo}

\begin{rem}[Universality of the computational model]
Statement (I) is independent of the computational model. Statement (II) is true for any reasonable model of computation for example the Turing model or the BSS model. 
\end{rem}

\begin{rem}[Computable function]
The term computable function refers to the standard definition in the literature. The reader unfamiliar with the term may think of a computable function as a function for which there exists an algorithm that can output an approximation to the function to arbitrary precision. For example, the functions $\exp$, $\cos$, $\sin$ are obviously computable, and there are a myriad of different ways to compute these functions.   
\end{rem}

\begin{rem}[Consequences of Theorem \ref{Th:main_thrm_lower_bound}]
Theorem \ref{Th:main_thrm_lower_bound} demonstrates that the Schr\"odinger equation evolution operator, regardless of potential, takes computable initial conditions to non-computable functions even in the $L^2$ sense. Moreover, one cannot even compute a local approximation on any open set.  This means that there is only a limited collection of Schr\"odinger equations for which there will exist algorithms that can compute the corresponding solutions. The question is: which such equations will have algorithms that allow for accurate computations?
\end{rem}

 With the lower bounds in Theorem \ref{Th:main_thrm_lower_bound} established, the discussion of the assumptions needed to ensure existence of algorithms for the problem. 
 
\begin{Assumption}[Initial states and potentials]\label{ass:1} 
 Let $C > 0$ and $\varepsilon > 0$. Choose any $T > 0$ and define $\Omega_{\mathrm{Lin}}$ and $\Omega_{\mathrm{NLS}}$ to be the collection of inputs 
 $(\varphi_0,V+V_{\operatorname{TD}}(t),C)$ such that we have the following. In the case of $\Omega_{\mathrm{Lin}}$ we let $k = 2$ and for $\Omega_{\mathrm{NLS}}$ we let $k = 3$.
 \begin{itemize}
 \item[(i)] $\varphi_0 \in H^{k+\varepsilon}_2$ has controlled local boundedness and bounded variation by $\omega: \mathbb{R}_+ \rightarrow \mathbb{N}$
 (recall Definition \ref{def:CLBV}), and $\Vert \varphi_0 \Vert_{H^{k+\varepsilon}_2} \le C$.
 \item[(ii)] $V_{\operatorname{TD}}(t) = u(t)V_{\operatorname{con}} + V$ and $V=W_{\operatorname{sing}}+W_{\operatorname{reg}}$ such that for $p < \infty$ and $(x_j)_j$ a sequence that is nowhere dense
 \begin{equation*}
\begin{split}
&W_{\operatorname{sing}} \in L^p \cap W^{k,\infty}_{\mathrm{loc}}(\mathbb{R}^d \setminus \cup_j\{x_j\}), \text{ and } \\
&\Vert W_{\operatorname{sing}} \Vert_{L^p} , \Vert \langle \bullet \rangle^{-2} W_{\operatorname{reg}} \Vert_{L^{\infty}}, \Vert \langle \bullet \rangle^{-2} V_{\operatorname{con}} \Vert_{L^{\infty}}\le C, \text{ as well as } \\
&u \in W^{1,1}_{\operatorname{pcw}}(0,T) \text{ with } \Vert u \Vert_{W^{1,1}_{\operatorname{pcw}}((0,T))} \le C.
\end{split}
\end{equation*}
\item[(iii)] Moreover, $W_{\operatorname{sing}}$ has controlled blowup by $f$ and both $W_{\operatorname{reg}}, V_{\operatorname{con}} \in  W^{k,\infty}_{\mathrm{loc}}$ have controlled local smoothness by $g$.  
\end{itemize}
\end{Assumption}

 Note that the only difference between the assumptions needed in the linear versus the non-linear case is the extra degree of smoothness. In particular, $H^{3+\varepsilon}_2$ and $W^{3,\infty}_{\mathrm{loc}}$ are needed in the non-linear case as opposed to $H^{2+\varepsilon}_2$ and $W^{2,\infty}_{\mathrm{loc}}$ in the linear case. 

We also consider the following alternative set of weaker assumptions on the initial state which we shall use for a more restrictive class of Schr\"odinger equations\footnote{At least for static potentials, this assumptions can be additionally weakened to also cover potentials with quadratic growth.}
\begin{Assumption}
\label{ass:2}
 Let $C > 0$ and $\varepsilon > 0$ and consider the Schr\"odinger operator $S=-\Delta + \vert x \vert^2$. Choose any $T > 0$ and define $\Omega_{\mathrm{Lin2}}$ and $\Omega_{\mathrm{NLS2}}$ to be the collection of inputs 
 $(\varphi_0,V+V_{\operatorname{TD}}(t),C)$ such that we have the following
 \begin{itemize}
 \item[(i)] $\Vert S^{\varepsilon} \varphi_0 \Vert_{L^2}\le C$ and $\varphi_0$ has controlled local boundedness and  bounded variation by $\omega: \mathbb{R}_+ \rightarrow \mathbb{N}$
 (recall Definition \ref{def:CLBV}). 
 \item[(ii)] $V_{\operatorname{TD}}(t) = u(t)V_{\operatorname{con}} + V$ 
where $\|V_{\operatorname{con}}\|_{L^{\infty}}$,$\|V\|_{L^{\infty}} \leq C$ and $V_{\operatorname{con}},V$  have controlled local $W^{\varepsilon,p}$-norm by $\Phi: \mathbb{Q}_{+} \rightarrow \mathbb{Q}_{+}$, as in Definition \ref{def:PCS} with $p$ as in Remark \ref{rem:bounded}, and controlled local bounded variation by $\omega: \mathbb{R}_+ \rightarrow \mathbb{N}$.
\end{itemize}
\end{Assumption}

We recall that an operator $T$ with domain $D(T)$ is called infinitesimally bounded with respect to an operator $S$ with domain $D(S)$ if $D(S) \subset D(T)$ and for every $\varepsilon>0$ there is $C_{\varepsilon}>0$ such that for all $x \in D(S)$
\[ \Vert Tx \Vert \le \varepsilon \Vert Sx \Vert+ C_{\varepsilon} \Vert x \Vert. \]
In particular, for the singular part of the potential, we require this one to be infinitesimally bounded with respect to the negative Laplacian. Sufficient conditions for this are summarized in the following remark:
\begin{rem}
\label{rem:bounded}
\cite[X.20]{RS2} A potential $V \in L^p(\mathbb R^d)$ is infinitesimally bounded with respect to the negative Laplacian $-\Delta$ if $p \ge 2 \operatorname{ for } d \in \left\{1,2,3 \right\} \operatorname{ and } p >\frac{d}{2} \operatorname{ for }d \ge 4.$ In particular, for any $\varepsilon>0$ there exists a constants $C_{\varepsilon}$ such that for all $\psi \in H^2(\RR^d)$ 
\begin{equation}
\label{eq:zero-bd}
\Vert V \psi \Vert_{L^2} \le \varepsilon \Vert \Delta \psi \Vert_{L^2}+C_{\varepsilon} \Vert \psi \Vert_{L^2}.
\end{equation}
\end{rem}

We can now present the main theorem on how to compute solutions to the linear Schr\"odinger equation. 


\begin{theo}[Upper bounds: global solution - linear Schr\"odinger equation]\label{BigTh:main_thrm2}
 Let $C$ and $T > 0$, and consider the linear Schr\"odinger equation \eqref{eq:bilSchr}. 
 \begin{itemize}
 \item[(i)]
 Define, for $\varepsilon>0$, $\Omega_{\mathrm{Lin}}$ as in Assumption \ref{ass:1} and 
 $
\Xi_{\mathrm{Lin}} : \Omega_{\mathrm{Lin}} \ni (\varphi_0,V) \mapsto \psi(\bullet,T) \in L^2(\mathbb{R}^d),
$
where $\psi$ satisfies \eqref{eq:bilSchr} and $\psi(\bullet,0)=\varphi_0$.
 \item[(ii)]
Define also $\Omega_{\mathrm{Lin2}}$ as in Assumption \ref{ass:2} and 
 $
\Xi_{\mathrm{Lin2}} : \Omega_{\mathrm{Lin2}} \ni (\tilde \varphi_0,V) \mapsto \psi(\bullet,T) \in L^2(\mathbb{R}^d),
$
where $\tilde \psi$ satisfies \eqref{eq:bilSchr} with initial value  $\tilde \psi(\bullet,0)=\tilde \varphi_0$.
\end{itemize}
Then there exist algorithms $\Gamma_1$ and $\Gamma_2$ with the following properties.
\begin{itemize}
\item[(I)] For $\Gamma_1$ we have that 
\[
\|\Gamma_1(\varphi_0,V,\epsilon) - \Xi_{\mathrm{Lin}}(\varphi_0,V)\|_{L^2(\mathbb{R}^d)} \leq \epsilon, \qquad \forall \, \epsilon > 0, \,  (\varphi_0,V) \in \Omega_{\mathrm{Lin}}.
\]
If there are no singularities in the potential $W_{\operatorname{sing}}$ then algorithm $\Gamma_1$ can be made so that it will have for each $\epsilon > 0$ uniformly bounded runtime for all $(\varphi_0,V) \in \Omega_{\mathrm{Lin}}$. 
\item[(II)]
For $\Gamma_2$ we have that 
\[
\|\Gamma_2(\tilde \varphi_0,V,\epsilon) - \Xi_{\mathrm{Lin2}}(\tilde \varphi_0,V)\|_{L^2(\mathbb{R}^d)} \leq \epsilon, \qquad \forall \, \epsilon > 0,  \, (\tilde \varphi_0,V) \in \Omega_{\mathrm{Lin2}}.
\]
 Moreover, for each $\epsilon > 0$, $\Gamma_2$ will have uniformly bounded runtime for all $(\varphi_0,V) \in \Omega_{\mathrm{Lin}}$. 
\end{itemize}
\end{theo}

\begin{rem}[The boundaries of computational quantum mechanics]
Theorem \ref{BigTh:main_thrm2} demonstrates upper bounds on the boundaries of what computers can achieve in computational quantum mechanics. They are, to the best of our knowledge, the most general conditions known. However, Theorem \ref{BigTh:main_thrm2} immediately begs the question about necessity of the assumptions on the potentials and the initial state. 
\end{rem}

 \subsection{Computing solutions to the NLS equation} In the linear Schr\"odinger case one can guarantee the existence of a unique solution for any time $T >0$. This is not the case when considering focusing NLS. From a computational point of view it becomes crucial to determine whether an algorithm can determine if the solution will blow up in finite time or not. As the following theorem reveals, this is not just impossible, but impossible to verify and falsify. 
 
  \begin{theo}[Blow up cannot be decided, in fact not verified nor falsified]\label{Th:main_thrm_blow}
 Consider the decision problems $\{\Xi_{\mathrm{BU}(1)},\Omega_{\mathrm{BU}(1)}\}$ and $\{\Xi_{\mathrm{BU}(2)},\Omega_{\mathrm{BU}(2)}\}$ defined in \eqref{eq:dec_prob_1} and \eqref{eq:dec_prob_2} concerning the blow up of the NLS. Then, there do not exist sequences of algorithms $\{\Gamma^1_k\}$, $\{\Gamma^2_k\}$, with $\Gamma^1_k: \Omega_{\mathrm{BU}(1)} \rightarrow \mathcal{M}$ and $\Gamma^2_k: \Omega_{\mathrm{BU}(2)} \rightarrow \mathcal{M}$ such that 
 \[
\lim_{k\rightarrow \infty} \Gamma^1_k(\varphi_0) = \Xi_{\mathrm{BU}(1)}(\varphi_0), \text{ such that }  \Gamma^1_k(\varphi_0) = \mathrm{No} \Rightarrow \Xi_{\mathrm{BU}(1)}(\varphi_0) = \mathrm{No}, 
 \]
 \[
\lim_{k\rightarrow \infty} \Gamma^2_k(\varphi_0) = \Xi_{\mathrm{BU}(2)}(\varphi_0), \text{ such that }  \Gamma^1_k(\varphi_0) = \mathrm{Yes} \Rightarrow \Xi_{\mathrm{BU}(2)}(\varphi_0) = \mathrm{Yes}. 
 \] 
These statements are universal independent of the computational model. 
 \end{theo} 

 \begin{rem}[Consequences of Theorem \ref{Th:main_thrm_blow}] Theorem \ref{Th:main_thrm_blow} demonstrates that regardless of how smooth and rapidly decaying the initial state $\varphi_0$ is, in particular, $ \left\lVert  \varphi_0 \right\rVert_{H^{\rho}_{\nu}(\mathbb R^d)} \leq C$ for arbitrary $\rho$ and $\nu$, one cannot determine blow up from point samples of $\varphi_0$. Moreover, one cannot even verify or falsify whether one has blow up or not by making an algorithm run forever. Hence, there is very little point of numerical computation of the focusing NLS unless one has extra knowledge of  lack of blow-up. 
 \end{rem}
 
However, for the defocusing NLS it is possible, subject to assumptions on the potential and initial state, to compute approximations to the solution. 

\begin{theo}[Upper bounds: global solution - NLS]\label{BigTh:main_thrm2_NLS}
 Let $C$ and $T > 0$ and consider the NLS \eqref{eq:GP} with either $\sigma=3$ (cubic NLS) or $\sigma=5$ (quintic NLS) and a defocussing nonlinearity, i.e. $\mu=1$.
\begin{itemize}
 \item[(i)]
 Define, for $\varepsilon>0$, $\Omega_{\mathrm{NLS}}$ as in Assumption \ref{ass:1} and 
 $
\Xi_{\mathrm{NLS}} : \Omega_{\mathrm{NLS}} \ni (\varphi_0,V) \mapsto \psi(\bullet,T) \in L^2(\mathbb{R}^d),
$
where $\psi$ satisfies \eqref{eq:GP} and $\psi(\bullet,0)=\varphi_0$. 
 \item[(ii)]
Define also $\Omega_{\mathrm{NLS2}}$ as in Assumption \ref{ass:2} and 
 $
\Xi_{\mathrm{NLS2}} : \Omega_{\mathrm{NLS2}} \ni (\tilde \varphi_0,V) \mapsto \psi(\bullet,T) \in L^2(\mathbb{R}^d),
$
where $\tilde \psi$ satisfies \eqref{eq:bilSchr} with initial value  $\tilde \psi(\bullet,0)=\tilde \varphi_0$.
\end{itemize}
Then there exist two algorithms $\Gamma_1$ and $\Gamma_2$ with the following properties.
\begin{itemize}
\item[(I)] For $\Gamma_1$ we have that 
\[
\|\Gamma_1(\varphi_0,V,\epsilon) - \Xi_{\mathrm{NLS}}(\varphi_0,V)\|_{L^2(\mathbb{R}^d)} \leq \epsilon, \qquad \forall \, \epsilon > 0, \,  (\varphi_0,V) \in \Omega_{\mathrm{NLS}}.
\]
If there are no singularities in the potential $W_{\operatorname{sing}}$ then algorithm $\Gamma_1$ can be made so that it will have for each $\epsilon > 0$ uniformly bounded runtime for all $(\varphi_0,V) \in \Omega_{\mathrm{NLS}}$. 
\item[(II)]
For $\Gamma_2$ we have that 
\[
\|\Gamma_2(\tilde \varphi_0,V,\epsilon) - \Xi_{\mathrm{NLS2}}(\tilde \varphi_0,V)\|_{L^2(\mathbb{R}^d)} \leq \epsilon, \qquad \forall \, \epsilon > 0,  \, (\tilde \varphi_0,V) \in \Omega_{\mathrm{NLS2}}.
\]
 Moreover, for each $\epsilon > 0$, $\Gamma_2$ will have uniformly bounded runtime for all $(\varphi_0,V) \in \Omega_{\mathrm{}}$. 
\end{itemize}
\end{theo}

For discrete (nonlinear) Schr\"odinger equations on $\ell^2(\mathbb Z^d)$ with both focussing and defocussing nonlinearity $\mu \in \{\pm 1\}$
 \begin{equation}
\begin{split}
\label{eq:fullNLS}
i  \partial_t v(k,t) &=- \Delta^d v(k,t) +\nu F_{\sigma}(v(k,t)) ,\ k \in \mathbb Z^d \\
v(0)&=v_0 \in \ell^2(\mathbb{Z}^d),
\end{split}
\end{equation}
with discrete nearest-neighbor Laplacian $ \Delta^d$ the analysis simplifies dramatically and it is possible to consider input data up to (and not including) the critical space $\ell^2$ such that for some $\eta, C>0$ the input space is
\[ 
\Omega_{\mathrm{discNLS}}:=\{ v_0 \in \ell^2(\mathbb Z^d); \Vert v_0 \Vert_{\ell^2_{\eta}} \le C \}
\]
and
 $
\Xi_{\mathrm{discNLS}}: \Omega_{\mathrm{NLS}} \ni (v_0,C,\varepsilon) \mapsto v(\bullet,T) \in \ell^2(\mathbb Z^d),
$
where $v$ satisfies \eqref{eq:fullNLS} and $v(\bullet,0)=v_0$.

 \begin{theo}[Upper bounds: Global solution can be computed for discrete NLS]\label{Th:main_thrm_discNLS}
Let $C > 0$ and $T > 0$. Let $\Omega_{\mathrm{discNLS}}$ and $\Xi_{\mathrm{discNLS}}$ be as above. Then there exists an algorithm $\Gamma$ such that 
 \[
\|\Gamma(\epsilon,v_0) - \Xi_{\mathrm{NLS}}(v_0,V)\|_{\ell^2} \leq \epsilon, \qquad \forall \, \epsilon > 0,
\]
and all $v_0 \in \Omega_{\mathrm{NLS}}$ and any given $C,\varepsilon, \eta>0$.
\end{theo}

\subsection{Connection to previous work and future directions}

Based on the ideas first introduced by Engquist and Majda \cite{EM}, a special case of the Schr\"odinger equation \eqref{eq:bilSchr} on the full domain has been studied, by restricting the analysis to a region of interest of finite measure, using non-reflecting boundary conditions in \cite{AES,LZZ18,T98,YZ}. A similar idea for nonlinear Schr\"odinger equations has been discussed in \cite{ABK11,Sz}, see also the review article \cite{AABES}.
Global numerical discretization schemes, based on Strang's splitting method \cite{L08,DT12,ESS16,JMS}, 
exponential integrators \cite{ORS19,KOS19,OS18}, and their convergence rates are well studied both in the linear and nonlinear setting. Yet, these convergence rates are usually limited to time discretizations \cite{L08,OS18} and an analysis justifying the reduction from a full discretization of $\mathbb R^d$ to a bounded domain has not been addressed to our knowledge. Attempts to reduce this to a finite basis expansion, usually rely on additional a priori information on the solutions, cf. Chapter $3$ in Lubich's monograph \cite{L08a}. We provide a comprehensive answer to this issue in our Section \ref{sec:num}. In particular, our analysis allows us also to obtain-under quite general assumptions-, see Section \ref{sec:semicllimit} uniform estimates in the semiclassical parameter- of the Schr\"odinger equation, justifying some considerations in the recent comprehensive article by Lasser and Lubich \cite{LL20}, and also Jin, Markowich, Sparber \cite{JMS11}.

It would be interesting to perform a similar analysis for Schr\"odinger equations with magnetic fields, as considered in \cite{HLW20, HL20}. Moreover, it would be desirable to consider Schr\"odinger equations with non-local (+non-linear) potentials such as the Hartree equation \cite{L08}. Finally, a detailed analysis of the scaling for multi-particle systems would be desirable.

\section{Preliminaries for the proofs - The SCI hierarchy}
The SCI hierarchy is a framework for establishing the boundaries of computational mathematics that allows for proving universal lower bounds and impossibility results independent of the model of computation. It is also flexible enough to encompass any model of computation for proving positive results and upper bounds. 
We will review some of the basic concepts starting with a computational problem. 
\begin{itemize}
\item[(i)]
$\Omega$ is some set, called the \emph{domain}.
\item[(ii)]
$\Lambda$ is a set of complex valued functions on $\Omega$, called the \emph{evaluation} set.
\item[(iii)]
$\mathcal{M}$ is a metric space.
\item[(iv)]
$\Xi : \Omega\to\mathcal{M}$ is called the \emph{problem} function.
\end{itemize}

\begin{defi}[Computational problem]
Given a domain $\Omega$, an evaluation set $\Lambda$, a metric space $\mathcal{M}$ and a problem function $\Xi:\Omega \to \mathcal{M}$, we call the collection $\{\Xi,\Omega,\mathcal{M},\Lambda\}$ a computational problem.
\end{defi}

A simple example of a computational problem would be the problem of computing the solution to the Schr\"odinger equation 
\eqref{eq:bilSchr} at some time $T > 0$. In particular, $\Omega$ would be a family of initial states and potentials. $\Lambda$ would be the collection of functions providing point samples of elements in $\Omega$, i.e. for $\varphi_0 \in \Omega$ and $f_x \in \Lambda$ we have 
$
f_x(\varphi_0) = \varphi_0(x).
$
In particular, the input to the algorithm will be point samples of the initial state and the potential.  
In this paper 
\[
\Lambda = \{f_{x,t} \, \vert \, f_{x,t}(\varphi_0,V_{\operatorname{TD}}) = (\varphi_0(x),\varphi_0,V_{\operatorname{TD}}(x,t)), (x,t) \text{ have rational coordinates}\}.
\]

We could have $\mathcal{M} = L^2(\mathbb R^d)$ or $\mathcal{M}$ could for example be $H^\rho_\eta(\mathbb R^d)$ with appropriately chosen $\rho, \eta > 0$. Finally, we could have 
\[
\Xi(\varphi_0,V_{\operatorname{TD}}) = \psi(\bullet,T), \quad \psi \text{ satisfies \eqref{eq:bilSchr}}, \quad \psi(\bullet,0)=\varphi_0.
\]
Note that $\Xi$ will depend on the actual PDE we are considering (linear or non-linear). 

In order to compute approximate solutions to computational problems we need to define the concept of an \emph{algorithm}. The mainstay is the \emph{general algorithm} that allows universal lower bounds and impossibility results.
\begin{defi}[General Algorithm]\label{alg}
Given a  computational problem $\{\Xi,\Omega,\mathcal{M},\Lambda\}$, a \emph{general algorithm} is a mapping $\Gamma:\Omega\to\mathcal{M}$ such that for each $A\in\Omega$:
\begin{itemize}
\item[(i)] there exists a finite subset of evaluations $\Lambda_\Gamma(A) \subset\Lambda$,
\item[(ii)]  the action of $\,\Gamma$ on $A$ only depends on $\{A_f\}_{f \in \Lambda_\Gamma(A)}$ where $A_f := f(A),$
\item[(iii)]  for every $B\in\Omega$ such that $B_f=A_f$ for every $f\in\Lambda_\Gamma(A)$, it holds that $\Lambda_\Gamma(B)=\Lambda_\Gamma(A)$.
\end{itemize}
\end{defi}

\begin{rem}[The purpose of a general algorithm]
The purpose of a general algorithm is to have a definition that will encompass any model of computation, and that will allow lower bounds and impossibility results to become universal. Given that there are several non equivalent models of computation, lower bounds will be shown with a general definition of an algorithm. Upper bounds will always be done with more structure on the algorithms for example using a Turing machine or a Blum--Shub--Smale (BSS) machine. 
\end{rem}

The concept of a general algorithm, however, is not enough to describe the world of computational problems. For that we need the concept of \emph{towers of algorithms}. 

\begin{defi}[Tower of Algorithms]\label{tower_funct}
Given a computational problem $\{\Xi,\Omega,\mathcal{M},\Lambda\}$, a \emph{tower of algorithms of height $k$
 for $\{\Xi,\Omega,\mathcal{M},\Lambda\}$} is a family of sequences of functions
 $$\Gamma_{n_k}:\Omega
\rightarrow \mathcal{M},\ \Gamma_{n_k, n_{k-1}}:\Omega
\rightarrow \mathcal{M},\dots,\ \Gamma_{n_k, \hdots, n_1}:\Omega \rightarrow \mathcal{M},
$$
where $n_k,\hdots,n_1 \in \mathbb{N}$ and the functions $\Gamma_{n_k, \hdots, n_1}$ at the ``lowest level'' of the tower are general algorithms in the sense of Definition \ref{alg}. Moreover, for every $A \in \Omega$,
$$
\Xi(A)= \lim_{n_k \rightarrow \infty} \Gamma_{n_k}(A), \quad \Gamma_{n_k, \hdots, n_{j+1}}(A)= \lim_{n_j \rightarrow \infty} \Gamma_{n_k, \hdots, n_j}(A) \quad j=k-1,\dots,1.
$$
\end{defi}

In this paper we will discuss two types of towers: {\it General towers}, when there is no extra structure on the functions at the lowest level in the tower, and {\it Arithmetic towers}, that restricts the algorithm to arithmetic operations and comparisons.
A general tower will refer to the very general definition in Definition \ref{tower_funct} specifying that there are no further restrictions as will be the case for the other towers. All our lower bounds and impossibility results are with respect to general towers, whereas all upper bounds and positive results are with respect to arithmetic towers defined as follows.

\begin{defi}[Arithmetic towers]
Given a computational problem $\{\Xi,\Omega,\mathcal{M},\Lambda\}$, where $\Lambda$ is countable, we define the following: An \emph{Arithmetic tower of algorithms} of height $k$
 for $\{\Xi,\Omega,\mathcal{M},\Lambda\}$ is a tower of algorithms where the lowest functions $\Gamma = \Gamma_{n_k, \hdots, n_1} :\Omega \rightarrow \mathcal{M}$ satisfy the following:
 For each $A\in\Omega$ the mapping $(n_k, \hdots, n_1) \mapsto \Gamma_{n_k, \hdots, n_1}(A) = \Gamma_{n_k, \hdots, n_1}(\{A_f\}_{f \in \Lambda})$ is recursive, and $\Gamma_{n_k, \hdots, n_1}(A)$ is a finite string of complex numbers that can be identified with an element in $\mathcal{M}$. For arithmetic towers we let $\alpha = A$ 
\end{defi} 

\begin{rem}[Recursiveness] By recursive we mean the following. If $f(A) \in \mathbb{Q}$ for all $f \in \Lambda$, $A \in \Omega$, and $\Lambda$ is countable, then $\Gamma_{n_k, \hdots, n_1}(\{A_f\}_{f \in \Lambda})$ can be executed by a Turing machine \cite{Turing_Machine}, that takes $(n_k, \hdots, n_1)$ as input, and that has an oracle tape consisting of $\{A_f\}_{f \in \Lambda}$. If $f(A) \in \mathbb{R}$ (or $\mathbb{C}$) for all $f \in \Lambda$, then $\Gamma_{n_k, \hdots, n_1}(\{A_f\}_{f \in \Lambda})$ can be executed by a Blum-Shub-Smale (BSS) machine \cite{Smale_book} that takes $(n_k, \hdots, n_1)$, as input, and that has an oracle that can access any $A_f$ for $f \in \Lambda$. 
\end{rem}

The model of recursiveness in this paper will be the BSS machine, as this is the model closest to the standard tradition in numerical analysis. We do note, however, that with minor modifications, all our results will hold in the Turing model. 

\begin{defi}[Runtime]
Given an arithmetic tower of algorithms, the runtime of $\Gamma_{n_k, \hdots, n_1}(A)$, $A \in \Omega$ is sum of the number of arithmetic operations and comparisons done by the BSS machine executing the output $\Gamma_{n_k, \hdots, n_1}(A)$ plus $\sup\lbrace m \in \mathbb{N} \, \vert \, f_m \in \Lambda_{\Gamma_{n_k, \hdots, n_1}}(A) \rbrace.$
\end{defi} 

In any realistic model of computation there is a cost associated to accessing the information $\{f_m(A) \, \vert \,f_m \in \Lambda_{\Gamma_{n_k, \hdots, n_1}}(A)\}$. Our results do not address the actual complexity, and hence, for simplicity we only add $K = \sup\lbrace m \in \mathbb{N} \, \vert \, f_m \in \Lambda_{\Gamma_{n_k, \hdots, n_1}}(A) \rbrace$ to the cost of the computation. If a specified model is given, one should of course use $g(K)$ for some specified function $g:\mathbb{N} \rightarrow \mathbb{N}$.

Given the definitions above we can now define the key concept, namely, the Solvability Complexity Index: 

\begin{defi}[Solvability Complexity Index]\label{complex_ind}
A computational problem $\{\Xi,\Omega,\mathcal{M},\Lambda\}$ is said to have \emph{Solvability Complexity Index $\mathrm{SCI}(\Xi,\Omega,\mathcal{M},\Lambda)_{\alpha} = k$}, with respect to a tower of algorithms of type $\alpha$, if $k$ is the smallest integer for which there exists a tower of algorithms of type $\alpha$ of height $k$. If no such tower exists then $\mathrm{SCI}(\Xi,\Omega,\mathcal{M},\Lambda)_{\alpha} = \infty.$ If there exists a tower $\{\Gamma_n\}_{n\in\mathbb{N}}$ of type $\alpha$ and height one such that $\Xi = \Gamma_{n_1}$ for some $n_1 < \infty$, then we define $\mathrm{SCI}(\Xi,\Omega,\mathcal{M},\Lambda)_{\alpha} = 0$. The type $\alpha$ may be General, or Arithmetic, denoted respectively G and A. We may sometimes write $\mathrm{SCI}(\Xi,\Omega)_{\alpha}$ to simplify notation when $\mathcal{M}$ and $\Lambda$ are obvious. 
\end{defi}

We will let $\mathrm{SCI}(\Xi,\Omega)_{\mathrm{A}}$ and $\mathrm{SCI}(\Xi,\Omega)_{\mathrm{G}}$ denote the SCI with respect to an arithmetic tower and a general tower, respectively. Note that a general tower means just a tower of algorithms as in Definition \ref{tower_funct}, where there are no restrictions on the mathematical operations. Thus, clearly $\mathrm{SCI}(\Xi,\Omega)_{\mathrm{A}} \geq \mathrm{SCI}(\Xi,\Omega)_{\mathrm{G}}$. The definition of the SCI immediately induces the SCI hierarchy:

 \begin{defi}[The Solvability Complexity Index Hierarchy]
\label{1st_SCI}
Consider a collection $\mathcal{C}$ of computational problems and let $\mathcal{T}$ be the collection of all towers of algorithms of type $\alpha$ for the computational problems in $\mathcal{C}$.
Define 
\begin{equation*}
\begin{split}
\Delta^{\alpha}_0 &:= \{\{\Xi,\Omega\} \in \mathcal{C} \ \vert \   \mathrm{SCI}(\Xi,\Omega)_{\alpha} = 0\}\\
\Delta^{\alpha}_{m+1} &:= \{\{\Xi,\Omega\}  \in \mathcal{C} \ \vert \   \mathrm{SCI}(\Xi,\Omega)_{\alpha} \leq m\}, \qquad \quad m \in \mathbb{N},
\end{split}
\end{equation*}
as well as
\[
\Delta^{\alpha}_{1} := \{\{\Xi,\Omega\}  \in \mathcal{C}   \  \vert \ \exists \ \{\Gamma_n\}_{n\in \mathbb{N}} \in \mathcal{T}\text{ s.t. } \forall A \ d(\Gamma_n(A),\Xi(A)) \leq 2^{-n}\}. 
\]
\end{defi}

For problems, such as decision problems, that have extra structure on the metric space one can extend the SCI hierarchy. 

\begin{defi}[The SCI Hierarchy (totally ordered set)]\label{def:tot_ord}
Given the set-up in Definition \ref{1st_SCI} and suppose in addition that $\mathcal{M}$ is a totally ordered set. 
Define 
\begin{equation*}
\begin{split}
\Sigma^{\alpha}_0 &= \Pi^{\alpha}_0 = \Delta^{\alpha}_0,\\
\Sigma^{\alpha}_{1} &= \{\{\Xi,\Omega\} \in \Delta_{2}^{\alpha} \ \vert \  \exists \ \{\Gamma_{n}\} \in \mathcal{T} \text{ s.t. } \Gamma_{n}(A) \nearrow \Xi(A) \ \, \forall A \in \Omega\}, 
\\
\Pi^{\alpha}_{1} &= \{\{\Xi,\Omega\} \in \Delta_{2}^{\alpha} \ \vert \  \exists \ \{\Gamma_{n}\} \in \mathcal{T} \text{ s.t. } \Gamma_{n}(A) \searrow \Xi(A) \ \, \forall A \in \Omega\},
\end{split}
\end{equation*}
where $\nearrow$ and $\searrow$ denotes convergence from below and above respectively,
as well as, for $m \in \mathbb{N}$, 
\begin{equation*}
\begin{split}
\Sigma^{\alpha}_{m+1} &= \{\{\Xi,\Omega\} \in \Delta_{m+2}^{\alpha} \ \vert \  \exists \ \{\Gamma_{n_{m+1}, \hdots, n_1}\} \in \mathcal{T} \text{ s.t. }\Gamma_{n_{m+1}}(A) \nearrow \Xi(A) \ \, \forall A \in \Omega\}, \\
\Pi^{\alpha}_{m+1} &= \{\{\Xi,\Omega\} \in \Delta_{m+2}^{\alpha} \ \vert \  \exists \ \{\Gamma_{n_{m+1}, \hdots, n_1}\} \in \mathcal{T} \text{ s.t. }\Gamma_{n_{m+1}}(A) \searrow \Xi(A) \ \, \forall A \in \Omega\}.
\end{split}
\end{equation*}
\end{defi}

Schematically the general SCI hierarchy can be viewed as follows. 

\begin{equation}\label{SCI_hierarchy}
\begin{tikzpicture}[baseline=(current  bounding  box.center)]
  \matrix (m) [matrix of math nodes,row sep=1.2em,column sep=1.5em] {
  \Pi_0^{\alpha}   &                    & \MA{\Pi_1^{\alpha}} &    &  \Pi_2^{\alpha}&  & {}\\
  \Delta_0^{\alpha}&  \Delta_1^{\alpha} & \Sigma_1^{\alpha}\cup\Pi_1^{\alpha} & \Delta_2^{\alpha}&      \Sigma_2^{\alpha}\cup\Pi_2^{\alpha} & \Delta_3^{\alpha}& \cdots\\
	\Sigma_0^{\alpha}&                    & \MA{\Sigma_1^{\alpha}} & &  \Sigma_2^{\alpha}&  &{} \\
  };
 \path[-stealth, auto] (m-1-1) edge[draw=none]
                                    node [sloped, auto=false,
                                     allow upside down] {$=$} (m-2-1)
																		(m-3-1) edge[draw=none]
                                    node [sloped, auto=false,
                                     allow upside down] {$=$} (m-2-1)
																		
																		(m-2-2) edge[draw=none]
                                    node [sloped, auto=false,
                                     allow upside down] {$\subsetneq$} (m-2-3)
																		(m-2-3) edge[draw=none]
                                    node [sloped, auto=false,
                                     allow upside down] {$\subsetneq$} (m-2-4)
																		(m-2-4) edge[draw=none]
                                    node [sloped, auto=false,
                                     allow upside down] {$\subsetneq$} (m-2-5)
																		(m-2-5) edge[draw=none]
                                    node [sloped, auto=false,
                                     allow upside down] {$\subsetneq$} (m-2-6)
																		(m-2-6) edge[draw=none]
                                    node [sloped, auto=false,
                                     allow upside down] {$\subsetneq$} (m-2-7)

												(m-2-1) edge[draw=none]
                                    node [sloped, auto=false,
                                     allow upside down] {$\subsetneq$} (m-2-2)
											 (m-2-2) edge[draw=none]
                                    node [sloped, auto=false,
                                     allow upside down] {$\subsetneq$} (m-1-3)
											(m-2-2) edge[draw=none]
                                    node [sloped, auto=false,
                                     allow upside down] {$\subsetneq$} (m-3-3)
											 (m-1-3) edge[draw=none]
                                    node [sloped, auto=false,
                                     allow upside down] {$\subsetneq$} (m-2-4)
																		(m-3-3) edge[draw=none]
                                    node [sloped, auto=false,
                                     allow upside down] {$\subsetneq$} (m-2-4)
																		(m-2-4) edge[draw=none]
                                    node [sloped, auto=false,
                                     allow upside down] {$\subsetneq$} (m-1-5)
											(m-2-4) edge[draw=none]
                                    node [sloped, auto=false,
                                     allow upside down] {$\subsetneq$} (m-3-5)
																		(m-1-5) edge[draw=none]
                                    node [sloped, auto=false,
                                     allow upside down] {$\subsetneq$} (m-2-6)
																		(m-3-5) edge[draw=none]
                                    node [sloped, auto=false,
                                     allow upside down] {$\subsetneq$} (m-2-6)
											(m-2-6) edge[draw=none]
                                    node [sloped, auto=false,
                                     allow upside down] {$\subsetneq$} (m-1-7)
																		(m-2-6) edge[draw=none]
                                    node [sloped, auto=false,
                                     allow upside down] {$\subsetneq$} (m-3-7);
																		
\end{tikzpicture}
\end{equation}
Note that the $\Sigma_1^{\alpha}$ and $\Pi_1^{\alpha}$ classes become crucial in computer-assisted proofs. 

Finally, we define $\Delta^A_{1,\mathrm{b}} \subset \Delta^A_{1}$ to be the set of $\Delta^A_{1}$ problems for which there exists algorithms with bounded runtime. In particular, 
\[
\Delta^A_{1,\mathrm{b}} = \{\{\Xi,\Omega\} \in \Delta_1 \ \vert \  \exists \ \{\Gamma_{n}\} \in \mathcal{T} \text{ s.t. } \sup_{A \in \Omega} d(\Gamma_{n}(A),\Xi(A)) \leq 2^{-n}, \, \sup_{A \in \Omega} \mathrm{runtime}(\Gamma_n(A)) < \infty\}.
\]

\subsection{The main theorems in the SCI hierarchy language }
Our theorems are deliberately written in layman terms, thus, in order to make the statements absolutely precise we specify what the statements are in the language of the SCI hierarchy.  
In all of the computational problems the domain $\Omega$ will consist of functions and the 
collection $\Lambda$ is defined as follows:
\[
\Lambda = \{f_x : \Omega \rightarrow \mathbb{C} \, \vert \, f_x(\varphi) = \varphi(x), x \text{ has rational entries}\}.
\]
 It is implicitly assumed that there is an ordering of the countable elements in $\Lambda$, hence, we will always have that $\Lambda = \{f_m\}_{m\in\mathbb{N}}$. 
 
\begin{itemize}

\item[(i)] {\it Local solution of the free Schr\"odinger equation: Theorem \ref{Th:main_thrm_lower_bound}.}
The statements of Theorem \ref{Th:main_thrm_lower_bound} in the SCI hierarchy language are as follows. Define $\Xi^1_{\mathrm{free}} : \Omega^1_{\mathrm{free}} \rightarrow \mathcal{M}_1 = L^2(O)$ and $\Xi^2_{\mathrm{free}} : \Omega^2_{\mathrm{free}} \rightarrow \mathcal{M}_2 = H^\rho_0(\mathbb R)$ as in \eqref{eq:free_domain} and \eqref{eq:free_Xi}. Then 
\[
\{\Xi^1_{\mathrm{free}},\Omega^1_{\mathrm{free}},\mathcal{M}_1,\Lambda\}  \notin \Delta_1^G.
\]
Moreover, when considering $\{\Xi^2_{\mathrm{free}},\Omega^2_{\mathrm{free}},\mathcal{M}_2,\Lambda\}$, then any sequence of general algorithms $\{\Gamma_n\}$ with $\Gamma_n : \Omega^2_{\mathrm{free}} \rightarrow \mathcal{M}_2$ and $d(\Gamma_n(A), \Xi^2_{\mathrm{free}}(A) \leq 2^{-n}$ will have 
\[
\sup_{A \in \Omega^2_{\mathrm{free}}}\sup\lbrace m \in \mathbb{N} \, \vert \, f_m \in \Lambda_{\Gamma_n}(A) \rbrace = \infty,
\]
in particular, the runtime is not uniformly bounded.

\item[(ii)] {\it Global solution of the linear Schr\"odinger equation: Theorem \ref{BigTh:main_thrm2}.} 
The statements of Theorem \ref{BigTh:main_thrm2} in the SCI hierarchy language are as follows. Define 
$\Xi_{\mathrm{Lin}} : \Omega_{\mathrm{Lin}} \rightarrow \mathcal{M} =  L^2(\mathbb{R}^d)$ and 
$
\Xi_{\mathrm{Lin2}} : \Omega_{\mathrm{Lin2}} \rightarrow \mathcal{M} 
$
as in Theorem \ref{BigTh:main_thrm2}. Then
\begin{equation*}
\begin{split}
&\{\Xi_{\mathrm{Lin}} : \Omega_{\mathrm{Lin}}, \mathcal{M},\Lambda\}  \in \Delta_{1,\mathrm{b}}^A,
\\
&\{\Xi_{\mathrm{Lin2}} : \Omega_{\mathrm{Lin2}}, \mathcal{M},\Lambda\}  \in \Delta_{1,\mathrm{b}}^A.
\end{split}
\end{equation*}

\item[(iii)] {\it Blow up of focusing NLS cannot be decided: Theorem \ref{Th:main_thrm_blow}.} Consider the decision problems 
\[
\{\Xi_{\mathrm{BU}(1)},\Omega_{\mathrm{BU}(1)}\}, \quad \{\Xi_{\mathrm{BU}(2)},\Omega_{\mathrm{BU}(2)}\}
\]
defined in \eqref{eq:dec_prob_1} and \eqref{eq:dec_prob_2}. Then 
\[
\{\Xi_{\mathrm{BU}(1)},\Omega_{\mathrm{BU}(1)}\}  \notin \Pi_1^G, \quad 
\{\Xi_{\mathrm{BU}(1)},\Omega_{\mathrm{BU}(1)}\}  \notin \Sigma_1^G.
\]

\item[(iv)] {\it Upper bounds: global solution - NLS: Theorem \ref{BigTh:main_thrm2_NLS}.} The statements of Theorem \ref{BigTh:main_thrm2_NLS} in the SCI hierarchy language are as follows.
Define 
$\Xi_{\mathrm{NLS}} : \Omega_{\mathrm{NLS}} \rightarrow \mathcal{M} =  L^2(\mathbb{R}^d)$ and 
$
\Xi_{\mathrm{NLS2}} : \Omega_{\mathrm{NLS2}} \rightarrow \mathcal{M} 
$
as in Theorem \ref{BigTh:main_thrm2_NLS}. 
Then
\begin{equation*}
\begin{split}
&\{\Xi_{\mathrm{Lin}}, \Omega_{\mathrm{Lin}}, \mathcal{M},\Lambda\}  \in \Delta_{1,\mathrm{b}}^A,
\\
&\{\Xi_{\mathrm{Lin2}},\Omega_{\mathrm{Lin2}}, \mathcal{M},\Lambda\}  \in \Delta_{1,\mathrm{b}}^A.
\end{split}
\end{equation*}
\item[(v)] {\it Upper bounds: The discrete NLS: Theorem \ref{Th:main_thrm_discNLS}.} The statements of Theorem \ref{Th:main_thrm_discNLS} in the SCI hierarchy language are as follows.
Define 
$\Xi_{\mathrm{discNLS}}: \Omega_{\mathrm{NLS}} \rightarrow \mathcal{M} =  \ell^2(\mathbb Z^d)$. 
Then
\[
\{\Xi_{\mathrm{discNLS}}, \Omega_{\mathrm{NLS}}, \mathcal{M},\Lambda\}  \in \Delta_{1,\mathrm{b}}^A.
\]
\end{itemize}
 
 \subsection{The SCI hierarchy and computer-assisted proofs}
\label{sec:comp_ass_proofs}
Note that  $\Delta_1^A$ is the class of problems that are computable according to Turing's definition of computability \cite{Turing_Machine}. In particular, there exists an algorithm such that for any $\epsilon > 0$, the algorithm can produce an $\epsilon$-accurate output. Most infinite-dimensional spectral problems, unlike the finite-dimensional case, are $\notin \Delta_1^A.$ The simplest way to see this is to consider the problem of computing spectra of infinite diagonal matrices. Since this problem is the simplest of the infinite computational spectral problems and does not lie in $\Delta_1^A$, very few interesting infinite-dimensional spectral problems are actually in $\Delta_1^A$. This is why most of the literature on spectral computations provides algorithms that yield $\Delta_2^A$ classification results. In particular, an algorithm will converge, but error control may not be possible. 

Problems that are not in $\Delta_1^A$ are computed daily in the sciences, simply because numerical simulations may be suggestive rather than providing a rock-solid truth. Moreover, the lack of error control may be compensated for by comparing with experiments. However, this is not possible in computer-assisted proofs, where $100\%$ rigour is the only approach accepted. It may, therefore, be surprising that there are examples of famous conjectures that have been proven with numerical calculations of problems that are not in $\Delta_1^A$, i.e. problems that are non-computable according to Turing. 
A striking example is the proof of Kepler's conjecture \cite{Hales_Annals, hales_Pi}, where the decision problems computed are not in $\Delta_1^A$. The decision problems are of the form of deciding feasibility of linear programs given irrational inputs, shown in \cite{SCI_optimization} to not lie in $\Delta_1^A$. Similarly, the problem of obtaining the asymptotic of the ground state of the operator 
\[
H_{dZ} = \sum_{k=1}^d(-\Delta_{x_k} - Z|x_k|^{-1}) + \sum_{1\leq j \leq k \leq d}|x_j-x_k|^{-1},
\]
as $Z \rightarrow \infty$ was obtained by a computer-assisted proof \cite{fefferman1990, fefferman1992, fefferman1993aperiodicity,  fefferman1994, fefferman1994_2, fefferman1995, fefferman1996interval, fefferman1996, fefferman1997} by Fefferman and Seco, proving the Dirac-Schwinger conjecture, that relied on problems that were not in $\Delta_1^A$. The SCI hierarchy can describe these paradoxical phenomena.

\subsubsection{The $\Sigma^A_1$ and $\Pi^A_1$ classes}
The key to the paradoxical phenomena lies in the $\Sigma^A_1$ and $\Pi^A_1$ classes. These classes of problems are larger than $\Delta^A_1$, but can still be used in computer-assisted proofs. Indeed, if we consider computational spectral problems that are in $\Sigma^A_1$, then there is an algorithm that will never provide incorrect output. The output may not include the whole spectrum, but it is always sound. Thus, conjectures about operators never having spectra in a certain area could be disproved by a computer-assisted proof. Similarly, $\Pi^A_1$ problems would always be approximated from above, and thus conjectures on the spectrum being in a certain area could be disproved by computer simulations. 

In both of the above examples (the proof of the Dirac-Schwinger conjecture and Kepler's conjecture), one implicitly shows  that the relevant computational problems in the computer-assisted proofs are in $\Sigma^A_1$. 

\section{Roadmap to the proofs}

\subsection{Outline of the algorithm}
\label{sec:OoA}
For the linear and nonlinear Schr\"odinger equation as studied under Assumption \ref{ass:2} in Theorems \ref{BigTh:main_thrm2} and  \ref{BigTh:main_thrm2_NLS} it is enough to assume that the initial state satisfies for some explicit $C>0$ and some fixed $\varepsilon>0$
\[ \Vert (-\Delta + \vert x \vert^2)^{\varepsilon} \varphi_0 \Vert_{L^2} \le C\]
and for the linear Schr\"odinger equation, and some potentially different but arbitrary $\varepsilon>0$, the potentials $V,V_{\operatorname{con}} \in L^{\infty}(\RR^d) \cap W_{\operatorname{loc}}^{\varepsilon,p}(\RR^d)$ satisfy 
$
\Vert V\Vert_{L^{\infty}}, \Vert V_{\operatorname{con}}\Vert_{L^{\infty}} \le C.
$
Moreover, there exists a map
\begin{equation*}
 \mathbb{Q}_{+} \rightarrow \mathbb{Q}_{+} \text{ such that } \Vert V \Vert_{W^{\varepsilon,p}(B(0,r))},   \Vert V_{\operatorname{con}} \Vert_{W^{\varepsilon,p}(B(0,r))} \le \Phi(r).
\end{equation*}
It is then possible, as discussed in the proof of Theorems \ref{BigTh:main_thrm2} and \ref{BigTh:main_thrm2_NLS} to numerically approximate the initial state and potentials with smooth ones satisfying, among others, all of the following assumptions so that the algorithms of Theorems \ref{BigTh:main_thrm2} and  \ref{BigTh:main_thrm2_NLS} apply:

\begin{itemize}
\item \emph{Assumption on the input data for initial length scale estimate:}  We assume to have the following a priori estimates on the input data available to the algorithm: Three explicit constants $C_1,C_2,C_3>0$:
\begin{enumerate}
\item For an initial state $\varphi_0  \in H^2_2(\mathbb R^d)$ we require that $\Vert \varphi_0 \Vert_{H^2_2} \le C_1.$
\item For a control function $u \in W^{1,1}_{\operatorname{pcw}}((0,T))$ we require that $\Vert u \Vert_{W^{1,1}_{\operatorname{pcw}}((0,T))} \le C_2.$
\item For static and control potentials $V=W_{\operatorname{reg}}+W_{\operatorname{sing}}$, $V_{\operatorname{con}}$ satisfying a standard order condition \emph{(Assumption \ref{ass1})} we require that 
\[ \Vert W_{\operatorname{sing}} \Vert_{L^p} , \Vert \langle \bullet \rangle^{-2} W_{\operatorname{reg}} \Vert_{L^{\infty}}, \Vert \langle \bullet \rangle^{-2} V_{\operatorname{con}} \Vert_{L^{\infty}}\le C_3. \]
\end{enumerate} 
\item \emph{Step 1a - Initial length scale estimate:} In the first step, we assume the algorithm is given constants $C_1$ to $C_3$ and a final time $T$, only. 
It follows then from Lemmas \ref{eos} and \ref{NLS} that the solution to the (nonlinear) Schr\"odinger equation can be estimated, by a recursively defined constant $C_{C_1,C_2,C_3,T}>0,$ as
\[ \left\lVert \psi \right\rVert_{L^{\infty}((0,T),H^2_2(\mathbb R^d))}  \le C_{C_1,C_2,C_3,T}.\]
\item \emph{Step 1b - Restriction of the domain:} Given an error threshold $\varepsilon>0$, our Theorem \ref{theo2}, see also the preliminary discussion stated just before Theorem \ref{theo2}, then yields explicit estimates to identify a domain $\Omega$ or radius $R$ of bounded size such that using the a priori estimate $C_{C_1,C_2,C_3,T}$, the time-evolution in this bounded domain coincides with the true solution on the entire space up to an error $\varepsilon>0$ in $L^2$ norm.
\item \emph{Step 2 - Discretization of input data:} 
We assume the algorithm is able to evaluate the above input data on $\Omega$, which we can now freely modify outside $\Omega$. In particular, we can assume without loss of generality that our potentials are bounded at infinity. For our numerical scheme, we now impose slightly stronger assumptions to provide explicit rates of convergence in our numerical methods. 

\medskip
\underline{Linear Schr\"odinger equation:} The algorithm samples $(\varphi_0,V) \in  \Omega_{\mathrm{Lin}}$ with $\Omega_{\mathrm{Lin}}$ as in Assumption \ref{ass:1}.
\medskip

\underline{Nonlinear Schr\"odinger equation:} The algorithm samples $(\varphi_0,V) \in  \Omega_{\mathrm{NLS}}$ with $\Omega_{\mathrm{NLS}}$ as in Assumption \ref{ass:1}.

The algorithm then computes the \emph{cubic discretization}, by numerically evaluating the integral stated in Definition \ref{cubdisc}, using quasi MC methods \cite{MC}, of the preceding objects with explicit error bounds, see Proposition \ref{convrate}. The required cube size of the cubic discretization is determined by bearing in mind the additional error in the numerical schemes used in the subsequent step:
\item \emph{Step 3- Numerical methods:}  For the solution to the linear Schr\"odinger equation with time independent Schr\"odinger operator, we use the Crank-Nicholson method \eqref{eq:discrule} and the Strang splitting scheme \eqref{eq:Stranglin} to include the defocussing NLS or time-dependent control potential with explicit convergence rates. The convergence of the Crank-Nicholson scheme, with error bounds, is shown in Subsec. \ref{sec:LSE} and the convergence of the splitting scheme in Subsec. \ref{sec:defoc}. 
\end{itemize}

\begin{rem}
Our assumptions on the singular potentials include standard examples of singular potentials such as the Coulomb potential $W_{\operatorname{sing}}(x):=1/\vert x \vert$ in $\RR^3$, which has the property that for a smooth cut-off function $\chi_{\mathcal B_{\varepsilon}(0)}$ 
supported away from $\mathcal B_{\varepsilon}(0)$ 
such that $\chi_{\mathcal B_{\varepsilon}(0)}(x)=1$ for $x \notin \mathcal B_{2\varepsilon}(0)$ 
we can define a smooth approximation potential $\widetilde{W}_{\operatorname{sing}}(\varepsilon,x):=\chi_{\mathcal B_{\varepsilon}(0)}(x) /\vert x \vert \in C^{\infty}(\mathbb R^3)$ 
with the property that
\[ \Vert W_{\operatorname{sing}}-\widetilde{W}_{\operatorname{sing}}(\varepsilon) \Vert_{L^2} \le \left(\int_{\mathcal B_{2\varepsilon}(0)} \frac{1}{\vert x \vert^2} \ dx \right)^{1/2} = \sqrt{8\pi \varepsilon}.\] 
\end{rem}

\smallsection{Notation}
We introduce (generalized) Sobolev spaces for parameters $\rho,\eta  \ge 0$ 
\begin{equation}
\begin{split}
&H^\rho_\eta(\mathbb R^d) :=\left\{ f \in H^\rho(\mathbb R^d); \langle \bullet \rangle^{\eta} f \in L^2(\mathbb R^d) \right\}\text{ with norms }\\
&\left\lVert f \right\rVert^2_{H^\rho_{\eta}(\mathbb R^d)}:=\left\lVert \langle\bullet \rangle^{\rho} \widehat{f} \right\rVert_{L^2(\mathbb R^d)}^2+  \left\lVert  \langle\bullet \rangle^{\eta} f \right\rVert_{L^2(\mathbb R^d)}^2
\end{split}
\end{equation} 
where $\langle x \rangle :=\left(1+\vert x \vert^2\right)^{1/2}.$ For $\alpha_1> \alpha_2 \ge 0$ and $\beta_1>\beta_2\ge 0$ the inclusion of generalized Sobolev spaces $H^{\alpha_1}_{\beta_1}(\mathbb R^d) \hookrightarrow H^{\alpha_2}_{\beta_2}(\mathbb R^d)$ is compact \cite[Theorem B.3]{DZ18}.
Moreover, for the self-adjoint positive-definite operator $S=-\Delta+\vert x \vert^2$ on $L^2(\RR^d)$, we define the canonical norm on the domain $D(S^{\nu})$ for any $\nu \ge 0$ by
\[ \Vert f \Vert_{S^{\nu}}:= \Vert S^{\nu} f \Vert_{L^2}.\]
The space of piecewise $W^{1,1}$ functions on some interval $(0,T)$ is denoted by $W^{1,1}_{\operatorname{pcw}}(0,T)$ and consists of all functions $u$ that there exists a finite partition $I_i:=(t_i,t_{i+1})$ of $(0,T)$ such that 
$
 \Vert u \Vert_{W^{1,1}_{\operatorname{pcw}}(0,T)} = \sum_{i=1}^{n} \Vert u \Vert_{W^{1,1}(I_i)}.
 $

We frequently omit the domain of functions in function spaces to shorten the notation. The ball centred at $x_0$ with radius $r$ is denoted as $\mathcal B_r(x_0).$
We denote the spatially averaged integral as 
\[
\fint_A f(x) \ dx = \frac{1}{\vert A \vert} \int_A f(x) \ dx,
\] 
where $\vert A \vert = \int_A \ dx.$
The space of functions of bounded variation is denoted by $\operatorname{BV}(\mathbb R^d).$ The norm on the intersection $X \cap Y\ni x$ of two normed spaces is $\Vert x \Vert_{X \cap Y} = \operatorname{max} \{ \Vert x \Vert_X, \Vert x \Vert_Y \}.$
\medskip

We also introduce discrete weighted spaces
\begin{equation}
\label{eq:weighted}
\ell^2_{\eta}(\mathbb Z^d):=\left\{ (x_n)_{n \in \mathbb Z^d}; \Vert x \Vert^2_{\ell^2_{\eta}}:=\sum_{n \in \mathbb Z^d} \langle n \rangle^{2\eta} \vert x_n \vert^2< \infty \right\}.
\end{equation}

Throughout the text, we denote the standard mollifier by $\eta_t(x):=t^{-d}\eta(t^{-1}x)$ with $\eta \in C_c^{\infty}(\mathbb R^d, [0,\infty))$ where $\left\lVert \eta \right\rVert_{L^1(\mathbb R^d)}= 1$. We denote subsequences of sequences $(x_n)$ again by $(x_n)$ and time differentiation of a space-time dependent function $f$ is denoted by $f'$.

For a weakly, to some state $\psi$, convergent sequence $(\psi_n)$ we write $\psi_n \overset{}{\rightharpoonup} \psi$ and weak$^*$-convergence is denoted by $\psi_n \overset{\star}{\rightharpoonup} \psi.$
If there is a constant $C>0$, independent of $y$, such that $\left\lvert x \right\rvert \le C \left\lvert y \right\rvert$ we also write $\left\lvert x \right\rvert \lesssim \left\lvert y \right\rvert$ or $x=\mathcal O(\vert y \vert).$

\begin{rem}
\label{rem:genSob}
In Lemmas \ref{eos} and \ref{eos2} we establish existence of uniformly (in time) bounded solutions to the linear and nonlinear Schr\"odinger equation, respectively, in certain generalized Sobolev spaces $H^{\rho}_{\eta}(\mathbb R^d)$ on compact time intervals.
Such bounds allow us to identify bounded domains on which the solution is localized up to arbitrary small $L^2$-error. 
For any $R >0$ we have
\begin{equation}
\label{eq:outside}
 \left\lVert\varphi \indic_{\mathcal B_R(0)^c}  \right\rVert_{L^2(\mathbb R^d)} \le \left\lVert \varphi \right\rVert_{H_{\eta}(\mathbb R^d)}\/(1+R^2)^{-\eta/2} \le \left\lVert \varphi \right\rVert_{H_{\eta}(\mathbb R^d)}\/R^{-\eta}  
 \end{equation}
such that by choosing $R>(\left\lVert \varphi \right\rVert_{H_{\eta}(\mathbb R^d)}/\varepsilon)^{1/\eta}$ for some $\varepsilon>0$ it follows that $\left\lVert\varphi \indic_{\mathcal B_R(0)^c}  \right\rVert_{L^2(\mathbb R^d)}< \varepsilon.$
\end{rem}

\medskip 

We refer to both linear and nonlinear Schr\"odinger equations in this text as Schr\"odinger equations and write $\mathbb R^d$ in estimates that hold true for both the linear and nonlinear Schr\"odinger equations. However, in case of nonlinear Schr\"odinger equations we restrict us henceforth to the cubic and quintic NLS in $d=1.$

\section{Existence of solutions in generalized Sobolev spaces}
\label{eosol}
To show existence of solutions, we assume that the \emph{singular part} of the pinning potential is zero-bounded with respect to the negative Laplacian. Sufficient conditions in any dimension for this to hold are summarized in the following Remark \ref{rem:bounded}.
Henceforth, we assume that $u \in  W^{1,1}_{\operatorname{pcw}}(0,T)$. For both potentials in \eqref{eq:bilSchr} we impose the following integer $k$-parameterized assumption:

\begin{Assumption}[Potentials]
\label{ass1}
Consider a decomposition of the pinning potential $V=W_{\operatorname{sing}}+W_{\operatorname{reg}}$.The pinning potential $V$ and control potential $V_{\operatorname{con}}$ satisfy a standard condition if $W_{\operatorname{reg}}$ and $V_{\operatorname{con}}$ 
\begin{equation}
\begin{split}
\label{eq:potentials}
&W_{\operatorname{sing}} \in L^p(\mathbb{R}^d), \text{ with  }p<\infty \text{ as in Remark }\ref{rem:bounded},\text{ and both } \\
& \langle \bullet \rangle^{-2} W_{\operatorname{reg}}\text{ and } \langle \bullet \rangle^{-2}  V_{\operatorname{con}} \in L^{\infty}(\mathbb R^d).
\end{split}
\end{equation}
\end{Assumption}

We want to think of $W_{\operatorname{sing}}$ as the \emph{localized singular part} of the pinning potential $V$, that is relatively bounded with respect to the Laplacian, whereas $W_{\operatorname{reg}}$ describes the regular part of the pinning potential that is allowed to be unbounded as $\left\lvert x \right\rvert \rightarrow \infty,$ but should, in the above sense, not grow faster than the harmonic potential $\sim \left\lvert x \right\rvert^2.$ 

We will now start by discussing the existence of solutions to the linear Schr\"odinger equation and then extend this result to the NLS afterwards.
\subsection{The linear Schr\"odinger equation}
\label{sec:linnum}
The solution to \eqref{eq:semicl} can be constructed by a limiting procedure: 
For an approximate identity $(\eta_n) $, we consider the family of approximate Schr\"odinger equations
\begin{equation}
\begin{split}
\label{eq:bilSchrred}
i \partial_t \psi_n(x,t) &= \left(-\mu^2\Delta+(W_{\operatorname{sing}}+W_{\operatorname{reg}}\indic_{\mathcal B_n(0)})*\eta_n+ (V_{\operatorname{TD}}(t)\indic_{\mathcal B_n(0)})*\eta_n\right)\psi_n(x,t), \\
\psi_n(\bullet,0)&=\varphi_0 \in H^2_2(\mathbb R^d).
\end{split}
\end{equation}
In the sequel, we use the notation $V_{1_n}:=W_{\operatorname{sing}}*\eta_n$, $V_{2_n}:=(W_{\operatorname{reg}}\indic_{\mathcal B_n(0)}))*\eta_n$, $V_n:=V_{1_n}+V_{2_n}$, and $ V_{\operatorname{TD}_n}(t):=(V_{\operatorname{TD}}(t)\indic_{\mathcal B_n(0)})*\eta_n.$ We then take a suitable limit $n \rightarrow \infty$ and show that this provides a solution to the Schr\"odinger equation \eqref{eq:semicl}.

The existence of unique solutions to the mollified equation \eqref{eq:bilSchrred} follows from fixed-point arguments \cite[\S 5.2]{LiYo}. 
There are, of course, more restrictive conditions on the potentials such that the limiting construction is redundant. We summarize some of them in the following remark:
\begin{rem}[Essentially bounded potentials]
Let $W_{\operatorname{reg}}, V_{\operatorname{con}}$ be bounded multiplication operators on $L^2(\mathbb R^d)$ and $u \in L^1((0,T),\mathbb R)$ then \eqref{eq:semicl} possesses a unique mild solution in $C((0,T); L^2(\mathbb R^d))$ for any initial datum $\varphi_0 \in L^2(\mathbb R^d).$ \\
Extending this to solutions of higher spatial regularity is immediate: Let $V, V_{\operatorname{con}} \in W^{k,\infty}(\mathbb R^d)$ and $u \in L^1((0,T),\mathbb R)$. Equation \eqref{eq:semicl} possesses a unique mild solution in $C((0,T); H^k(\mathbb R^d))$ for any initial datum $\varphi_0 \in H^k(\mathbb R^d).$ 
\end{rem}

We start by explaining that smoothness of the initial state is necessary to ensure that the Schr\"odinger evolution preserves the decay of the initial state, cf. \cite[Lemma $3$]{BKP}. We illustrate this by the following explicit initial state that is compactly supported, of low-regularity, and disperses immediately under the free Schr\"odinger dynamics:
\begin{ex}[Spatial regularity and decay] 
The indicator function $\varphi_0:=\indic_{[-\tfrac{1}{2},\tfrac{1}{2}]}$ is in any $H^{\rho}_{\eta}(\mathbb R)$ for $\rho<\tfrac{1}{2}$ \emph{(low regularity)} and any $\eta>0$ \emph{(rapid decay)}. Its Fourier transform is $\widehat{\varphi_0}(p)=\operatorname{sinc}(p).$
Under the free evolution $i \varphi'(t) = (- \Delta \varphi)(t),$
the solution satisfies then $\widehat{\varphi}(p,t) = e^{-i\vert p \vert^2 t} \widehat{\varphi_0}(p).$
Thus, although the initial state is compactly supported one finds that already the first moment is not square integrable $\langle \bullet \rangle \varphi(t) \notin L^2(\mathbb R)$ for $t \neq 0$, since
\[ \mathcal F(x \mapsto -i x\varphi(t,x))(p) =  \partial_p \widehat{\varphi}(p,t)= e^{-i\vert p \vert^2 t} \left( \frac{\cos(p)}{p}-\frac{\sin(p)}{p^2}-2it \sin(p) \right)\notin L^2(\mathbb R)\]
which implies $\langle \bullet \rangle \varphi(t) \notin L^2(\mathbb R).$
In other words, initial states of low spatial regularity but compact support can strongly disperse under the Schr\"odinger dynamics. 
\end{ex}
We then define the Dirichlet and Neumann spaces for $\rho \ge 2$
\begin{equation}
\begin{split}
\label{eq:DirichletaNeumann}
X^D_{\rho}(\Omega)&:=H^{\rho}(\Omega)\cap H_0^1(\Omega) \text{ and }X^N_{\rho}(\Omega):=H^{\rho}(\Omega)\cap \left\{ \psi; \langle \nabla \psi, n \rangle\vert_{\partial \Omega}= 0 \right\}.
\end{split}
\end{equation}
In the following we denote either space in \eqref{eq:DirichletaNeumann} just by $X_{\rho}.$
The next Lemma, that relies on energy estimates, introduced in \cite[Theorem $1$]{BKP}, yields the existence of solutions to \eqref{eq:semicl} in $X_{\rho}(\Omega)$ and is established by showing that solutions to \eqref{eq:bilSchrred} possess a weak$^*$-convergent subsequence that converges to the (unique) solution of \eqref{eq:semicl} in generalized Sobolev spaces.

\begin{lemm}[Existence of solutions to linear Schr\"odiger eq.]
\label{eos}
Let $\varphi_0$ be an initial state in $X_2(\Omega)$ to \eqref{eq:bilSchr} and $\Omega=\mathcal B_r(0)$ or $\Omega=\mathbb R^d$. We consider control functions $u \in W^{1,1}_{\operatorname{pcw}}(0,T)$ and potentials $V,V_{\operatorname{con}}$ satisfying a standard condition with $p$ as in \emph{Assumption \ref{ass1}}. Then, there exists a solution $\psi \in L^{\infty}((0,T),X_2(\Omega))$ with $\psi' \in L^{\infty}((0,T),L^2(\Omega))$ to \eqref{eq:semicl} such that for $\Omega' = \mathbb R^d$ or if $\Omega$ was a ball $\mathcal B_r(0)$ then for $\Omega' = \mathcal B_{r'}(0)$ uniformly in $r' \ge r,$ we have for a recursively defined function 
\[C = C(T,\Vert u \Vert_{W^{1,1}_{\operatorname{pcw}}((0,T))},\Vert W_{\operatorname{sing}} \Vert_{L^p}, \Vert \langle \bullet \rangle^{-2} W_{\operatorname{reg}} \Vert_{L^{\infty}}, \Vert \langle \bullet \rangle^{-2}V_{\operatorname{con}} \Vert_{L^{\infty}},C_{\mu^2/2}),\]
where $C_{\mu^2/2}$ is defined in \eqref{eq:zero-bd}, that 
\begin{equation}
\label{eq:estm63}
 \left\lVert \psi \right\rVert_{L^{\infty}((0,T),H^{2,\operatorname{sem}}_2(\Omega'))} +\left\lVert \psi' \right\rVert_{L^{\infty}((0,T),L^2(\Omega'))} \le C \left\lVert \varphi_0 \right\rVert_{H_2^{2,\operatorname{sem}}(\Omega')}.
 \end{equation}
Moreover, as functions in $L^{\infty}((0,T); X_2(\Omega))$, with time-derivative in $L^{\infty}((0,T); L^2(\Omega)),$ there exists a subsequence of solutions to \eqref{eq:bilSchrred} such that both 
\begin{equation}
\begin{split}
&\psi_n \overset{\star}{\rightharpoonup} \psi,\text{ in }L^{\infty}((0,T),X_2(\Omega)) \text{ and }\psi_n' \overset{\star}{\rightharpoonup} \psi'\text{ in }L^{\infty}((0,T); L^2(\Omega)).
\end{split}
\end{equation} Finally, the solution $\psi$ to \eqref{eq:bilSchr} is unique in $X_2(\Omega)$.
\end{lemm}
\begin{proof}
By multiplying \eqref{eq:bilSchrred} with $\left\lvert x \right\rvert^4  \overline{\psi_n}$, integrating by parts, and taking the imaginary part it follows that 
\begin{equation}
\label{eq:initeq}
 \frac{d}{dt} \int_{\Omega} \left\lvert x \right\rvert^4 \left\lvert \psi_n(x,t) \right\rvert^2 \ dx \lesssim \int_{\Omega}   \mu^2 \left\lvert x \right\rvert^2 \left\lvert \nabla \psi_n(x,t) \right\rvert^2 \ dx + \left\lVert \psi_n(t) \right\rVert_{H_2(\Omega)}^2.
 \end{equation}
To obtain a bound on the first term on the right-hand side of \eqref{eq:initeq}, we multiply \eqref{eq:bilSchrred} by $\left\lvert x \right\rvert^2 \overline{ \psi_n}$ and integrate this time the real part over $\Omega$ such that by the assumption on the potentials 
\begin{equation}
\label{eq:estm1}
 \int_{\Omega} \mu^2  \left\lvert x \right\rvert^2 \left\lvert \nabla \psi_n(x,t) \right\rvert^2 \ dx \lesssim \mu^2 \left\lVert \nabla \psi_n(t) \right\rVert^2_{L^2(\Omega)} + \left\lVert \psi_n(t) \right\rVert^2_{H_2(\Omega)}+ \left\lVert \psi_n'(t) \right\rVert^2_{L^2(\Omega)} .
 \end{equation}
Finally, to obtain a bound on the gradient appearing on the right-hand side of \eqref{eq:estm1} we multiply \eqref{eq:bilSchrred} by $\overline{\psi}_n$ and integrate the equation \eqref{eq:bilSchrred} over the entire space. Then, for the real part of that expression we obtain, by the zero-boundedness of the singular part of the potential, the desired bound 
\[\mu^2\left\lVert \nabla \psi_n(t) \right\rVert^2_{L^2(\Omega)}  \lesssim \left\lVert \psi_n'(t) \right\rVert^2_{L^2(\Omega)} + C_{\mu^2/2}\left\lVert \psi_n(t) \right\rVert^2_{H_1(\Omega)}.\]

\medskip

Combining this estimate on the gradient with \eqref{eq:estm1} yields, by invoking \eqref{eq:initeq} and the preservation of the $L^2$ norm for solutions to \eqref{eq:bilSchrred}, 
\[ \frac{d}{dt} \left\lVert \psi_n(t) \right\rVert_{H_{2}(\Omega)}^2=   \frac{d}{dt} \int_{\Omega}  (1+ \left\lvert x \right\rvert^4) \left\lvert  \psi_n(x,t) \right\rvert^2 \ dx \lesssim \left\lVert \psi_n'(t) \right\rVert^2_{L^2(\Omega)} + C_{\mu^2/2}\left\lVert \psi_n(t) \right\rVert^2_{H_2(\Omega)}.\]
Integrating this bound over a compact time interval $[0,t]$ shows that 
 \begin{equation}
\begin{split}
\label{eq:important one}
\left\lVert \psi_n(t) \right\rVert_{H_{2}(\Omega)}^2 \lesssim \left\lVert \psi_n(0) \right\rVert_{H_{2}(\Omega)}^2 + \int_0^t \left(\left\lVert \psi_n'(s) \right\rVert^2_{L^2(\Omega)} +C_{\mu^2/2} \left\lVert \psi_n(s) \right\rVert^2_{H_2(\Omega)} \right)\ ds.
  \end{split}
  \end{equation} 
We now want to bound $ \left\lVert \psi_n'(s) \right\rVert^2_{L^2(\Omega)}$ on the right-hand side of \eqref{eq:important one}. This term is the only missing ingredient to control the $H^2$ norm of the solution to the Schr\"odinger equation, since directly from the Schr\"odinger equation \eqref{eq:bilSchrred} we conclude that
\begin{equation}
\begin{split}
\label{eq:h2estm}
\left\lVert \psi_n(t) \right\rVert_{H^{2,\operatorname{sem}}(\Omega)}& \lesssim \left\lVert \psi_n'(t) \right\rVert_{L^2(\Omega)} + \left\lVert \psi_n(t) \right\rVert_{H_2(\Omega)}. 
\end{split}
\end{equation}

To bound the time derivative appearing on the right-hand side of \eqref{eq:h2estm} we write $\chi_n(t):=\psi_n'(t)$ and observe that this function satisfies a PDE 
\begin{equation}
\label{eq:diffpde}
i \chi_n'(t)=- \mu^2\Delta \chi_n(t)+ V_n\chi_n(t)+V_{\operatorname{TD}_n}'(t)\psi_n(t)+V_{\operatorname{TD}_n}(t)\chi_n(t)
\end{equation}
for some initial value $\left\lVert \chi_n(0)\right\rVert_{L^2(\Omega)} \lesssim \left\lVert \varphi_0 \right\rVert_{H_2^{2,\operatorname{sem}}(\Omega)}$ where this bound follows from the Schr\"odinger equation at zero. Multiplying the PDE \eqref{eq:diffpde} by $\overline{\chi_n}$ and integrating in space yields by taking the imaginary part of that expression
\begin{equation*}
\begin{split}
 \frac{1}{2} \frac{d}{dt} \left\lVert \chi_n(t) \right\rVert_{L^2(\Omega)}^2 
 &= \Im \left\langle V_{\operatorname{TD}_n}'(t) \psi_n(t), \chi_n(t) \right\rangle_{L^2(\Omega)} \\
 &\le \left\lVert V_{\operatorname{TD}_n}'(t)/\langle \bullet \rangle^2 \right\rVert_{L^{\infty}(\Omega)} \left\lVert \psi_n(t) \right\rVert_{H_2(\Omega)} \left\lVert \chi_n(t) \right\rVert_{L^2(\Omega)}.
 \end{split}
 \end{equation*}

Integrating this bound in time shows that
\begin{equation*}
\begin{split}
 \left\lVert \psi_n'(t) \right\rVert_{L^2(\Omega)}^2 
  &\lesssim \left\lVert \varphi_0 \right\rVert^2_{H_2^{2,\operatorname{sem}}(\Omega)} + \int_0^t  \left\lVert \tfrac{V_{\operatorname{TD}_n}'(s)}{\langle \bullet \rangle^2} \right\rVert_{L^{\infty}(\Omega)}\left( \left\lVert \psi_n(s) \right\rVert_{H_2(\Omega)}^2 +\left\lVert \psi_n'(s) \right\rVert_{L^2(\Omega)}^2\right) \ ds.
 \end{split}
 \end{equation*}
Thus, combining this estimate with \eqref{eq:h2estm} and adding \eqref{eq:important one} to it, implies by Gr\"onwall's lemma that 
\begin{equation}
\label{eq:estm53}
\left\lVert \psi_n \right\rVert^2_{L^{\infty}((0,T),H^{2,\operatorname{sem}}_2(\Omega))} +  \left\lVert \psi_n' \right\rVert_{L^{\infty}((0,T),L^2(\Omega))}^2 \lesssim \left\lVert \varphi_0 \right\rVert^2_{H^{2,\operatorname{sem}}_2(\Omega)}.
\end{equation}
By Alaoglu's theorem, there is a subsequence of
 \[\psi_n, \psi \in W:= \left\{ u \in L^{\infty}((0,T); X_2^{\operatorname{sem}}(\Omega)); u' \in L^{\infty}((0,T);L^2(\Omega)) \right\}\]
 such that both $\psi_n \overset{\ast}{\rightharpoonup} \psi$ in $L^{\infty}((0,T); X_2(\Omega)\cap H^{2,\operatorname{sem}}_2(\Omega))$ and $\psi_n' \overset{\ast}{\rightharpoonup} \psi'$ in $L^{\infty}((0,T);L^2(\Omega))$.
 We note that since $W_{\operatorname{sing}} \in L^p(\mathbb R^d)$ it follows that $\lim_{n \rightarrow \infty} (W_{\operatorname{sing}} \indic_{\Omega'}) * \eta_n = W_{\operatorname{sing}}\indic_{\Omega'}$ in $L^p(\Omega)$ and in the $H^0_{-2-\varepsilon}(\mathbb R^d)$ sense we have the limits
 \[(W_{\operatorname{reg}} \indic_{\mathcal B_n(0)})*\eta_n \rightarrow W_{\operatorname{reg}} \operatorname{ and } \quad (V_{\operatorname{con}} \indic_{\mathcal B_n(0)})*\eta_n \rightarrow V_{\operatorname{con}}.\]
 The lower semicontinuity together with the continuity of the trace operator imply that in case of Dirichlet boundary conditions
 \[ \Vert  \psi \Vert_{L^{\infty}((0,T),L^2(\partial \Omega))} \le \liminf_{n \rightarrow \infty} \Vert  \psi_n \Vert_{L^{\infty}((0,T),L^2(\partial \Omega))} =0 \]
 and similarly for the Neumann boundary condition.
 
 Thus, taking the limit in the weak formulation of the linear Schr\"odinger equation for arbitrary $\zeta \in \mathscr S(\RR^d)$
\[\langle \psi_n(t), \zeta \rangle =  \langle \varphi_0, \zeta \rangle + \int_0^t \langle \psi_n'(s), \zeta  \rangle \ ds =  \langle \varphi_0, \zeta \rangle - i \int_0^t \langle \left(H_0 + V_{\operatorname{TD}}(s)\right)_n \psi_n(s), \zeta  \rangle \ ds \] 
 implies the existence of a distributional solution 
 \[\psi \in L^{\infty}((0,T),X_2(\Omega))\text{ with }\psi' \in  L^{\infty}((0,T),L^2(\Omega))\]
\[\langle \psi(t), \zeta \rangle 
=  \langle \varphi_0, \zeta \rangle + \int_0^t \langle \psi'(s), \zeta  \rangle \ ds =  \langle \varphi_0, \zeta \rangle - i \int_0^t \langle \left(H_0 + V_{\operatorname{TD}}(s)\right) \psi(s), \zeta  \rangle \ ds .\]
The uniqueness of the solution follows then for example from Gr\"onwall's lemma. Estimate \eqref{eq:estm63} follows from \eqref{eq:estm53} applied to $\psi.$
\end{proof}

The time-evolution operator defined by the linear Schr\"odinger equation \eqref{eq:bilSchr} will be denoted by $(\phi_{\operatorname{S}}(t,s))_{0\le s \le t \le T}.$

We continue by showing that the assumptions of Lemma \ref{eos} allow for both very dispersive and localized time-evolution. In particular, the following example shows that without further assumptions on the potentials, the exponential dispersion of the state cannot be improved.
\begin{ex}[(Inverted) harmonic oscillator]
Let $\varphi_0(x) = \pi^{-1/4} \exp(-x^2/2)$ and consider the Schr\"odinger equation
\[ i \partial_t \psi(t) = - \frac{1}{2} \left( \partial_x^2+ x^2\right) \psi(t) \text{ with } \psi(0)=\varphi_0.\]
The solution to this equation satisfies then
\[ \vert \psi(x,t)\vert^2 = \frac{1}{\pi^{1/2}A(t)} \exp(-x^2/A(t)^2) \]
where $A(t)^2 = \cosh(2t)$. In particular, the variance of the state increases exponentially fast. This is consistent with the Gronwall estimates in the proof of Lemma \ref{eos}.
\end{ex}
On the other hand, the initial state $\varphi_0$ in the above example is an eigenstate to the operator $ H=-\partial_x^2+ x^2$ and is therefore fixed under the time evolution of the Schr\"odinger equation. 

\medskip

These two examples illustrate that the time-evolution by the Schr\"odinger equation, satisfying the conditions of Lemma \ref{eos}, is highly model-dependent but, as shown in the previous Lemma \ref{eos}, always confined to a bounded domain. To derive estimates for practical applications it is therefore desirable to use a bound that is more tailored to the potential configuration.

\subsection{The nonlinear defocusing Schr\"odinger equation}
We now turn to the existence of solutions in generalized Sobolev spaces for the defocusing NLS. We start with a technical lemma to deal with the nonlinearities: 
In the following we write $F$ for the nonlinear term in the respective Schr\"odinger equation. In particular, $F_{\operatorname{S}}:=0$ for the linear Schr\"odinger equation and $F_{\sigma}(u):=\vert u \vert^{\sigma-1} u$ for the defocusing nonlinearity.
\begin{lemm}[Local Lipschitz conditions]
\label{F}
The \emph{cubic} $\sigma=3$ and \emph{quintic} $\sigma=5$ nonlinearity satisfy for every natural number $n \ge 0$ the weighted estimate
\begin{equation*}
\begin{split}
\left\lVert  \vert x \vert^n(F_{\sigma}(u)-F_{\sigma}(v))\right\rVert_{L^2(\mathbb R)} &\lesssim \left(\left\lVert u \right\rVert^{\sigma-1}_{H^1(\mathbb R)}+\left\lVert v \right\rVert^{\sigma-1}_{H^1(\mathbb R)} \right) \left\lVert \vert x \vert^n (u-v) \right\rVert_{L^2(\mathbb R)}.
\end{split}
\end{equation*}
Moreover, for any $k\ge 0$
\begin{equation}
\label{eq:Sob}
 \left\lVert F_{\sigma}(u) \right\rVert_{H^{k}(\RR)} \lesssim  \left\lVert u \right\rVert_{H^{k}(\mathbb R)}\left\lVert u \right\rVert^{\sigma-1}_{H^{1}(\mathbb R)}.
 \end{equation}
 \end{lemm}
\begin{proof}
The weighted estimate follows immediately from 
\begin{equation*}
\begin{split}
\left\lVert  \vert x \vert^n(F_{\sigma}(u)-F_{\sigma}(v))\right\rVert_{L^2(\mathbb R)} &\lesssim \left\lVert  \vert u\vert^{\sigma-1} \vert x \vert^n (u-v) \right\rVert_{L^2(\mathbb R)} +\left\lVert \vert x \vert^n (\vert u \vert^{\sigma-1}-\vert v\vert^{\sigma-1})   v \right\rVert_{L^2(\mathbb R)} \\
&\lesssim \left(\left\lVert u \right\rVert^{\sigma-1}_{L^{\infty}(\mathbb R)}+\left\lVert v \right\rVert^{\sigma-1}_{L^{\infty}(\mathbb R)}\right) \left\lVert \vert x \vert^n (u-v) \right\rVert_{L^2(\mathbb R)}\\
&\lesssim  \left(\left\lVert u \right\rVert^{\sigma-1}_{H^1(\mathbb R)}+\left\lVert v \right\rVert^{\sigma-1}_{H^1(\mathbb R)}\right) \left\lVert \vert x \vert^n (u-v) \right\rVert_{L^2(\mathbb R)}.
 \end{split}
 \end{equation*}
 To show the estimate \eqref{eq:Sob} on the derivatives, one can for example use the Fourier representation of the Sobolev spaces such that for $u,v \in H^{\max \{1,k\}}(\RR)$
 \begin{equation*}
\begin{split}
\langle \xi \rangle^{k}  \vert \widehat{u v } (\xi) \vert 
&\le \langle \xi \rangle^{k} \int_{\RR} \vert \widehat{u}(\xi-\eta) \widehat{v}(\eta) \vert d\eta \\ 
&\lesssim \int_{\mathbb R} \langle \xi -\eta  \rangle^{k}  \vert \widehat{u}(\xi-\eta) \widehat{v}(\eta) \vert d\eta +  \int_{\mathbb R} \langle \eta \rangle^{k}\vert \widehat{u}(\xi-\eta) \widehat{v}(\eta) \vert d\eta \\
& \lesssim \left( \vert \langle \bullet \rangle^{k}\widehat{u}\vert * \vert \widehat{v} \vert +\vert \widehat{u} \vert* \vert \langle  \bullet \rangle^{k}\widehat{v}\vert \right)(\xi). 
\end{split}
 \end{equation*}
Squaring and integrating this estimate and applying Young's convolution inequality then shows that ($\Vert \widehat v\Vert_{L^1} \lesssim \Vert v\Vert_{H^1}$ in dimension one)
\[ \Vert u v \Vert_{H^{k}} \lesssim \Vert u \Vert_{H^{k}} \Vert v \Vert_{H^{1}} +   \Vert u \Vert_{H^{1}} \Vert v \Vert_{H^{k}}\]
which shows that we can iteratively \emph{peal off} individual factors from the nonlinearity such that a single $H^k(\RR)$ norm remains. The $H^1(\RR)$-norm satisfies by Sobolev's embedding theorem in one dimension $ \Vert uv \Vert_{H^1(\RR)} \le    \Vert u \Vert_{H^1(\RR)} \Vert v \Vert_{H^1(\RR)}$ which yields \eqref{eq:Sob}.
 \end{proof}
By Banach's fixed point theorem and the previous Lemma the solution to the NLS \eqref{eq:GP} exists for short times $t \in [0,\tau]$ and initial states $\psi(0)=\varphi_0 \in H_2^2(\mathbb R)$ and satisfies the variation of constant formula
\begin{equation}
\begin{split}
\label{eq:NLSsol}
\psi(t) &= \phi_{\operatorname{S}}(t,0)\varphi_0-i \int_0^t \phi_{\operatorname{S}}(t,s) F_{\sigma}(\psi(s)) \ ds.
\end{split}
\end{equation}
We now show that such solutions must indeed be global in time.
\begin{lemm}[Global existence of solutions to defoc.\@ NLS]
\label{NLS}
Let $\varphi_0 \in H_2^2(\mathbb R)$ with $u \in W^{1,1}_{\operatorname{pcw}}{(0,T)}$ and potentials $V,V_{\operatorname{con}}$ satisfying a standard condition \emph{(Assumption \ref{ass1})}. The solution to the defocussing \emph{NLS} exists on every compact time interval and satisfies for a recursively defined function
\[ 
C = C(T, \Vert u \Vert_{W^{1,1}_{\operatorname{pcw}}}, \Vert W_{\operatorname{sing}} \Vert_{L^p},\Vert \langle \bullet \rangle^{-2} V_{\operatorname{reg}} \Vert_{L^{\infty}},\Vert \langle \bullet \rangle^{-2} V_{\operatorname{con}} \Vert_{L^{\infty}}) 
\]
such that
\begin{equation}
\label{eq:H1}
 \left( \left\lVert \psi \right\rVert^2_{L^{\infty}((0,T),H^1_1(\mathbb R))}+   \left\lVert  F_{\sigma}(\psi) \overline{ \psi} \right\rVert_{L^{\infty}((0,T),L^1(\mathbb R))} \right)  \le C \left(  \left\lVert \varphi_0 \right\rVert^2_{H^1_1(\mathbb R)}+
 \left\lVert \varphi_0 \right\rVert^4_{H^1(\mathbb R)}\right).
 \end{equation}
Moreover, we have
\begin{equation}
\label{eq:H2}
\left\lVert \psi \right\rVert_{L^{\infty}((0,T),H^2_2(\mathbb R))} \le C \left\lVert \varphi_0 \right\rVert_{H^2_2(\mathbb R)}. 
 \end{equation}
\end{lemm}
\begin{proof}
Let $\tau>0$ be the time of existence as in the fixed-point argument \eqref{eq:NLSsol} and consider any time $t \in (0,\tau).$
Multiplying the NLS by $\overline{\psi'}$ and taking the real part yields after rearranging
\begin{equation}
\begin{split}
\label{eq:bound0}
&\frac{d}{dt} \int_{\mathbb R}\left( \frac{1}{2} \vert \nabla \psi(x,t) \vert^2 + \frac{1}{\sigma+1} F_{\sigma}(\psi(x,t))\overline{\psi(x,t)}+\frac{1}{2} \left(V+V_{\operatorname{TD}}(t)\right)\vert \psi(x,t) \vert^2 \right)\ dx  \\
&=  \int_{\mathbb R} u'(t)V_{\operatorname{con}}(x) \vert \psi(x,t) \vert^2 \ dx \lesssim \vert u'(t) \vert \left\lVert \psi(t) \right\rVert_{H_1(\mathbb R)}^2
\end{split}
\end{equation}
which implies the dependence of $C$ on $\Vert u \Vert_{W^{1,1}_{\operatorname{pcw}}}$
On the other hand, by multiplying the NLS with $(1+\vert \bullet \vert^2) \overline{\psi}$ we obtain for the imaginary part integrated over $\mathbb R$ that 
\begin{equation}
\label{eq:bound0a}
 \frac{d}{dt} \int_{\mathbb R} (1+\vert x \vert^2) \vert \psi(x,t) \vert^2 \ dx  \lesssim \left\lVert \psi(t) \right\rVert_{H_1^1(\mathbb R)}^2.
 \end{equation}
Consider then the energy 
\[ E(t) = \left\lVert \nabla \psi(t) \right\rVert_{L^2(\mathbb R)}^2 + \lambda  \left\lVert \langle \bullet\rangle  \psi(t) \right\rVert_{L^2(\mathbb R)}^2 +\int_{\mathbb R}F_{\sigma}(\psi(x,t)) \overline{\psi(x,t)} \ dx\]
for some $\lambda>0.$
From the above bounds \eqref{eq:bound0} and \eqref{eq:bound0a} we conclude that 
\[ E'(t) \lesssim \frac{d}{dt} \int_{\mathbb R}  \left(V(x)+V_{\operatorname{TD}}(x,t)\right)\vert \psi(x,t) \vert^2 \  dx +(1+\vert u'(s)   \vert ) E(t). \]
By a sufficiently large choice of $\lambda$, we can absorb $\frac{d}{dt} \int_{\mathbb R}  \left(V(x)+V_{\operatorname{TD}}(x,t)\right)\vert \psi(x,t) \vert^2 \  dx$ in the derivative of the energy on the left-hand side, using the control on the norm of the potential. Thus, we have shown that \eqref{eq:H1} holds with 
$
E(t) \lesssim \left\lVert \varphi_0 \right\rVert_{H^1_1(\mathbb R)}^2 +   \left\lVert \varphi_0 \right\rVert_{H^1(\mathbb R)}^4.
$
The bound on the $H_2^2$ norm follows then immediately from applying Gronwall's inequality to \eqref{eq:NLSsol} and the estimates in Lemma \ref{F} together with the boundedness of the $H^1_1$ norm.
\end{proof}

We discuss in the subsequent lemma sufficient conditions for the solution to \eqref{eq:bilSchr} to be in Sobolev spaces of order larger than two. This is because we require slightly higher regularity for the discretization of the Laplacian to converge to the numerical approximation. 
\begin{lemm}[$H^{2+\varepsilon}_2$ regularity:]
\label{eos2} 
Let $\Omega = \RR^d$ or $\Omega$ and consider an initial state $\varphi_0 \in H^{2+\varepsilon}_2(\Omega)$, satisfying Dirichlet or Neumann boundary conditions, for some $\varepsilon \in (0,1)\cup (1,2)$. For potentials $V$,$V_{\operatorname{con}} \in W^{\nu,\gamma}(\Omega)$  with $\gamma \in (2d,\infty)$, for some $\nu=1$, if $\varepsilon \in (0,1)$ and $\nu=2$ if $\varepsilon \in (1,2),$ the solution $\psi$ to \eqref{eq:bilSchr} and \eqref{eq:GP} is in $L^{\infty}((0,T); H^{2+\varepsilon}_2(\Omega))$ and satisfies for a recursively defined function
\[
C = C(T,\Vert u \Vert_{W^{1,1}_{\operatorname{pcw}}((0,T))},\Vert V \Vert_{W^{\nu,\gamma} }, \Vert V_{\operatorname{con}} \Vert_{W^{\nu,\gamma}} , \Vert \phi_{\operatorname{S}}\varphi_0 \Vert_{L^{\infty}((0,T), X^2(\Omega))}),
\]
where $\phi_{\operatorname{S}}$ is the linear Schr\"odinger flow, the estimate
\begin{equation}
 \label{eq:H22}
 \left\lVert \psi \right\rVert_{L^{\infty}((0,T),H^{2+\varepsilon}_2(\Omega))} \le C \left\lVert \varphi_0 \right\rVert_{H^{2+\varepsilon}_2(\Omega)}.
 \end{equation} 

\end{lemm}
\begin{proof}
We start with the linear Schr\"odinger equation and start by establishing $H^{2+\varepsilon}_2$ regularity, for which we use Lemmas \ref{eos} and \ref{NLS}.  
Applying the fractional Laplacian $(-\Delta)^{\varepsilon/2}$ to the linear Schr\"odinger equation yields 
 \begin{equation}
 \label{eq:fracSchr}
 i  (-\Delta)^{\varepsilon/2}\psi'(t)  = (- \Delta)^{1+\varepsilon/2} \psi(t) +  (-\Delta)^{\varepsilon/2}(\psi(t) V(t)). 
 \end{equation}
We then find by rearranging this equation, using 
\begin{equation}
\begin{split}
 \Vert (-\Delta)^{\varepsilon/2} (\psi(t) V(t)) \Vert_{L^2} &\le  \Vert V(t)\psi(t) \Vert_{H^1}  \lesssim_{\Vert V(t) \Vert_{W^{1,\gamma}}} \Vert \psi(t) \Vert_{H^1} \text{ for } \varepsilon \in (0,1) \text{ and } \\
  \Vert (-\Delta)^{\varepsilon/2} (\psi(t) V(t)) \Vert_{L^2} &\le  \Vert V(t)\psi(t) \Vert_{H^2}\lesssim_{\Vert V(t) \Vert_{W^{2,\gamma}}} \Vert \psi(t) \Vert_{H^2}  \text{ for } \varepsilon \in (1,2).
  \end{split}
 \end{equation}
 Here, we used in the first inequality that 
 \begin{equation}
\begin{split}
&\Vert V(t) \nabla \psi(t) \Vert_{L^2} \le \Vert V(t)  \Vert_{L^\infty}\Vert \nabla \psi(t) \Vert_{L^2} \lesssim \Vert V(t) \Vert_{W^{1,\gamma}}\Vert \psi(t) \Vert_{H^1} \\
&\Vert (\nabla V(t)) \psi(t) \Vert_{L^2} \le \Vert V(t) \Vert_{W^{1,\gamma}} \Vert \psi(t) \Vert_{L^{\gamma/(\gamma-1)}} \lesssim \Vert V(t) \Vert_{W^{1,\gamma}} \Vert \psi(t) \Vert_{H^1}
  \end{split}
 \end{equation}
and similar arguments for the second inequality.
Hence, we have from \eqref{eq:fracSchr} that
 \begin{equation}
 \label{eq:auxi}
 \Vert  ( -\Delta)^{1+\varepsilon/2} \psi(t) \Vert_{L^2} \lesssim \Vert  (-\Delta)^{\varepsilon/2} \psi'(t)  \Vert_{L^2}+\Vert \psi \Vert_{H^2}. 
 \end{equation}
 Since  $t\mapsto \Vert \psi(t) \Vert_{H^2}$ is a.e. uniformly bounded on compact time intervals by Lemma \ref{eos}, it suffices to analyze the term $\Vert  (-\Delta)^{\varepsilon/2} \psi'(t)  \Vert_{L^2}$.
For this purpose, we introduce the auxiliary function $\chi(t):=(-\Delta)^{\varepsilon/2} \psi'(t).$ We then have from differentiating \eqref{eq:fracSchr} in time
\begin{equation}
\begin{split} 
\label{eq:H41a}
i \partial_t \chi(t) 
&=  - \Delta \chi(t) +  (-\Delta)^{\varepsilon/2}(\psi'(t) V(t))+ (-\Delta)^{\varepsilon/2}(\psi(t) V'(t)).
 \end{split}
 \end{equation}
Using Sobolev embeddings $H^{2\varepsilon d/\gamma} \hookrightarrow L^{\frac{2}{1-2\varepsilon/\gamma}}$ and $W^{1,\gamma} \hookrightarrow L^{\infty}$,
we have the following product estimates \cite[Lem. $6$]{BM} on Sobolev spaces with $\gamma \in (2d,\infty)$ arbitrary for $\varepsilon \in (0,1)$
 \begin{equation}
\begin{split} 
\Vert (-\Delta)^{\varepsilon/2} (\psi'(t)V(t)) \Vert_{L^2} &\lesssim \Vert \psi'(t) \Vert_{H^{\varepsilon}} \Vert V(t) \Vert_{L^{\infty}}  \\
& \quad + \Vert  \psi'(t) \Vert_{L^{\frac{2}{1-2\varepsilon/\gamma}}} \Vert V(t) \Vert^{\varepsilon}_{W^{1,\gamma}} \Vert V(t) \Vert^{1-\varepsilon}_{L^{\infty}}    \\
&\lesssim \Vert \psi'(t) \Vert_{H^{\varepsilon}} \Vert V(t) \Vert_{L^{\infty}}  \\
& \quad + \Vert  \psi'(t) \Vert_{H^{2\varepsilon d/\gamma}} \Vert V(t) \Vert^{\varepsilon}_{W^{1,\gamma}} \Vert V(t) \Vert^{1-\varepsilon}_{L^{\infty}} \\
&\lesssim \Vert V(t) \Vert_{L^{\infty}} + \Vert V(t) \Vert^{\varepsilon}_{W^{1,\gamma}} \Vert V(t) \Vert^{1-\varepsilon}_{L^{\infty}}\Vert \psi'(t) \Vert_{H^{\varepsilon}}
 \end{split}
 \end{equation}
 and for $\gamma \in (2d,\infty)$ arbitrary for $\varepsilon \in (1,2)$
  \begin{equation}
\begin{split} 
\Vert (-\Delta)^{\varepsilon/2} (\psi'(t)V(t)) \Vert_{L^2} &\lesssim \Vert \psi'(t) \Vert_{H^{\varepsilon}} \Vert V(t) \Vert_{L^{\infty}}  \\
& \quad + \Vert  \psi'(t) \Vert_{L^{\frac{2}{1-\varepsilon/\gamma}}} \Vert V(t) \Vert^{\varepsilon/2}_{W^{2,\gamma}} \Vert V(t) \Vert^{1-\varepsilon/2}_{L^{\infty}}   \\
&\lesssim (\Vert V(t) \Vert_{L^{\infty}} + \Vert V(t) \Vert^{\varepsilon/2}_{W^{2,\gamma}} \Vert V(t) \Vert^{1-\varepsilon/2}_{L^{\infty}})\Vert \psi'(t) \Vert_{H^{\varepsilon}}.
 \end{split}
 \end{equation}
Then, multiplying equation \eqref{eq:H41a} by $\overline{\chi(t)}$, integrating over $\Omega$, and taking the imaginary part yields
 \begin{equation}
\begin{split} 
\label{eq:H420}
\frac{1}{2} \frac{d}{dt}\Vert \chi(t) \Vert_{L^2}^2  &\le \Vert \psi'(t) \Vert^2_{H^{\varepsilon}}  \Vert V(t) \Vert_{L^{\infty}}+ \Vert \psi'(t) \Vert^2_{H^{\varepsilon}} \Vert V(t) \Vert^{\varepsilon}_{W^{\nu,\gamma}} \Vert V(t) \Vert^{1-\varepsilon}_{L^{\infty}} \\
&\quad + \Vert \psi'(t) \Vert_{H^{\varepsilon/2}}  \Vert \psi(t) \Vert_{H^2} \vert u'(t) \vert \Vert V_{\operatorname{con}} \Vert_{W^{\nu,\gamma}}.
 \end{split}
 \end{equation}
 Applying Gronwall's inequality to \eqref{eq:H420} yields then the Gronwall estimate in \eqref{eq:H22} with a constant depending on the specified objects, only.

\medskip

To extend the preceding bounds to the nonlinear Schr\"odinger equations, we estimate the $H^{2+\varepsilon}_2(\Omega)$ norm using the local Lipschitz conditions from Lemma \ref{F}, the boundedness of the $H^1(\Omega)$ norm that we established in Lemma \ref{NLS}, and the boundedness of the linear Schr\"odinger dynamics in $H^{2+\varepsilon}_{2}(\Omega)$ that we just verified, as follows
\begin{equation}
\begin{split}
\label{eq:regularitynon}
\left\lVert \psi(t) \right\rVert_{H^{2+\varepsilon}_{2}} 
&\lesssim \left\lVert \varphi_0 \right\rVert_{H^{2+\varepsilon}_{2}} + \int_0^t \left\lVert F(\psi(s)) \right\rVert_{H^{2+\varepsilon}_{2}} \ ds \\ 
&\lesssim \left\lVert \varphi_0 \right\rVert_{H^{2+\varepsilon}_{2}} + \int_0^t\left\lVert \psi(s) \right\rVert^{\sigma-1}_{H^1} \left\lVert \psi(s) \right\rVert_{H^{2+\varepsilon}_{2}} \ ds 
\end{split}
\end{equation}
which by Gronwall's lemma yields the claim.
\end{proof}

\section{Reduction to bounded domains}
\label{sec:num}
The aim of this section is to show (global) convergence of a (cubic)-discretization of the Schr\"odinger equation to its actual solution with an explicit rate of convergence. 
Sufficiently high regularity of the solution and potentials is also required for the finite-difference scheme to converge to the actual solution (with a fixed rate implying uniform runtime). 

\medskip

In applications however, the potentials and the initial state may a priori not be as regular as necessary for a numerical implementation. However, since smoothness is a local property it can be recovered for instance by mollification.

\medskip
To control the error of this approximation, we decompose the full approximation into several steps. We start by introducing the prerequisites of our finite-difference schemes:

\subsection{Reduction of Schr\"odinger equation}

Our next Lemma shows that it suffices to replace the singular part, $W_{\operatorname{sing}}$, of the potential $V$ and the initial state by an approximation thereof. This implies in particular, that even though we cannot use singular potentials directly in our numerical method, we can always use a smooth approximation thereof and capture all dynamical features with error control:
\begin{lemm}[Perturbation of singular potentials \& initial states]
\label{potlem}
Consider two singular potentials $\widetilde{W}_{\operatorname{sing}},W_{\operatorname{sing}}$, one regular potential $W_{\operatorname{reg}}$, and one control potential $V_{\operatorname{con}}$ satisfying a standard condition \emph{(Assumption \ref{ass1})} with a control $u \in W^{1,1}_{\operatorname{pcw}}(0,T)$. Let $\widetilde{\varphi}_0, \varphi_0 \in H^2_2(\mathbb R^d)$ be two initial conditions. Then the solution to 
\begin{equation*}
\begin{split}
i \partial_t \widetilde{\psi}(x,t) &= \left(-\Delta+(\widetilde{W}_{\operatorname{sing}}+W_{\operatorname{reg}} + V_{\operatorname{TD}}(t))\right)\widetilde{\psi}(x,t)+F(\widetilde{\psi}(x,t)), \quad (x,t) \in \mathbb R^d \times (0,\infty) \\
\widetilde{\psi}(\bullet,0)&=\widetilde{\varphi}_0
\end{split}
\end{equation*}
converges uniformly in time $L^2(\mathbb R^d)$ to the solution of 
\begin{equation*}
\begin{split}
i \partial_t \psi(x,t) &= \left(-\Delta+(W_{\operatorname{sing}}+W_{\operatorname{reg}} + V_{\operatorname{TD}}(t))\right)\psi(x,t)+F(\psi(x,t)), \quad (x,t) \in \mathbb R^d \times (0,\infty) \\
\psi(\bullet,0)&=\varphi_{0}
\end{split}
\end{equation*}
as $\widetilde{W}_{\operatorname{sing}}  \rightarrow_{L^p} W_{\operatorname{sing}}$ and $\widetilde{\varphi_0} \rightarrow_{L^2} \varphi_0.$
Moreover, there exists a recursive function 
\[
C=C(T, \Vert \psi \Vert_{L^{\infty}((0,T), H^2(\RR^d))},\Vert \widetilde{\psi} \Vert_{L^{\infty}((0,T), H^2(\RR^d))})
\]
such that 
\begin{equation}
\label{eq:convinis}
 \left\lVert \psi-\widetilde{\psi} \right\rVert_{L^{\infty}((0,T),L^2(\mathbb R^d))}  \le C\left( \left\lVert W_{\operatorname{sing}}-\widetilde{W}_{\operatorname{sing}} \right\rVert_{L^p(\mathbb R^d)}+\left\lVert \varphi_0-\widetilde{\varphi}_0 \right\rVert_{L^2(\mathbb R^d)}\right).
\end{equation}
In particular, if in addition we have two different control potentials $\widetilde{V}_{\operatorname{con}},V_{\operatorname{con}} \in L^p(\RR^d)$ in either equation, then also for a recursive function 
\[C=C(T, \Vert \psi \Vert_{L^{\infty}((0,T), H^2(\RR^d))},\Vert \widetilde{\psi} \Vert_{L^{\infty}((0,T), H^2(\RR^d))}, \Vert u \Vert_{W^{1,1}_{\operatorname{pcw}}})\]
we find
\begin{equation}
\begin{split}
\label{eq:citeme}
 \left\lVert \psi-\widetilde{\psi} \right\rVert_{L^{\infty}((0,T),L^2(\mathbb R^d))}  \le &C\Bigg( \left\lVert W_{\operatorname{sing}}-\widetilde{W}_{\operatorname{sing}} \right\rVert_{L^p(\mathbb R^d)}+\left\lVert V_{\operatorname{con}}-\widetilde{V}_{\operatorname{con}} \right\rVert_{L^p(\mathbb R^d)}\\
 &+\left\lVert \varphi_0-\widetilde{\varphi}_0 \right\rVert_{L^2(\mathbb R^d)}\Bigg).
 \end{split}
\end{equation}
\end{lemm}
\begin{proof}
By subtracting the two equations from each other and introducing the auxiliary function $\xi:=\psi-\widetilde{\psi}$ one finds that $\xi$ satisfies the perturbed Schr\"odinger equation
\[i\partial_t \xi = (-\Delta +\widetilde{W}_{\operatorname{sing}} +W_{\operatorname{reg}}+ V_{\operatorname{TD}}(t)) \xi +  (W_{\operatorname{sing}}-\widetilde{W}_{\operatorname{sing}}) \psi + F(\psi)-F(\widetilde{\psi})  \]
with initial condition $\xi(0)=\varphi_0-\widetilde{\varphi}_0.$ 
Multiplying the above equation by $\overline{\xi}$, integrating over $\mathbb R^d,$ and taking the imaginary part yields 
\begin{equation*}
\begin{split}
&\frac{1}{2} \frac{d}{dt} \int_{\mathbb R^d} \left\lvert \xi(x,t) \right\rvert^2 \ dx = \Im \left(\int_{\mathbb R^d} \left(\left((W_{\operatorname{sing}}-\widetilde{W}_{\operatorname{sing}}) \psi +F(\psi)-F(\widetilde{\psi})\right)\overline{\xi} \right)(x,t)  \ dx \right) \\
& \le \left\lVert \xi(t) \right\rVert_{L^2(\mathbb R^d)}\left(\left\lVert (W_{\operatorname{sing}}-\widetilde{W}_{\operatorname{sing}}) \psi(t)\right\rVert_{L^2(\mathbb R^d)} + \left\lVert F(\psi)-F(\widetilde{\psi}) \right\rVert_{L^2(\mathbb R^d)} \right).
\end{split}
\end{equation*}
Integrating in time, using that $\psi \in L^{\infty}((0,T),H^2(\mathbb R^d))$, which follows from Lemma \ref{eos} or Lemma \ref{NLS}, and the local Lipschitz condition in Lemma \ref{F}, shows that for $t \in (0,T)$
\begin{equation*}
\begin{split}
\left\lVert \xi(t) \right\rVert^2_{L^2(\mathbb R^d)} 
&\lesssim \left\lVert \xi(0) \right\rVert^2_{L^2(\mathbb R^d)} +   \left\lVert W_{\operatorname{sing}}-\widetilde{W}_{\operatorname{sing}}\right\rVert^2_{L^p(\mathbb R^d)} \int_0^t \left\lVert \psi(s) \right\rVert^2_{H^2(\mathbb R^d)} \ ds \\
&\quad +\int_0^t \left\lVert \xi(s) \right\rVert_{L^2(\mathbb R^d)}^2 \ ds\\
&\lesssim \left\lVert \xi(0) \right\rVert^2_{L^2(\mathbb R^d)} +  t \left\lVert W_{\operatorname{sing}}-\widetilde{W}_{\operatorname{sing}}\right\rVert^2_{L^p(\mathbb R^d)} +\int_0^t \left\lVert \xi(s) \right\rVert_{L^2(\mathbb R^d)}^2 \ ds.
\end{split}
\end{equation*}
Gr\"onwall's inequality shows then that on any finite time interval $(0,T)$
\[\left\lVert \xi \right\rVert_{L^{\infty}((0,T),L^2(\mathbb R^d))}   \lesssim \left\lVert W_{\operatorname{sing}}-\widetilde{W}_{\operatorname{sing}} \right\rVert_{L^p(\mathbb R^d)}+\left\lVert \varphi_0-\widetilde{\varphi}_0 \right\rVert_{L^2(\mathbb R^d)}.\]
Estimate \eqref{eq:citeme} can be obtained analogously.
\end{proof}

\begin{lemm}[Perturbation of regular potentials]
\label{potlem2}
Let $\varphi_0 \in H^2_2(\mathbb R^d)$ be an initial state and consider potentials $V$ and $V_{\operatorname{con}}$ satisfying a standard condition \emph{(Assumption \ref{ass1})}.
Let $\widetilde{\psi}$ be the solution, to the same initial value $\varphi_0$, but for possibly different $\widetilde{W}_{\operatorname{reg}}, \widetilde{V}_{\operatorname{con}}$ satisfying a standard condition as well. Then, the solutions satisfy for a recursive function $C=C(T, \Vert \psi \Vert_{L^{\infty}((0,T), H_2(\RR^d))},\Vert \widetilde{\psi} \Vert_{L^{\infty}((0,T), H_2(\RR^d))})$ 
\begin{equation}
\label{eq:pertpot}
\left\lVert \psi-\widetilde \psi \right\rVert_{L^{\infty}((0,T),L^2(\mathbb R^d))} \le C \int_0^t \left\lVert  \tfrac{W_{\operatorname{reg}}+ V_{\operatorname{TD}}(s)-\widetilde{W}_{\operatorname{reg}}- \widetilde{V}_{\operatorname{TD}}(s) }{\langle \bullet \rangle^{2}}\right\rVert_{L^{\infty}(\mathbb R^d)} \ ds.
\end{equation}
\end{lemm}
\begin{proof}
Let $\psi$ be the solution to the unperturbed problem and $\widetilde{\psi}$ the solution of the Schr\"odinger equation with potentials $\widetilde{V}, \widetilde{V}_{\operatorname{con}},$ then the difference $\xi = \psi-\widetilde{\psi}$ satisfies the equation
\begin{equation*}
\begin{split}
i \partial_t \xi = \left(-\Delta_{\mathbb R^d} +\widetilde{V} + \widetilde{V}_{\operatorname{TD}}(t) \right) \xi +  F(\psi)- F(\widetilde{\psi}) + \left(V+ V_{\operatorname{TD}}(t)-(\widetilde{V} + \widetilde{V}_{\operatorname{TD}}(t)) \right)\psi
\end{split}
\end{equation*}
with initial condition zero.
Multiplying this equation by $\widetilde{\xi}$, integrating over $\mathbb R^d$, taking the imaginary part and applying Gr\"onwall's inequality together with Lemmas \ref{eos}, \ref{NLS} yields together with the local Lipschitz condition in Lemma \ref{F} equation \eqref{eq:pertpot}.
\end{proof}

We now show that the Schr\"odinger equation is Lipschitz continuous with respect to control functions. This Lemma allow us to assume that the controls are locally constant in time.
\begin{lemm}[Perturbation of controls]
\label{pertcon}
Let $\varphi_0$ be an initial state in $H^2_2(\mathbb R^d)$ to either \eqref{eq:bilSchr} or \eqref{eq:GP}. We consider control functions $u,v \in W^{1,1}_{\operatorname{pcw}}{(0,T)}$ and potentials $V,V_{\operatorname{con}}$ satisfying a standard condition \emph{(Assumption \ref{ass1})}. By Lemma \ref{eos} for the Schr\"odinger equation and Lemma \ref{NLS} for the \emph{NLS} there are two solutions $\psi_u, \psi_v$ in $L^{\infty}((0,T),H^2_2(\mathbb R^d))$ for each of the two control functions. Then, the solutions satisfy a Lipschitz condition in terms of a recursively defined function $C=C(T, \Vert  \psi_u \Vert_{L^{\infty}((0,T); H^1_1)}, \Vert  \psi_v \Vert_{L^{\infty}((0,T); H^1_1)}) $
\[ \left\lVert \psi_u-\psi_v \right\rVert_{L^{\infty}((0,T),H_1^1(\mathbb R^d))} \le C \left\lVert u-v \right\rVert_{L^1(0,T)}.\]
\end{lemm}
\begin{proof}
We first subtract the Schr\"odinger equations \eqref{eq:bilSchr} of the two solutions $\psi_u,\psi_v$ from each other and obtain for the function $\xi(t) = \psi_u(t)-\psi_v(t)$
\begin{equation}
\begin{split}
\label{eq:spde}
&(i\partial_t + \Delta - V-u(t)V_{\operatorname{con}})\xi(t)-F(\psi_u)+F(\psi_v)\\
&\qquad +(v(t)-u(t))V_{\operatorname{con}}\psi_v(t)=0 \\
&\xi(0)=0.
\end{split}
\end{equation}
By multiplying the above Schr\"odinger equation with $\overline{(\psi_u-\psi_v)}(1+\left\lvert x \right\rvert^2)$, integrating it over $\mathbb R^d$, and taking the imaginary part, we obtain a bound on the difference of the two solutions in $H_1(\mathbb R^d)$ norm
\begin{equation*}
\begin{split}
&\left\lVert \xi(t) \right\rVert_{H_1(\mathbb R^d)}^2 
\lesssim \int_0^t   \left(\left\lVert \xi(s) \right\rVert_{H^1_1(\mathbb R^d)}^2 +  \left\lvert u(s)-v(s) \right\rvert^2 \left\lVert \psi_v(s) \right\rVert^2_{H_2(\mathbb R^d)}\right) \ ds  \\
& \qquad + \int_0^t \underbrace{\left\lVert F(\psi_u(s))-F(\psi_v(s)) \right\rVert_{H_1(\mathbb R^d)}}_{\le \left(\left\lVert \psi_u(s) \right\rVert^{\sigma-1}_{H^1(\mathbb R^d)} + \left\lVert \psi_v(s) \right\rVert^{\sigma-1}_{H^1(\mathbb R^d)}\right)\left\lVert \xi(s) \right\rVert_{H_1(\mathbb R^d)}}  \left\lVert \xi(s) \right\rVert_{H_1(\mathbb R^d)}+\left\lVert \xi(s) \right\rVert^2_{H_2(\mathbb R^d)}  ds
\end{split}     
\end{equation*}
where we applied Lemma \ref{F} in the last step.
To bound the $H^1(\mathbb R^d)$ norm, we multiply the above equation \eqref{eq:spde} by $\overline{\psi'_u-\psi'_v},$ integrate over $\mathbb R^d$ and take the real part. Then, by the zero-boundedness of the potentials, there is an $\varepsilon>0$ such that by Lemma \ref{F}
\begin{equation*}
\begin{split}
\left\lVert \xi(t) \right\rVert_{H^1(\mathbb R^d)}^2 
&\lesssim \varepsilon \left\lVert \xi(t) \right\rVert_{H^1(\mathbb R^d)}^2+ \left\lVert  \xi(t) \right\rVert_{H_1(\mathbb R^d)}^2 \\  
&\quad +\int_0^t  \left(\left\lvert u(s)-v(s) \right\rvert^2 \left\lVert \psi_u(s) \right\rVert^2_{H_2(\mathbb R^d)}+\left\lVert\xi'(s) \right\rVert^2_{L^2(\mathbb R^d)}\right) \ ds \\
&\quad +\int_0^t \underbrace{\left\lVert F(\psi_u(s))-F(\psi_v(s)) \right\rVert_{L^2(\mathbb R^d)}}_{\le \left(\left\lVert \psi_u(s) \right\rVert^{\sigma-1}_{H^1(\mathbb R^d)} + \left\lVert \psi_v(s) \right\rVert^{\sigma-1}_{H^1(\mathbb R^d)}\right)\left\lVert \xi(s) \right\rVert_{H_1(\mathbb R^d)}}\left\lVert \xi'(s) \right\rVert_{L^2(\mathbb R^d)} \ ds.
\end{split}
\end{equation*}
Adding together the two preceding estimates, using the boundedness of the solution in $H_2^2$ norm as obtained in Lemmas \ref{eos} and \ref{NLS}, an application of Gr\"onwall's inequality on the $H_1^1(\mathbb R^d)$ norm shows $\left\lVert \xi \right\rVert_{L^{\infty}((0,T),H_1^1(\mathbb R^d))} \lesssim \left\lVert u-v \right\rVert_{L^1(0,T)}.$
\end{proof}

Our next theorem, Theorem \ref{theo2}, provides error bounds for reducing the Schr\"odiger equation on $\mathbb R^d$ to a boundary-value problem (BVP) on a bounded domain $\mathcal B_R(0)$.
We emphasize that this theorem therefore implies the reducibility of the PDE on $\mathbb R^d$ to an equation on a finite domain $\mathcal B_R(0)$ with full a priori error control. The theorem assumes initial states in $H^2_2(\mathbb R^d)$ for which the solution to the Schr\"odinger equation is bounded in $H^2_2(\mathbb R^d)$ as well. Since the regular and control potentials grow at most quadratically, \eqref{eq:pertpot} implies that one can replace unbounded potentials by bounded ones with full error control.

\medskip

The estimate \eqref{eq:dirichletestimate} in Theorem \ref{theo2} below implies then the convergence of the solution to the Schr\"odinger equation on a bounded domain to the solution on the full domain by the following argument: 

\medskip 
\begin{lemm}
\label{lemm:auxlemmes}
There is $C>0$ such that for all $r\ge 1$ and $f$ in generalized Sobolev spaces $H^{2,\operatorname{sem}}_2(\mathbb R^d),$ we have
\[ 
\mu^2 \int_{\partial \mathcal B_{r}(0)} f(x) (\nabla_n f)(x) \ dS(x)  \le \frac{C\left\lVert  f \right\rVert^2_{H^{2,\operatorname{sem}}_2}}{r}.
\]
\end{lemm}

\begin{proof}

Let $f,g$ be functions in a generalized Sobolev space $H_1(\mathbb R^d),$ then the co-area formula 
\[ \int_{\mathcal B_R(0)} \vert f(x) \vert^2 dx = \int_0^R \int_{\partial \mathcal B_{r}(0)} \vert f(x) \vert^2 \ dS(x) \ dr \] 
and Remark \ref{rem:genSob} imply that for any $R>0$ there exists an explicit upper bound on the smallest $r>R$ such that $ \int_{\partial \mathcal B_{r}(0)} \vert f(x) \vert \vert \nabla_n f(x)\vert dS(x)<\varepsilon.$  To see this, take a partition $I_i:=[i,i+1]$ with $[1,\infty)=\bigcup_{i \in \mathbb N} I_i$, then it follows by Remark \ref{rem:genSob} that 
\begin{equation}
\begin{split}
\label{eq:infimum}
 \inf_{r\in I_i} \int_{\partial \mathcal B_{r}(0)} \vert f(x)\nabla_n f(x) \vert \ dS(x)
&\le \int_{i}^{i+1} \int_{\partial \mathcal B_{r}(0)} \vert f(x)\nabla_n f(x)\vert \ dS(x) \ dr\\
&\le \Vert f \indic_{ \mathcal B_{i}(0)^c} \Vert_{L^2}\Vert \nabla_n  \indic_{ \mathcal B_{i}(0)^c} \Vert_{L^2} \\
&\le \frac{\left\lVert  f \right\rVert_{H_1}\left\lVert  \nabla_n f(x) \right\rVert_{H^1}}{i}.
 \end{split}
\end{equation}
Since $h(r):=\int_{\partial \mathcal B_{r}(0)} \vert f(x)\nabla_n f(x) \vert \ dS(x)$ is a continuous function of $r$ the above infimum is attained at some radius $r_i \in I_i$ such that 
\[  \inf_{r\in I_i} \int_{\partial \mathcal B_{r}(0)} \vert f(x)\nabla_n f(x) \vert \ dS(x)= \int_{\partial \mathcal B_{r_i}(0)} \vert f(x)\nabla_n f(x) \vert \ dS(x).\]
Consider now any other $s_i \in I_i$

By Green's formula and Remark \ref{rem:genSob} it follows that for an annulus $A_{s_i,r_i}=\mathcal B_{s_i}(0) \backslash \mathcal B_{r_i}(0)$ (if $s_i>r_i$ and vice versa otherwise), $A_i:=\RR^d \setminus \mathcal B_{r_i}(0)$ and normal derivative $\nabla_n f:=\langle \nabla f, n \rangle$, with unit normal $n$,
\begin{equation}
\begin{split}
\label{eq:Green}
\left\lvert \mu^2\int_{\partial A_{s_i,r_i}} f(x) (\nabla_n f)(x) \ dS(x) \right\rvert 
&= \left\lvert  \mu^2\int_{ A_{s_i,r_i}} f(x) (\Delta f)(x) \ dx \right\rvert+ \int_{ A_{s_i,r_i}} \mu^2\vert \nabla f(x) \vert^2\ dx  \\
&\le \mu^2 \left\lVert \indic_{A} f \right\rVert_{L^2} \left\lVert \indic_{ A} \Delta f \right\rVert_{L^2} + \mu^2 \left\lVert \indic_{A} \nabla f \right\rVert_{L^2}^2 \\
& \le 2 \mu^2 \left\lVert \indic_{A} f \right\rVert_{L^2} \left\lVert \indic_{ A} \Delta f \right\rVert_{L^2} + \mu^2 \left\lvert \int_{\partial \mathcal B_{r_i}(0)} f(x) (\nabla_n f)(x) \ dS(x) \right\rvert  \\
& \lesssim 2 \frac{\left\lVert  f \right\rVert^2_{H^{2,\operatorname{sem}}_2}}{i}.
\end{split}
\end{equation}

Here, we bounded $\left\lVert \indic_{A} \nabla f \right\rVert_{L^2}^2$ by Green's formula and the Cauchy-Schwarz inequality
\begin{equation}
\begin{split}
\int_{A} \vert \nabla f(x) \vert^2\ dx &\le \int_{A} \vert (\Delta f)(x) \vert \vert f(x) \vert \ dx + \int_{\partial B_{r_i}(0)}  \vert f(x) \vert \vert  (\nabla_n f)(x) \vert \ dS(x) \\
&\le  \Vert \indic_{A} f \Vert_{H^2} \Vert \indic_{A} f \Vert_{L^2} + \int_{\partial B_{r_i}(0)}  \vert f(x) \vert \vert  (\nabla_n f)(x) \vert \ dS(x) \\
& \lesssim \frac{\Vert f \Vert^2_{H^{2,\operatorname{sem}}_2}}{i}.
\end{split}
\end{equation}
We can now use that for all $r \in I_i:$ we have $\left\lvert 1/r -1/i \right\rvert \le \frac{\vert r - i \vert}{ri} \le \frac{1}{r}$ since $\vert r-i \vert \le 1$ and $i \ge 1.$ 
This implies that for all $s_i \in I_i$ 
\begin{equation}
\begin{split}
\label{eq:si}
\left\lvert \mu^2\int_{\partial A_{s_i,r_i}} f(x) (\nabla_n f)(x) \ dS(x) \right\rvert & \lesssim 2 \left\lVert  f \right\rVert^2_{H^{2,\operatorname{sem}}_2} \left( s_i^{-1}+ \vert i^{-1}-s_i^{-1} \vert \right) \lesssim 4 \left\lVert  f \right\rVert^2_{H^{2,\operatorname{sem}}_2} s_i^{-1}
\end{split}
\end{equation}
and since $i$ was arbitrary, the claim of Lemma \ref{lemm:auxlemmes} follows. In particular, for all $s_i \in I_i$
\begin{equation}
\begin{split}
\left\lvert \mu^2\int_{\partial B_{s_i}(0)} f(x) (\nabla_n f)(x) \ dS(x) \right\rvert & \le 
\left\lvert \mu^2\int_{\partial A_{s_i,r_i}} f(x) (\nabla_n f)(x) \ dS(x) \right\rvert  \\
&\quad + \left\lvert \mu^2\int_{\partial B_{r_i}(0)} f(x) (\nabla_n f)(x) \ dS(x) \right\rvert\\
&\lesssim  \left\lVert  f \right\rVert^2_{H^{2,\operatorname{sem}}_2} s_i^{-1}.
\end{split}
\end{equation}

\end{proof}

\begin{theo}[Reduction to bounded domains]
\label{theo2}
Let $\varphi_0$ be an initial state in $H^2_2(\mathbb R^d)$ to either \eqref{eq:bilSchr}, \eqref{eq:GP} ($h=1$) or the semiclassical Schr\"odinger equation \eqref{eq:semicl} ($h>0$) and consider potentials $V$ and $V_{\operatorname{con}}$ satisfying a standard condition \emph{(Assumption \ref{ass1})}. Then for control functions $u \in W^{1,1}_{\operatorname{pcw}}(0,T)$ and any compact time interval the solution of the corresponding Schr\"odinger equation with solution $\psi$ can be approximated by $\psi^{D,N}$, the solution to an auxiliary \emph{BVP}, as introduced in Lemma \ref{eos} or Lemma \ref{eos2}, on $\mathcal B_R(0)$ with $R>0$, where $\chi_{\mathcal B_R(0)}$ is a suitably chosen smooth cut-off function supported in $\mathcal B_R(0)$. In particular, the difference $\xi = \psi-\psi^{D,N}$ of the two solutions then satisfies for a recursive function, in terms of $C_{\mu^2/2}$ as in \eqref{eq:zero-bd}, 
\[C = C(T,\Vert u \Vert_{W^{1,1}_{\operatorname{pcw}}((0,T))},\Vert W_{\operatorname{sing}} \Vert_{L^p}, \Vert \langle \bullet \rangle^{-2} W_{\operatorname{reg}} \Vert_{L^{\infty}}, \Vert \langle \bullet \rangle^{-2} V_{\operatorname{con}} \Vert_{L^{\infty}},C_{\mu^2/2})\]
\begin{equation}
\begin{split}
\label{eq:dirichletestimate}
\left\lVert \xi  \right\rVert^2_{L^{\infty}((0,T),L^2(\RR^d))} \le C(T)\Bigg(& \sup_{t \in (0,T)}\left\lvert \mu^2 \int_{\partial \mathcal B_R(0)} ((\nabla_n \xi)\overline{\xi})(x,t)  \ dS(x) \right\rvert +\frac{\Vert \varphi_0 \Vert_{H^{2,\operatorname{sem}}_2(\RR^d)}}{R^2} \Bigg).
\end{split}
\end{equation}
In particular, the approximation error in \eqref{eq:dirichletestimate} decays explicitly like $\mathcal O(R^{-1})$, using that the right hand side is controlled by Lemma \ref{lemm:auxlemmes}. 
\end{theo}
\begin{proof}
We first use \eqref{eq:convinis} to reduce the initial state to an initial state in $\mathcal B_R(0)$. To justify this, just recall the $\mathcal O(R^{-2})$ decay in the $L^2$ sense, that we obtain from the initial state being bounded in $H^2_2$ by \eqref{eq:outside}. 
To verify \eqref{eq:dirichletestimate}, we may separate the dynamics of the solution $\psi$ outside the ball $\mathcal B_R(0)$, where the solution to the BVP vanishes anyway, from inside the ball. This can be controlled by combining the estimate \eqref{eq:outside} with the respective estimate from Lemmas \ref{eos} and \ref{eos2} such that
\[ \sup_{t \in [0,T]}\Vert \indic_{\mathcal B_R(0)^c} \varphi(t) \Vert_{H^{2,\operatorname{sem}}_2} \le R^{-2} \sup_{t \in [0,T]}\Vert  \varphi(t) \Vert_{H^{2,\operatorname{sem}}_2} \le R^{-2} C(T) \Vert \varphi_0 \Vert_{H^{2,\operatorname{sem}}_2}.\]
Consider the solution $\psi^{\operatorname{D,N}}$ to the BVP for the Schr\"odinger equation on $\mathcal B_R(0).$ Taking the difference of the true solution and the solution of the BVP yields on $\mathcal B_R(0)$ for $\xi:= \psi- \psi^{\operatorname{D,N}}$ the differential equation 
\begin{equation*}
\begin{split}
 i \partial_t\xi (t) &= \left(-\mu^2\Delta +V+ V_{\operatorname{TD}}(t) \right) \xi(t) + F_{\sigma}(\psi(t))-F_{\sigma}(\psi^{\operatorname{D,N}}(t)) \operatorname{ and }\xi(0) = 0.
 \end{split}
 \end{equation*}
Then multiplying this equation by $\overline{\xi}$, integrating over $\mathcal B_R(0)$, and taking the imaginary part shows that 
\begin{equation*}
\begin{split}
\Re \int_{\mathcal B_R(0)}\partial_t\xi (x,t) \overline{\xi(x,t)} \ dx 
= \Im \left( \int_{\mathcal B_R(0)} \left( -\mu^2\Delta\xi(x,t)+ F_{\sigma}(\psi(x,t))-F_{\sigma}(\psi^{\operatorname{D,N}}(x,t)) \right) \overline{\xi(x,t)} \ dx \right).   
\end{split}
\end{equation*}
Using Green's formula and Lemmas \ref{eos}, \ref{F}, and \ref{eos2} we deduce that 
\begin{equation*}
\begin{split}
\frac{d}{dt} \int_{\mathcal B_R(0)} \left\lvert \xi (x,t) \right\rvert^2 \ dx &\lesssim  \mu^2 \left\lvert \int_{\partial \mathcal B_R(0)} (\nabla_n \xi)(x,t) \overline{\xi(x,t)}  \ dS(x) \right\rvert \\
&\quad + \int_{ \mathcal B_R(0)} \left\lvert F_{\sigma}(\psi(x,t))-F_{\sigma}(\psi^{\operatorname{D,N}}(x,t)) \right\rvert  \left\lvert  \overline{\xi(x,t)} \right\rvert \ dx \\
&\lesssim\left\lvert \int_{\partial \mathcal B_R(0)} \mu^2(\nabla_n \xi)(x,t) \overline{\xi(x,t)}  \ dS(x) \right\rvert+ \left\lVert \xi(t) \right\rVert^2_{L^2(\mathcal B_R(0))}.
\end{split}
\end{equation*}
Thus, Gr\"onwall's inequality implies 
\[\left\lVert \xi (t) \right\rVert^2_{L^2(\mathcal B_R(0))} \lesssim \left\lvert \int_{\partial \mathcal B_R(0)} \mu^2(\nabla_n \xi)(x,t) \overline{\xi(x,t)}  \ dS(x) \right\rvert.\]
\end{proof}

\section{Discretization of the Schr\"odinger equation}
We now discuss numerical methods for the analysis of the Schr\"odinger equation. To do so, we make -without loss of generality- the following simplifying assumption:
\begin{Assumption}[Locally constant controls and global methods]
In this section, we assume that the dynamics is restricted to some cube $\Omega:=[-R,R]^d$ with Dirichlet boundary conditions, since we already showed in Theorem \ref{theo2} that it suffices to analyze the dynamics on a compact domain with Dirichlet or Neumann boundary conditions.
Moreover, Lemma \ref{pertcon} allows us then to make, with full error control, the following simplifying assumption on the controls:
By choosing the time-step in our methods sufficiently small, we assume that the control functions $u$ are constant in every time step $\tau.$
\label{pcwconst}
\end{Assumption}

\subsection{Discretization}

We start by defining a spatial discretization of $L^1_{\operatorname{loc}}(\mathbb R^d)$ functions which allows us to study low-regularity function by handling a sequence of countably many values.

\begin{defi}[Cubic discretization]
\label{cubdisc}
Consider a lattice of side length $h$ and lattice points $(x_j)_{j \in \mathbb Z^d}$ and a family of cubes $Q_{x_j}:=\times_{i=1}^d  [x_j^i - h/2, x_j^i + h/2)$ with $j \in \mathbb Z^d$ that form a disjoint decomposition of $\mathbb R^d$ up to a set of measure zero. The \emph{cubic approximation} of a function $f \in L^1_{\operatorname{loc}}(\mathbb R^d,\mathbb C)$ is defined by
\[
 f_{Q}(x):= \sum_{j \in \mathbb Z^d}  \fint_{Q_{x_j}} f(s) \ ds \  \indic_{Q_{x_j}}(x),
 \]
where 
\[
\fint_{Q_{x_j}} f(s) \ ds = \frac{1}{\mathrm{Vol}(Q_{x_j})}\int_{Q_{x_j}} f(s) \ ds.
\]
We define the \emph{standard decomposition}, with inverse side length (grid size) $m \in \mathbb N$, to be the uniform decomposition of $\mathbb R^d$ into cubes
$\times_{i=1}^d [n_i+\tfrac{k_i}{m},n_i+\tfrac{k_i+1}{m})$ with mid-points $x^i = n_ i + \frac{2k_i + 1 }{2m}$ for $n \in \mathbb Z^d$ and $k \in \left\{0,..,m-1 \right\}^d.$
\end{defi}

Note that we need to numerically compute $\fint_{Q_{x_j}} f(s) \ ds$ for $j \in \mathbb Z^d$. Before we embark on the numerical approximation, the reader unfamiliar with the concept of Halton sequences may want to review this material. An excellent reference is \cite{N92} (see p. 29 for definition). 

In order to numerically approximate the integral so we let $\rho_{x_j}$ denote the canonical linear map that maps $[0,1)^d$ to $Q_{x_j}$  and $S =\{t_k\}_{k\in\mathbb{N}}$, where $t_k \in [0,1]^d$ is a Halton sequence (see \cite{N92} p. 29 for definition) in the pairwise relatively prime 
bases $b_1,\hdots, b_d$ (note that the particular choice of the $b_j$s is not important). 

\begin{defi}[Numerical cubic discretization]\label{def:numerical_disc}
Consider the setup in Definition \ref{cubdisc} and define for
$f \in L^1_{\operatorname{loc}}(\Omega,\mathbb C)$ 
\begin{equation}\label{eq:f}
f_{Q,N}(x):= \sum_{x_j \in \Omega} \frac{1}{N}\sum_{k=1}^Nf(\rho_{x_j}(t_k))  \indic_{Q_{x_j}}(x).
\end{equation}
Also, for any any function $f_{Q,N}$ of the form \eqref{eq:f} we say that $f^M_{Q,N}$ is an $M$ approximation to $f_{Q,N}$ if 
\[
f^M_{Q,N} = \sum_{x_j \in \Omega} \frac{1}{N}\sum_{k=1}^Nf^M(\rho_{x_j}(t_k))  \indic_{Q_{x_j}}(x)
\]
and
\[
|f^M(\rho_{x_j}(t_k)) - f(\rho_{x_j}(t_k))| \leq 2^{-M} \qquad \forall \, x_j \in \Omega, k \leq N.
\]
\end{defi}



\begin{defi}
Let $\{t_1,\hdots t_N\}$ be a sequence in $[0,1]^d$. Then we define the \emph{star discrepancy} of 
$\{t_1,\hdots t_N\}$  to be 
$$
D^*_N(\{t_1,\hdots t_N\}) = \sup_{K \in \mathcal{K}}\left|\frac{1}{N}\sum_{k=1}^N\chi_{K}(t_k) - \nu(K)\right|,
$$  
where 
$\mathcal{K}$ denotes the family of all subsets of $[0,1]^d$ of the form $\prod_{k=1}^d[0,b_k),$ $\chi_K$ denotes the characteristic function on $K$, $b_k \in (0,1]$ and $\nu$ denotes the Lebesgue measure. 
\end{defi}

\begin{theo}[\cite{N92}]\label{disc_bound}
If $\{t_k\}_{k\in\mathbb{N}}$ is the Halton sequence in $[0,1]^d$ in the pairwise relatively prime 
bases $b_1,\hdots , b_d$, then 
\begin{equation}\label{star}
D^*_N(\{t_1,\hdots t_N\}) \leq \frac{d}{N} + \frac{1}{N}\prod_{k=1}^d\left(\frac{b_k-1}{2\log(b_k)}\log(N) + \frac{b_k+1}{2}\right) \qquad N \in \mathbb{N}.
\end{equation}
\end{theo}

For a proof of this theorem see \cite{N92}, p. 29. Note that as the right-hand side of (\ref{star}) is somewhat cumbersome to work with, it is convenient to define the following constant.

\begin{defi}\label{C*}
Define $C^*(b_1,\hdots,b_d)$ to be the smallest integer such that for all $N \in \mathbb{N}$
$$
\frac{d}{N} + \frac{1}{N}\prod_{k=1}^d\left(\frac{b_k-1}{2\log(b_k)}\log(N) + \frac{b_k+1}{2}\right) \leq C^*(b_1,\hdots,b_d)\frac{\log(N)^d}{N}
$$
where $b_1,\dots,b_d$ are as in Theorem \ref{disc_bound}.
\end{defi}

\begin{prop}\label{prop:bound_disc_int}
For $f \in L^{\infty}(\mathbb R^d,\mathbb C)$ then
\[
\| f_{Q} - f_{Q,N} \|_{L^{\infty}(\mathbb{R}^d)} \leq \left(\sup_j \mathrm{TV}\left(f \circ \rho_{x_j} \right)\right) C^*(b_1,\hdots,b_d)\frac{\log(N)^d}{N}
 \]
 and
 \[
 \| f_{Q} - f^M_{Q,N} \|_{L^{\infty}(\mathbb{R}^d)} \leq \left(\sup_j \mathrm{TV}\left(f \circ \rho_{x_j} \right)\right) C^*(b_1,\hdots,b_d)\frac{\log(N)^d}{N} + N2^{-M}.
\]
For $g \in L^1_{\operatorname{loc}}(\mathbb R^d,\mathbb C) \cap L^2(\mathbb R^d,\mathbb C)$ then
\[
\| g_{Q} - g_{Q,N} \|_{L^2(\mathbb{R}^d)} \leq \left(\sum_j \left(\mathrm{TV}\left(f \circ \rho_{x_j} \right) \right)^2\right)^{\frac{1}{2}} C^*(b_1,\hdots,b_d)\frac{\log(N)^d}{N}
\]
and
\[
\| g_{Q} - g^M_{Q,N} \|_{L^2(\mathbb{R}^d)} \leq \left(\sum_j \left(\mathrm{TV}\left(f \circ \rho_{x_j} \right) \right)^2\right)^{\frac{1}{2}} C^*(b_1,\hdots,b_d)\frac{\log(N)^d}{N} + N2^{-M}|\{x_j \, \vert \, x_j \in \Omega\}| 
\]
\end{prop}

\begin{proof}
Note that, by the multi-dimensional Koksma--Hlawka inequality (Theorem 2.11 in \cite{N92}) it follows that 
\begin{equation*}
\begin{split}
\left|  I^h_{x_j}(f) - \fint_{Q_{x_j}} f(s) \ ds  \right| &\leq \mathrm{TV}\left(f \circ \rho_{x_j} \right) D^*_N(t_1,\hdots,t_N).\\
&\leq \mathrm{TV}\left(f \circ \rho_{x_j} \right) C^*(b_1,\hdots,b_d)\frac{\log(N)^d}{N}.
\end{split}
\end{equation*}
\end{proof}

To approximate the Laplacian of the Schr\"odinger operator we use the finite-difference approximation of the derivative: 
\begin{defi}[Derivative discretization]
\label{def:derdisc}
Let $(\tau^i_hf)(t):=f(t-h\widehat{e}_i)$ be the translation by $h$ and $\delta_{h}^i:=(\tau_{-{h}}^i-\tau_{h}^i)/(2h_i)$ the discretized symmetric derivative in direction $i$ with step size $h>0.$ Then, we can define the discretized Laplacian $\Delta^{h}:=\sum_{i=1}^d \left(\delta_{h}^i\right)^2$ and discretized gradient $\nabla^{h}:=(\delta_{h}^i)_{i}.$
\end{defi}
Let $f,g \in L^2(\RR^d)$ then $\langle \tau^i_x f,g \rangle  = \langle f,\tau^i_{-x}g \rangle$ and thus $\left(\delta_{h}^i\right)^*= -\delta_h^i.$
Moreover, the following version of the product rule holds
$
(\delta_hfg)(x) = f(x+h)(\delta_hg)(x)+g(x-h)(\delta_hf)(x).
$
We record elementary convergence properties of the finite-difference scheme in the following proposition:
\begin{prop}
\label{convrate}
Let $n \in \mathbb{Z}_+$, $k \in \{1,.., d\}$ and $f \in W^{n,p}(\mathbb R^d)$. Then, it follows that for $p \in [1,\infty)$ \[  \left\lVert (\delta_h^k)^n f_Q-\partial_k^n f\right\rVert_{L^p(\mathbb R^d)} = o(1)\text{ as }h \downarrow 0, \] 
and for $f \in W^{n+1,p}(\mathbb R^d)$ and $p \in [1,\infty]$, there is an explicit constant $C>0$, independent of $f$ and $h$, such that 
\[
\left\lVert (\delta_h^k)^n f_Q-\partial_k^n f\right\rVert_{L^p(\mathbb R^d)} \le  C \Vert f \Vert_{W^{n+1,p}(\RR^d)} h.
\] 
In particular, $\Vert (\delta_h^k)^n f \Vert_{L^p} \le C \Vert f \Vert_{W^{n,p}(\RR^d)}.$
Moreover, for $p \in (1,\infty)$, $f \in W^{n+\varepsilon,p}(\mathbb R^d)$, and some $\varepsilon \in (0,1]$ the convergence satisfies the rate $\left\lVert (\delta_h^k)^n f_Q-\partial_k^n f\right\rVert_{L^{p}(\mathbb R^d)} = \mathcal O(\Vert f \Vert_{W^{n+\varepsilon,p}(\mathbb R^d)} h^{\varepsilon}).$
\end{prop}

\begin{proof}
The proof is stated in the appendix in Subsection \ref{sec:CD}. 
\end{proof}

\subsection{The linear Schr\"odinger equation}
\label{sec:LSE}
To complete the reduction of the PDE to a discretized finite-difference equation for the linear Schr\"odinger equation, we study the linear Schr\"odinger equation with discretized Laplacian, introduced in Definition \ref{def:derdisc} 
\begin{equation}
\begin{split}
\label{eq:Schrh}
i \partial_t \psi^h(t) &= \left(-\Delta^h +V + V_{\operatorname{TD}}(t) \right)\psi^h(t),\quad \psi^h(0) = \varphi_0.
\end{split}
\end{equation}

\begin{lemm}
\label{H22discrete}
The Schr\"odinger equation \eqref{eq:Schrh} for potentials $V,V_{\operatorname{con}} \in W^{2,\infty}(\Omega)$, a control function $u \in L^1(0,T)$, has a unique solution in $H^2(\Omega) \cap H^{1}_0(\Omega) $ such that for some recursive function 
\[
C=C(T,\Vert u \Vert_{L^1}, \Vert V\Vert_{W^{2,\infty}(\Omega)},\Vert V_{\operatorname{con}} \Vert_{W^{2,\infty}(\Omega)})
\]
we have that
\[\left\lVert \psi^h \right\rVert_{L^{\infty}((0,T),H^{2}(\Omega))}  \le C \left\lVert \varphi_0 \right\rVert_{H^{2}(\Omega)}.\]
\end{lemm}
\begin{proof} 
The free Schr\"odinger operator defines a linear operator 
\[H^h:=-\Delta^h + V \in \mathcal L\left(  H^2(\Omega) \cap H^1_0(\Omega)\right).\]

Thus, the linear Schr\"odinger equation $ i \partial_t \psi^h(t) = \left(H^h + V_{\operatorname{TD}}(t)\right) \psi^h(t),$ with $\psi(0) = \varphi_0 \in H^2(\Omega) \cap H^1_0(\Omega)$
 has a unique solution in $H^2(\Omega) \cap H^1_0(\Omega)$ and the flow is bounded in $H^2(\Omega)$ as the variation of constant formula, which implies
 \begin{equation*}
 \begin{split}
\left\lVert \psi^h(t) \right\rVert_{H^2(\Omega)} \le \left\lVert \psi^h(0) \right\rVert_{H^2(\Omega)}+  \int_0^t  \left\lVert V_{\operatorname{TD}}(s)\psi^h(s) \right\rVert_{H^2(\Omega)} \ ds,
 \end{split}
 \end{equation*}
and Gronwall's lemma show.
\end{proof}
The next Lemma allows us to relate the dynamics defined by the linear Schr\"odinger equation with discrete Laplacian \eqref{eq:Schrh} to the dynamics of the actual linear Schr\"odinger equation \eqref{eq:bilSchr} with fixed error rate. 
We therefore consider norms 
\[
\Vert \psi\Vert_{H^1_h}:=\sqrt{\Vert \psi\Vert^2_{L^2} + \Vert\nabla^h \psi\Vert^2_{L^2}}\text{ and }\Vert \psi\Vert_{H^2_h}:=\sqrt{\Vert \psi\Vert^2_{L^{2}} + \Vert \Delta^h \psi\Vert^2_{L^{2}}}
\] 
and analyze convergence of the solution to a fully discretized equation. We record that summation by parts implies that 
\[ 
\Vert \nabla^h f_Q \Vert^2_{L^2}  \le \Vert f_Q \Vert_{L^2} \Vert \Delta^h f_Q \Vert_{L^2}.
\]

\begin{lemm}
\label{dynlem}
\underline{$L^2$-convergence:} For an initial state $\varphi_0 \in H^{2+\varepsilon}_2(\Omega) \cap H^1_0(\Omega)$ with $\varepsilon \in (0,1]$, a control function $u \in W^{1,1}_{\operatorname{pcw}}(0,T),$ and potentials $V,V_{\operatorname{con}} \in W^{2,\infty}(\Omega)$, the difference of the solution $\psi$ to the linear Schr\"odinger equation
\begin{equation*}
\begin{split}
i \partial_t \psi(x,t) &= (-\Delta+V +V_{\operatorname{TD}}(t))\psi(t), \quad \psi(0)=\varphi_0
\end{split}
\end{equation*}
and the solution $\psi^h_{\infty}$ to the discretized Schr\"odinger equation where $U_{\infty}(t):=V_Q+u(t) \left(V_{\operatorname{con}}\right)_Q$
\begin{equation}
\begin{split}
\label{eq:psiinftyl2}
i \partial_t \psi^h_{\infty}(t) &= (-\Delta^h +U_{\infty}(t))\psi^h_{\infty}(t), \quad \varphi(0) = (\varphi_0)_Q
\end{split}
\end{equation}
satisfy an error bound in terms of some recursively defined function
\[ C= C(T,\Vert u \Vert_{L^1}, \Vert V\Vert_{W^{2,\infty}},\Vert V_{\operatorname{con}} \Vert_{W^{2,\infty}},\Vert \varphi_0 \Vert_{H^{2+\varepsilon}_2}) \]
such that 
\[ \left\lVert \psi-\psi^h_{\infty} \right\rVert_{L^{\infty}((0,T),L^2)}  \le C(\Vert u \Vert_{L^1}, \Vert V\Vert_{W^{2,\infty}},\Vert V_{\operatorname{con}} \Vert_{W^{2,\infty}},\Vert \varphi_0 \Vert_{H^{2+\varepsilon}_2}) h^{\varepsilon}.\]

\medskip

\underline{$H^1_h$-convergence:} For an initial state $\varphi_0 \in H^{3+\varepsilon}_{2}(\Omega) \cap H^1_0(\Omega)$ with $\varepsilon \in (0,1]$, a control function $u \in W^{1,1}_{\operatorname{pcw}}(0,T)$ and potentials $V,V_{\operatorname{con}} \in W^{3,\infty}(\Omega)$, the difference of the solution $\psi$ to the linear Schr\"odinger equation
\begin{equation*}
\begin{split}
i \partial_t \psi(x,t) &= (-\Delta+V +V_{\operatorname{TD}}(t))\psi(t), \quad \psi(0)=\varphi_0
\end{split}
\end{equation*}
and the solution $\psi^h_{\infty}$ to the discretized Schr\"odinger equation where $U_{\infty}(t):=V_Q+u(t) \left(V_{\operatorname{con}}\right)_Q$
\begin{equation}
\begin{split}
\label{eq:psiinfty}
i \partial_t \psi^h_{\infty}(t) &= (-\Delta^h +U_{\infty}(t))\psi^h_{\infty}(t), \quad \varphi(0) = (\varphi_0)_Q
\end{split}
\end{equation}
satisfy an error bound in terms of some recursively defined function
\[ C=C(\Vert u \Vert_{W^{1,1}_{\operatorname{pcw}}}, \Vert V\Vert_{W^{3,\infty}},\Vert V_{\operatorname{con}} \Vert_{W^{3,\infty}},\Vert \varphi_0 \Vert_{H^{3+\varepsilon}_{2}}) \]
such that 
\[
 \left\lVert \psi-\psi^h_{\infty} \right\rVert_{L^{\infty}((0,T),H^1_h)}  \le C(\Vert u \Vert_{W^{1,1}_{\operatorname{pcw}}}, \Vert V\Vert_{W^{3,\infty}},\Vert V_{\operatorname{con}} \Vert_{W^{3,\infty}},\Vert \varphi_0 \Vert_{H^{3+\varepsilon}_{2}}) h^{\varepsilon}.
\]
\end{lemm}

\begin{proof}
We start by first replacing the Laplacian with its discretization in the above equation and then proceed by replacing the remaining quantities. Thus, we first study the auxiliary equation
\begin{equation*}
\begin{split}
i \partial_t \psi^h(t) &= (-\Delta^h +V +V_{\operatorname{TD}}(t))\psi^h(t), \quad \psi(0) = \varphi_0.
\end{split}
\end{equation*}
Subtracting the two solutions $\psi$ and $\psi^h$ from each other and introducing the auxiliary function $\xi^h:=\psi-\psi^h$ shows that 
\begin{equation*}
\begin{split}
\label{eq:diffeqlin}
i \partial_t \xi^h &= \left(-\Delta^h + V+ V_{\operatorname{TD}}(t) \right) \xi^h + \left( \Delta^h-\Delta\right)\psi, \ \xi^h(0)= 0. 
\end{split}
\end{equation*}
Multiplying by $\xi^h$, integrating over $\Omega$, and taking the imaginary part implies the claim by Gronwall's lemma, Lemma \ref{eos2}, and Proposition \ref{convrate}. 

\medskip
For the $H^1_h$ norm, we find analogously by applying the discretized gradient
\begin{equation*}
\begin{split}
\label{eq:diffeq}
i \partial_t  \nabla^h\xi^h &= -\nabla^h\Delta^h \xi^h+\nabla^h \left(\left(V+ V_{\operatorname{TD}}(t) \right) \xi^h\right) + \left( \Delta^h-\Delta\right)\nabla^h\psi, \\
 \xi^h(0)&= 0. 
\end{split}
\end{equation*}
Multiplying by $\nabla^h \overline{\xi^h}$, integrating over $\Omega$, and taking the imaginary part implies
\[ \frac{1}{2} \frac{d}{dt}  \left\Vert \nabla^h \xi^h(x) \right\Vert^2_{L^2} \ dx \lesssim (1+ \Vert V+ V_{\operatorname{TD}}(t) \Vert_{W^{1,\infty}}^2) \Vert \xi^h \Vert_{H^1_h}^2 + \left\lVert  \left( \Delta^h-\Delta\right)\nabla^h\psi \right\Vert_{L^2}^2 \]  which yields the claim by Gronwall's lemma, Lemma \ref{eos2}, and Proposition \ref{convrate}. 

Finally, let $\nu^h:=\psi^h-\psi^h_{\infty}$ then we have for the differences in the $L^2$ and $H^1_h$ norm
\begin{equation*}
\begin{split}
\left\Vert \nu^h(t) \right\Vert_{L^2} &\le \left\Vert \nu^h(0) \right\Vert_{L^2} + \int_0^t \Big(\left\Vert (V+V_{\operatorname{TD}}(s)-U_{\infty}(s)) \right\rVert_{L^{\infty}} \left\Vert \psi^{h}(s) \right\Vert_{L^2} \\
& \qquad \qquad \qquad \qquad \qquad + \left\Vert U_{\infty}(s)\right\rVert_{L^{\infty}} \left\Vert \nu^h(s) \right\Vert_{L^2}  \Big )\ ds
 \end{split}
\end{equation*}
and similarly in $H^1_h$ norm
\begin{equation*}
\begin{split}
\left\Vert \nu^h(t) \right\Vert_{H^1_h} &\le \left\Vert \nu^h(0) \right\Vert_{H^1_h} + \int_0^t \Big(\left\Vert (V+V_{\operatorname{TD}}(s)-U_{\infty}(s)) \right\rVert_{W^{1,\infty}_h} \left\Vert \psi^{h}(s) \right\Vert_{H^1_h} \\
& \qquad \qquad \qquad \qquad \qquad + \left\Vert U_{\infty}(s)\right\rVert_{W^{1,\infty}_h} \left\Vert \nu^h(s) \right\Vert_{H^1_h}  \Big )\ ds
 \end{split}
\end{equation*}
respectively. By Gronwall's lemma this implies the claim, and in particular also the recursivity of $C$, as we have 
$
  \left\Vert \nu^h(0) \right\Vert_{L^2},\left\Vert \nu^h(0) \right\Vert_{H^1_h}= \mathcal O(h^{\varepsilon})
 $ 
and
\[  \left\Vert (V+V_{\operatorname{TD}}(s)-U_{\infty}(s)) \right\rVert_{L^{\infty}},\left\Vert (V+V_{\operatorname{TD}}(s)-U_{\infty}(s)) \right\rVert_{W^{1,\infty}_h} = \mathcal O(h)\]
by the assumptions on the initial states and potentials, again using Proposition \ref{convrate}.
\end{proof}

 We now analyze the convergence of a Crank-Nicholson discretization scheme with time step $\tau:=t_{k+1}-t_k$ for the linear Schr\"odinger equation \eqref{eq:psiinfty}:
For the linear Schr\"odinger equation we use an (implicit) Crank-Nicholson scheme
\begin{equation}
\begin{split}
\label{eq:discrule}
i \left(\frac{\psi^{h}_{\operatorname{CN}}(t_{k+1})-\psi^{h}_{\operatorname{CN}}(t_{k})}{\tau} \right) &= \frac{1}{2} \left(-\Delta^h + V_Q  \right) \left(\psi^{h}_{\operatorname{CN}}(t_{k+1})+\psi^{h}_{\operatorname{CN}}(t_{k}) \right), \\
 \psi^{h}_{\operatorname{CN}}(0) &:= (\varphi_0)_Q.
\end{split}
\end{equation}

\begin{prop}
\label{prop:Schr}
Consider the solution to the linear time-independent Schr\"odinger equation
\begin{equation}
\begin{split}
\label{eq:trueone}
i \partial_t \psi_{\infty}^h(t) &= \left(-\Delta^h +V_Q\right)\psi^h_{\infty}(t), \quad \psi^h_{\infty}(0) = (\varphi_0)_Q,
\end{split}
\end{equation}
with bounded potentials $V \in L^{\infty}(\Omega)$ and initial datum $\varphi_0 \in H^{2+\varepsilon}_2(\Omega) \cap H^1_0(\Omega).$ The solution $\psi^{h}_{\operatorname{CN}}$ obtained from the Crank-Nicholson method \eqref{eq:discrule} preserves the $L^2$ norm, is $H^2_h$ bounded, and convergent in both $L^2$ and $H^2_h$, such that for some recursively defined function $C\equiv C(\Vert V \Vert_{L^{\infty}}, \Vert \varphi_0\Vert_{H^2_h} )>0$ 
and any $k$
\begin{equation}
\begin{split}
\label{eq:estimateII}
 \left\lVert \psi_{\infty}^h (t_{k})-\psi_{\operatorname{CN}}^h (t_{k}) \right\rVert_{L^2(\Omega)} &\le  \left\lVert \psi_{\infty}^h (t_{0})-\psi_{\operatorname{CN}}^h (t_{0}) \right\rVert_{L^2(\Omega)} + \tfrac{kC\tau^3}{h^4} \text{ and }\\
 \left\lVert \psi_{\infty}^h (t_{k})-\psi_{\operatorname{CN}}^h (t_{k}) \right\rVert_{H^2_h(\Omega)} &\le  \left\lVert \psi_{\infty}^h (t_{0})-\psi_{\operatorname{CN}}^h (t_{0}) \right\rVert_{H^2_h(\Omega)} + \tfrac{kC \tau^3}{h^6}.
\end{split}
\end{equation}
In particular, let $\tau= o(h^{4/3})$ or $\tau=o(h^2)$ respectively, then the above scheme is convergent.
\end{prop}

\begin{proof}
$L^2$-norm preservation follows immediately from the Cayley transform representation:  
That is, in terms of the self-adjoint operator $H^h_{\infty}:= -\Delta^h+ V_Q,$ the Crank-Nicholson method reads 
\begin{equation}\label{eq:the_C}
\begin{split}
\psi_{\operatorname{CN}}^h (t_{k+1}) &= \left(1-\frac{i\tau}{2} H^h_{\infty}\right)\left(1+\frac{i\tau}{2} H^h_{\infty}\right)^{-1} \psi_{\operatorname{CN}}^h(t_k)=:\mathcal C^{\tau}_{h} \psi_{\operatorname{CN}}^h(t_k).
\end{split}
\end{equation}
Thus, it follows that
\[ \left\lVert H^h_{\infty} \psi_{\operatorname{CN}}^h (t_{k+1}) \right\rVert_{L^2} =  \left\lVert H^h_{\infty} \psi_{\operatorname{CN}}^h (t_{k}) \right\rVert_{L^2}.\]
On the other hand, we have that 
\begin{equation*}
\begin{split}
&\Vert H^h_{\infty} f_Q \Vert_{L^2} \le \Vert \Delta^h f_Q \Vert_{L^2} + \Vert V_Q f_Q \Vert_{L^2} \le \Vert \Delta^h f_Q \Vert_{L^2} + \Vert V_Q \Vert_{L^{\infty}} \Vert  f_Q \Vert_{L^2} \text{ and } \\
&\Vert \Delta^h f_Q \Vert_{L^2} \le \Vert  H^h_{\infty} f_Q \Vert_{L^2} + \Vert V_Q f_Q \Vert_{L^2} \le \Vert H^h_{\infty} f_Q \Vert_{L^2} + \Vert V_Q \Vert_{L^{\infty}} \Vert  f_Q \Vert_{L^2}
\end{split}
\end{equation*}
which shows the equivalence of norms 
$
\Vert H^h_{\infty} f_Q \Vert_{L^2} + \Vert f_Q \Vert_{L^2}
$
and
$
\Vert \Delta^h f_Q \Vert_{L^2} + \Vert f_Q \Vert_{L^2}.
$

We can decompose the solution to \eqref{eq:trueone}, into the output from the Crank-Nicholson method and an error term $\zeta^h(t_k)$
\begin{equation*}
\psi_{\infty}^h (t_{k+1}) = e^{-iH^h_{\infty}\tau} \psi_{\infty}^h (t_{k})  = \mathcal C^{\tau}_{h} \psi_{\infty}^h(t_k) + \zeta^{h}(t_k).
\end{equation*}
 
Hence, we conclude by the functional calculus for the bounded self-adjoint operator $H^h_{\infty}$ with spectrum $\sigma(H_{\infty}^h)$ and eigenfunctions ($\varphi_n^h$) that 
for any Borel function $f:\mathbb R \rightarrow \mathbb R$
\[ \left\lVert f(H_{\infty}^h) \psi_h \right\rVert_{L^2(\Omega)}^2   = \sum_{\lambda \in \sigma(H_{\infty}^h)} \vert f(\lambda) \vert^2  \vert \langle  \psi_h, \varphi_n^h \rangle_{L^2(\Omega)} \vert^2. \]
If we combine this with the fact that for $\lambda \in \mathbb R$ and $t \in [0,T]$ there is $C_T>0$ such that
\[ \left\vert e^{-it\lambda}-\left(\frac{1-i\frac{t}{2} \lambda}{1+i\frac{t}{2} \lambda} \right) \right\vert  \le C_T t^3 \vert \lambda \vert^3,\]
we see that there is some constant $C>0$ independent of $h$ and $\tau$ such that 
\begin{equation*}
\begin{split}
\left\lVert  \zeta^{h}(t_k)\right\rVert_{L^2(\Omega)} &\le \left\lVert \left(e^{-iH_{\infty}^h(t_k)\tau}  - \mathcal C^{\tau}_{h} \right)\psi_{\infty}^h (t_{k}) \right\rVert_{L^2(\Omega)}\\
&\le \widetilde{C}  \tau^3 \left\lVert (H^h_{\infty})^3\psi_{\infty}^h (t_{k}) \right\rVert_{L^2(\Omega)}\\
&\le C(\Vert (\varphi_0)_Q \Vert_{H^2_h}) \tau^3/h^4 \text{ and similarly} \\
\left\lVert  \zeta^{h}(t_k)\right\rVert_{H^2_h(\Omega)} &\le C\tau^3/h^6.
\end{split}
\end{equation*}
Here, we used that $\Vert H^h_{\infty} \Vert =\mathcal O(h^{-2}).$ In particular, this computation implies that $C$ is recursively defined.
We notice that since $\varphi_0 \in H^{2+\varepsilon}$ the expression $\Vert (\varphi_0)_Q \Vert_{H^2_h}$ remains uniformly bounded as $h \downarrow 0.$
This implies that 
\begin{equation*}
\begin{split} \left\lVert \psi_{\infty}^h (t_{k+1})-\psi_{\operatorname{CN}}^h (t_{k+1}) \right\rVert_{L^2(\Omega)} 
&\le  \left\lVert \psi_{\infty}^h (t_{k})-\psi_{\operatorname{CN}}^h (t_{k}) \right\rVert_{L^2(\Omega)} + C\tau^3/h^4 \\
&\le  \left\lVert \psi_{\infty}^h (t_{0})-\psi_{\operatorname{CN}}^h (t_{0}) \right\rVert_{L^2(\Omega)} + (k+1)C\tau^3/h^4 \text{ and }\\
 \left\lVert \psi_{\infty}^h (t_{k+1})-\psi_{\operatorname{CN}}^h (t_{k+1}) \right\rVert_{H^2_h(\Omega)} 
&\le  \left\lVert \psi_{\infty}^h (t_{0})-\psi_{\operatorname{CN}}^h (t_{0}) \right\rVert_{H^2_h(\Omega)} + (k+1)C\tau^3/h^6.
\end{split}
\end{equation*}
\end{proof}

\subsection{Time-dependent Schr\"odinger equation, defocusing NLS \& Strang splitting scheme} 

\medskip

We start by first discussing how to include a time-dependent potential to the numerical analysis of linear Schr\"odinger evolutions on bounded domains:

\subsubsection{Time-dependent linear Schr\"odinger equation}

Consider first the time-independent Schr\"odinger operator $(H\psi)= -\Delta \psi + V\psi$ on a bounded domain $\Omega \subset \mathbb R^n.$
The solution to the linear Schr\"odinger equation \eqref{eq:bilSchr} can be obtained from separating the time-dependent part from the time-homogeneous part using the following Strang splitting scheme:

By writing $t_k:= k \tau$ and $\tau$ for the time step, the Strang splitting scheme, corresponding to the midpoint rule in the integral, in one time step and continuous space is given on bounded domains by
\begin{equation}
\begin{split}
\label{eq:Stranglin}
&(1) \quad \psi^-_{k+\tfrac{1}{2}} := e^{-\frac{i\tau H}{2}}\psi^{\operatorname{lin}}_{\operatorname{S}}(t_k), \\
&(2) \quad \psi^+_{k+\tfrac{1}{2}}:=\exp\left( \tau X_{t_k}^{\operatorname{lin}} \right) \psi_{k+\tfrac{1}{2}}^-, \operatorname{ and }\\
&(3) \quad \psi^{\operatorname{lin}}_{\operatorname{S}}(t_{k+1}) :=e^{-\frac{i\tau H}{2}}\psi^+_{k+\tfrac{1}{2}}
\end{split}
\end{equation}
with initial condition $\psi^{\operatorname{lin}}_{\operatorname{S}}(0):=\varphi_0$ and $X_{t_k}^{\operatorname{lin}}:=-iV_{\operatorname{TD}}(t_k)\psi.$ The approach in \eqref{eq:Stranglin} is of course an idealised setting that has to be approximated:
The Strang splitting for the cubic discretization then satisfies with $\mathcal C^{\tau}_h$, the Crank-Nicholson method defined in \eqref{eq:discrule},
\begin{equation}
\begin{split}
\label{eq:Stranglin2}
&(1) \quad \varphi^-_{k+\tfrac{1}{2}} := \mathcal C^{\tau/2}_h \varphi^{\operatorname{lin}}_{\operatorname{S}}(t_k), \\
&(2) \quad \varphi^+_{k+\tfrac{1}{2}}:=\exp_K\left(-i \tau (X_{t_k}^{\operatorname{lin}})_Q \right) \varphi_{k+\tfrac{1}{2}}^-, \operatorname{ and }\\
&(3) \quad \varphi^{\operatorname{lin}}_{\operatorname{S}}(t_{k+1}) :=\mathcal C^{\tau/2}_h \varphi^+_{k+\tfrac{1}{2}}
\end{split}
\end{equation}
where $\varphi^{\operatorname{lin}}_{\operatorname{S}}(0):=(\varphi_0)_Q$, $(X_{t_k}^{\operatorname{lin}})_Q:=-iV_{\operatorname{TD}}(t_k)_Q$, and the function 
\begin{equation}
\label{eq:expK}
\operatorname{exp}_K(x):= \sum_{n=0}^K  \frac{(-1)^n}{(2n)!}x^{2n} + i \frac{(-1)^n}{(2n+1)!} x^{2n+1}.
\end{equation}
For our subsequent error analysis of the above scheme, we need the following technical Lemma:
\begin{lemm}
\label{expK}
For every $\varepsilon>0$ and $r \in (0,\infty)$ there is a recursive map $(\varepsilon,r)  \mapsto K(\varepsilon, r) \in \mathbb N$ such that for all $x \in [-r,r]$
$
\left\lvert  \operatorname{exp}(ix)- \operatorname{exp}_K(ix) \right\rvert \le \varepsilon. 
$
\end{lemm}  
\begin{proof}
From Taylor's formula, we have
\begin{equation*}
\begin{split}
&\left\lvert  \operatorname{exp}(ix)-\left( \sum_{n=0}^N  \frac{(-1)^n}{(2n)!}x^{2n} + i \frac{(-1)^n}{(2n+1)!} x^{2n+1} \right) \right\rvert \\
&\le \left\lvert  \cos(x)-\sum_{n=0}^N  \frac{(-1)^n}{(2n)!}x^{2n} \right\rvert + \left\lvert  \sin(x) - \sum_{n=0}^N \frac{(-1)^n}{(2n+1)!} x^{2n+1} \right\rvert \\
&\le \left\lvert   \frac{x^{2N+2}}{(2N+2)!} \right\rvert + \left\lvert   \frac{x^{2N+3}}{(2N+3)!}  \right\rvert.
\end{split}
\end{equation*}
By Stirling's approximation $N! \ge \sqrt{2\pi} N^{N+1/2} e^{-N},$ it follows that $\frac{x^N}{N!} \le \left( \frac{x}{eN} \right)^N \frac{1}{\sqrt{2\pi N}}$ which implies the claim. \end{proof}

We then get the following convergence result:

\begin{prop}
\label{prop:lintime}
Consider an initial state $\varphi_0 \in H^{2+\varepsilon}_2(\Omega) \cap H^1_0(\Omega)$ with $\varepsilon \in (0,1)$, controls $u \in W^{1,1}_{\operatorname{pcw}}(0,T)$, and potentials $V,V_{\operatorname{con}} \in W^{2,\infty}(\Omega)$.
Then there exist recursive maps 
\begin{equation*}
\begin{split}
&T, \Vert \varphi_0 \Vert_{H^{2+\varepsilon}},\Vert V \Vert_{W^{2,\infty}},\Vert V_{\operatorname{con}} \Vert_{W^{2,\infty}}, h,\Vert u \Vert_{W^{1,1}_{\operatorname{pcw}}}\\
& \qquad \mapsto \tau(T,\Vert \varphi_0 \Vert_{H^{2+\varepsilon}},\Vert V \Vert_{W^{2,\infty}},\Vert V_{\operatorname{con}} \Vert_{W^{2,\infty}}, h,\Vert u \Vert_{W^{1,1}_{\operatorname{pcw}}})
\end{split}
\end{equation*} 
and 
\begin{equation*}
\begin{split}
&T,\tau, \Vert \varphi_0 \Vert_{H^{2+\varepsilon}},\Vert V \Vert_{W^{2,\infty}},\Vert V_{\operatorname{con}} \Vert_{W^{2,\infty}},  h,\Vert u \Vert_{W^{1,1}_{\operatorname{pcw}}}  \\
& \qquad \mapsto K(T,\tau, \Vert \varphi_0 \Vert_{H^{2+\varepsilon}},\Vert V \Vert_{W^{2,\infty}},\Vert V_{\operatorname{con}} \Vert_{W^{2,\infty}}, h,\Vert u \Vert_{W^{1,1}_{\operatorname{pcw}}})
\end{split}
\end{equation*}
such that the solution $\varphi_S$ obtained from the Strang splitting scheme \eqref{eq:Stranglin2} satisfies with respect to the full solution $\psi$ of \eqref{eq:bilSchr} for some recursively-defined function
\[
C=C(T, \Vert \varphi_0 \Vert_{H^{2+\varepsilon}}, \Vert V\Vert_{W^{2,\infty}},\Vert V_{\operatorname{con}}  \Vert_{W^{2,\infty}}, \Vert u \Vert_{W^{1,1}_{\operatorname{pcw}}})
\] 
the bound
\[ \max_{k} \Vert \psi(\tau k) - \varphi^{\operatorname{lin}}_{S}(\tau k) \Vert_{L^2} \le C \Vert \varphi_0 -(\varphi_0)_Q \Vert_{L^2}.\]
\end{prop}

Next, we turn to the numerical analysis of nonlinear Schr\"odinger equations.

\subsubsection{Nonlinear Schr\"odinger equation}\label{sec:defoc}

We now extend the scheme to the nonlinear Schr\"odiger equation, cf. also \cite{L08}. For the NLS, we require an $H^1_h$ convergent scheme on a bounded domain $\Omega \subset \mathbb R^n$ to control the nonlinearity.
In the Strang splitting scheme, the dynamics due to the nonlinearity and the time-dependent potential is separated from the linear Schr\"odinger dynamics that we discussed in Proposition \ref{prop:Schr} using the Crank-Nicholson method:

By writing $t_k:=k \tau$ and $\tau$ for the time step, the Strang splitting scheme, corresponding to the midpoint rule in the integral, in one time step and continuous space is given by 
\begin{equation}
\begin{split}
\label{eq:StrangNLS}
&(1) \quad \psi^-_{k+\tfrac{1}{2}} := e^{-\frac{i\tau H}{2}}\psi^{\operatorname{NLS}}_{\operatorname{S}}(t_k), \\
&(2) \quad \psi^+_{k+\tfrac{1}{2}}:=\exp\left(\tau X_{t_k}^{\operatorname{NLS}}\left(\psi_{k+\tfrac{1}{2}}^-\right)  \right) \psi_{k+\tfrac{1}{2}}^-, \operatorname{ and }\\
&(3) \quad \psi^{\operatorname{NLS}}_{\operatorname{S}}(t_{k+1}) :=e^{-\frac{i\tau H}{2}}\psi^+_{k+\tfrac{1}{2}}.
\end{split}
\end{equation}
where $\psi^{\operatorname{NLS}}_{\operatorname{S}}(0):=\varphi_0$ and $(X_{t_k}^{\operatorname{NLS}}(\psi))_Q:=-i(V_{\operatorname{TD}}(t_k) + \vert  \psi \vert^{\sigma-1}).$

The Strang splitting scheme for the cubic discretization then satisfies with $\mathcal C^{\tau}_h$ being the Crank-Nicholson method defined in \eqref{eq:discrule}
\begin{equation}
\begin{split}
\label{eq:StrangNLS2}
&(1) \quad \varphi^-_{k+\tfrac{1}{2}} := \mathcal C^{\tau/2}_h \varphi^{\operatorname{NLS}}_{\operatorname{S}}(t_k), \\
&(2) \quad \varphi^+_{k+\tfrac{1}{2}}:=\exp_K\left(-i \tau X_{t_k}^{\operatorname{NLS}}(\varphi^-_{k+\tfrac{1}{2}})_Q \right) \varphi_{k+\tfrac{1}{2}}^-\\
&(3) \quad \varphi^{\operatorname{NLS}}_{\operatorname{S}}(t_{k+1}) :=\mathcal C^{\tau/2}_h \varphi^+_{k+\tfrac{1}{2}}
\end{split}
\end{equation}
where $\varphi^{\operatorname{NLS}}_{\operatorname{S}}(0):=(\varphi_0)_Q$, $X_{t_k}^{\operatorname{NLS}}(\psi)\psi:=-i(V_{\operatorname{TD}}(t_k) + \vert  \psi \vert^{\sigma-1})\psi,$ and $\exp_K$ defined in \eqref{eq:expK}.

\medskip

For the NLS we show convergence of the numerical scheme in $H^1_h$ which requires one integer higher Sobolev exponents in the initial state and potentials than what is needed for $L^2$ convergence of the linear Schr\"odinger equation, cf. Prop. \ref{prop:lintime}.

\begin{prop}
\label{prop:nontime}
Consider an initial state $\varphi_0 \in H^{3+\varepsilon}_2(\Omega) \cap H^1_0(\Omega)$ for some $\varepsilon>0$, potentials $V,V_{\operatorname{con}} \in W^{3,\infty}(\Omega)$, and a control function $u \in W^{1,1}_{\text{pcw}}((0,T))$.
Then there exist recursive maps 
\begin{equation*}
\begin{split}
&T,\Vert \varphi_0 \Vert_{H^{3+\varepsilon}},\Vert V \Vert_{W^{3,\infty}},\Vert V_{\operatorname{con}} \Vert_{W^{3,\infty}},  h,\Vert u \Vert_{W^{1,1}_{\operatorname{pcw}}}  \\
&\qquad \mapsto \tau(T,\Vert \varphi_0 \Vert_{H^{3+\varepsilon}},\Vert V \Vert_{W^{3,\infty}},\Vert V_{\operatorname{con}} \Vert_{W^{3,\infty}},  h,\Vert u \Vert_{W^{1,1}_{\operatorname{pcw}}})
\end{split}
\end{equation*} 
and 
\begin{equation*}
\begin{split}
&\tau, T,\Vert \varphi_0 \Vert_{H^{3+\varepsilon}},\Vert V \Vert_{W^{3,\infty}},\Vert V_{\operatorname{con}} \Vert_{W^{3,\infty}}, h,\Vert u \Vert_{W^{1,1}_{\operatorname{pcw}}} \\
& \qquad \mapsto K(\tau,T, \Vert \varphi_0 \Vert_{H^{3+\varepsilon}},\Vert V \Vert_{W^{3,\infty}},\Vert V_{\operatorname{con}} \Vert_{W^{3,\infty}}, h,\Vert u \Vert_{W^{1,1}_{\operatorname{pcw}}})
\end{split}
\end{equation*} 
such that the solution $\varphi_S$ obtained from the Strang splitting scheme \eqref{eq:StrangNLS2} satisfies with respect to the full solution $\psi$ to \eqref{eq:GP} for some recursively defined function 
\[
C=C(T, \Vert \varphi_0 \Vert_{H^{3+\varepsilon}}, \Vert V\Vert_{W^{3,\infty}},\Vert V_{\operatorname{con}}  \Vert_{W^{3,\infty}}, \Vert u \Vert_{W^{1,1}_{\operatorname{pcw}}})
\]
the estimate
\[
 \max_{k} \Vert \psi(\tau k) - \varphi^{\operatorname{NLS}}_{S}(\tau k) \Vert_{H^1_h} \le C  \Vert \varphi_0 -(\varphi_0)_Q \Vert_{H^1_h}.
 \]
\end{prop}

\medskip
\begin{proof}[Proof of Prop. \ref{prop:lintime} and \ref{prop:nontime}]
The proof consists of the following steps:
\begin{enumerate}
\item We first approximate the full solution by splitting it into an evolution of a linear Schr\"odinger equation and a potential or nonlinear part (Strang splitting scheme)
\item The time-dependent or nonlinear part is approximated by discretizing the exponential on cubes. 
\item The linear evolution is first approximated by a discrete Laplacian using Lemma \ref{dynlem} and this one is propagated by a Crank-Nicholson method, cf. Prop. \ref{prop:Schr}.
\item Finally, we argue how to choose $K$ in \eqref{eq:Stranglin2} and \eqref{eq:StrangNLS2} using Lemma \ref{expK}.
\end{enumerate}
\smallsection{(1):} \quad The first part of the proof, the reduction by the Strang splitting scheme, follows along the lines of \cite{L08}:
It suffices to consider a single time-step as the claim then follows from summing over all time steps in the bounded interval $[0,T]$.

We write $\varphi_Y^t(\psi)$ for the solution to $\frac{d}{dt} \varphi^t_Y = Y \varphi^t_Y$ with initial value $\varphi^0_Y=\psi.$
To analyze this Strang splitting method, we then introduce the Lie derivative $(\mathscr L_Y G)(\psi)=\frac{d}{dt} \big \vert_{t=0}G(\varphi_{Y}^t(\psi))$ along a vector field $Y$ and the exponential map 
$
 \left(e^{t \mathscr L_Y}G\right)(\psi) := G(\varphi_Y^t(\psi)). 
 $
In particular, we consider vector fields
\[X^{\operatorname{lin}}_{t_k}(\psi)\psi:=-iV_{\operatorname{TD}}(t_k)\psi \text{ and }X^{\operatorname{NLS}}_{t_k}(\psi)\psi:=-iV_{\operatorname{TD}}(t_k)\psi-i\vert \psi \vert^{\sigma-1} \psi. \] 
The variation of constant formula reads then for $t \in [0,t_{k+1}-t_{k}]$
\begin{equation}
\begin{split}
\label{eq:VoC}
\left(e^{t \mathscr L_{-iH+X_{t_k}}}G\right)(\psi)= \left(e^{t \mathscr L_{-iH}}G + \int_0^t e^{(t-s)  \mathscr L_{-iH+X_{t_k}}} \mathscr L_{X_{t_k}} e^{s \mathscr L_{-iH}} G \ ds \right)(\psi).
\end{split}
\end{equation}

We then obtain from this expression a formula for the solution to the Schr\"odinger equation by choosing $G=\operatorname{id}$. It follows from comparing this exact expression with the Strang splitting method that the leading order error of the splitting scheme is given by the error of the midpoint rule applied to the function 
\[ f(s):=e^{(\tau-s) \mathscr L_{-iH}} \mathscr L_{X_{t_k}} e^{s \mathscr L_{-iH}} \varphi_0 \text{ for } s \in [0,\tau].\] 
The midpoint rule satisfies an error estimate
\begin{equation} 
\begin{split}
\label{eq:midpoint}
 \left\lVert \tau f(\tau/2)- \int_0^{\tau} f(s) \ ds \right\rVert \le \tau^2 \int_0^1 \vert \kappa_{\operatorname{mid}}(s) \vert \left\lVert f'( \tau s) \right\rVert \ ds,  
 \end{split}
\end{equation}
where
\begin{equation} 
\begin{split}
\kappa_{\operatorname{mid}}(s)&= \begin{cases} -s^2/2&\text{ if }s \le 1/2\text{ and }\\
-(s-1)^2/2&\text{ if }s> 1/2,
\end{cases}
\end{split}
\end{equation}
is the continuous Peano kernel.
The error in the Strang splitting scheme is composed of the error of the midpoint rule \eqref{eq:midpoint} and the $L^2/H^1$ norm of functions $r_1$ and $r_2$, see \cite[4.4]{L08} for details, where for $t \in [0,\tau]$
\[ r_1(t):=\int_0^{t} \int_0^{t-s} e^{(t-s-\sigma) \mathscr L_{-iH+X_{t_k}}} \mathscr L_{X_{t_k}} e^{\sigma \mathscr L_{-iH}} \mathscr L_{X_{t_k}} e^{s \mathscr L_{-iH}} \varphi_0 \ d \sigma \ ds \]
and
\[ r_2(t):=t^2 \int_0^{1} (1-\theta) e^{\frac{t}{2} \mathscr L_{-iH}}e^{\theta t \mathscr L_{X_{t_k}}}  \mathscr L_{X_{t_k}}^2 e^{\frac{t}{2} \mathscr L_{-iH}}\varphi_0 \ d\theta. \] 
We then have for the integrand in $r_1$
\begin{equation} 
\begin{split}
 &e^{\rho \mathscr L_{-iH+X_{t_k}}} \mathscr L_{X_{t_k}} e^{\sigma \mathscr L_{X_{t_k}}}\mathscr L_{X_{t_k}}e^{s\mathscr \mathscr L_{-iH}}\varphi_0  \\
 &=e^{-is H} DX_{t_k}(e^{-i \sigma H} \varphi^{\rho}_{-iH+X_{t_k}}(\psi) ) e^{-i\sigma H}X_{t_k} ( \varphi^{\rho}_{-iH+X_{t_k}}(\psi)) 
\end{split}
\end{equation}
and for the integrand in $r_2$ with $\eta:=e^{-i\theta \tau X_{t_k}(\phi)} \phi$ where $\phi=e^{ \tfrac{-itH}{2}}\varphi_0$
\begin{equation} 
\begin{split}
&e^{\tfrac{t}{2} \mathscr L_{-iH}} \mathscr L_{X_{t_k}} e^{\theta t \mathscr L_{X_{t_k}}} \mathscr L_{X_{t_k}}^2 e^{\tfrac{t}{2}\mathscr L_{-iH}}\varphi_0=e^{-i \tfrac{tH}{2} }DX_{t_k}(\eta)(X_{t_k}(\eta)).
\end{split}
\end{equation}

For the linear Schr\"odinger equation, there is a recursive function $C(\Vert V_{\operatorname{TD}}(t_k) \Vert_{L^{\infty}(\Omega)})>0$ such that for $t \in [0,\tau]$
$
\Vert r_1(t) \Vert_{L^2(\Omega)}+\Vert r_2(t) \Vert_{L^2(\Omega)} \le C \tau^2
$
and for the nonlinear Schr\"odinger equation, there is a recursive function 
$
C(\Vert V_{\operatorname{TD}}(t_k) \Vert_{W^{1,\infty}(\Omega)},\Vert \varphi_0 \Vert_{H^1(\Omega)} )>0
$ such that for $t \in [0,\tau]$
\[ \Vert r_1(t) \Vert_{H^1(\Omega)}+\Vert r_2(t) \Vert_{H^1(\Omega)} \le C \tau^2.\]

Taking the $L^2$ norm in \eqref{eq:midpoint}, we find for the linear Schr\"odinger equation
\begin{equation}
\label{eq:f'} 
f'(s) = e^{-is H} [H,u(t)V_{\operatorname{con}}] (e^{-i(\tau-s)H} \varphi_0).
\end{equation}
A computation shows then that for arbitrary $\psi \in H^1$
\begin{equation} 
\begin{split}
[-\Delta+V,u(t)V_{\operatorname{con}}]\psi &= - 2 u(t) \langle \nabla V_{\operatorname{con}}, \nabla \psi \rangle - (\Delta V(t_k) ) \psi
 \end{split}
\end{equation}
such that
\begin{equation*} 
\begin{split}
\label{eq:linest}
\left\lVert [-\Delta+V,u(t)V_{\operatorname{con}}]\psi \right\rVert_{L^2}&\lesssim  \vert u(t) \vert \Vert V_{\operatorname{con}} \Vert_{W^{2,\infty}(\Omega)} \Vert \psi \Vert_{H^1(\Omega)}.
 \end{split}
\end{equation*}
We may apply this estimate in our setting as the solution to linear Schr\"odinger equation is uniformly bounded in $H^1,$ cf. Lemma \ref{eos}. 

Similarly, for the nonlinear Schr\"odinger equation a computation shows then that
\begin{equation*} 
\begin{split}
[-\Delta,X^{\operatorname{NLS}}_{t_k}] \psi&= (-\Delta )( (V+V_{\operatorname{TD}}(t_k)) \psi+\vert \psi \vert^{\sigma-1} \psi)-(V+V_{\operatorname{TD}})(-\Delta \psi)\\
&\quad -\frac{\sigma+1}{2} \vert \psi \vert^{\sigma-1} (-\Delta)\psi -\frac{\sigma-1}{2} \vert \psi \vert^{\sigma-3}\psi^2 (-\Delta \overline{\psi})
 \end{split}
\end{equation*}
which implies by the Sobolev embedding that 
\begin{equation}
\label{eq:NLSest}
\left\lVert [-\Delta,X^{\operatorname{NLS}}_{t_k}](\psi) \right\rVert_{H^1(\Omega)} \lesssim (\Vert V \Vert_{W^{2,\infty}(\Omega)}+\Vert V_{\operatorname{TD}}(t_k) \Vert_{W^{2,\infty}(\Omega)} ) \Vert \psi \Vert_{H^3(\Omega)} \Vert \psi \Vert^{\sigma-1}_{H^2(\Omega)} .
\end{equation}
To apply this estimate, we use that the continuous space solution is bounded in $H^3,$ cf. Lemma \ref{eos2}.

To obtain a quadratic error in $\tau$ for a single step (and thus a linear error in the time step on the entire time interval) from the midpoint rule \eqref{eq:midpoint}, it suffices to estimate the term \eqref{eq:f'}. For the linear Schr\"odinger equation this is \eqref{eq:linest} and for the NLS this is \eqref{eq:NLSest}. 

\smallsection{(2):} Next, we compare the exponential step (2) in \eqref{eq:Stranglin} and \eqref{eq:StrangNLS} (in continuous space) to the respective discretized exponential steps in \eqref{eq:Stranglin2} or \eqref{eq:StrangNLS2} with $K=\infty$, first.
Since the dynamics in steps (1) and (3) for the discretized evolution is just governed by the evolution of a discretized Schr\"odinger operator, we can use Lemma \ref{dynlem} and Prop. \ref{prop:Schr} to study the evolution in step (2). 
We then have that for the $L^2$ norm with $C_1, C_2$ two recursively defined functions
\begin{equation}
\begin{split}
\label{eq:estimateIIII}
&\left\Vert e^{-i\tau X_{t_k}\left(\psi_{k+1/2}^-\right) } \psi_{k+\tfrac{1}{2}}^- - \operatorname{exp}\left(-i\tau X_{t_k}\left(\varphi_{k+1/2}^-\right)_Q \right) \varphi^-_{k+\tfrac{1}{2}}\right\rVert_{L^2} \\
&\le \left\Vert
\psi_{k+\tfrac{1}{2}}^- -\varphi^-_{k+\tfrac{1}{2}} \right\rVert_{L^2}+\left\Vert  \left(e^{-i\tau X_{t_k}\left(\psi_{k+1/2}^-\right)}-  \operatorname{exp}\left(-i\tau X_{t_k}\left(\varphi_{k+1/2}^-\right)_Q \right)\right) \varphi^-_{k+\tfrac{1}{2}} \right\rVert_{L^2} \\
&\le (1+\tau C_1(\Vert V_{\operatorname{TD}}(t_k) \Vert_{L^{\infty}})) \left\lVert \psi_{k+\tfrac{1}{2}}^- - \varphi^-_{k+\tfrac{1}{2}} \right\rVert_{L^2} \\
&\quad + \tau C_2(\Vert \psi_{k+1/2}^-\Vert)\left\lVert V_{\operatorname{TD}}(t_k)-V_{\operatorname{TD}}(t_k)_Q \right\rVert_{L^{\infty}}.
\end{split}
\end{equation}
Moreover, we conclude from the product rule of the discrete Laplacian that there are $C_1, C_2$ two recursively defined functions
\begin{equation}
\begin{split}
\label{eq:estimateIII}
&\left\Vert  \nabla_h  \left(e^{-i\tau X^{\operatorname{NLS}}_{t_k}\left(\psi_{k+1/2}^-\right) } \psi_{k+\tfrac{1}{2}}^- - \operatorname{exp}\left({-i\tau X^{\operatorname{NLS}}_{t_k}\left(\varphi^-_{k+\tfrac{1}{2}}\right)_Q} \varphi^-_{k+\tfrac{1}{2}} \right)\right)\right\rVert_{L^2} \\
&\quad \le \left\Vert \nabla_h \left(\operatorname{exp}\left({-i\tau X^{\operatorname{NLS}}_{t_k}\left(\varphi^-_{k+\tfrac{1}{2}}\right)_Q} \varphi^-_{k+\tfrac{1}{2}} \right)\left(\psi_{k+\tfrac{1}{2}}^- - \varphi^-_{k+\tfrac{1}{2}}\right)\right)\right\rVert_{L^2}\\
&\quad +\left\Vert \nabla_h \left(\left(e^{-i\tau X_{t_k}^{\operatorname{NLS}}\left(\psi_{k+1/2}^-\right)}- \operatorname{exp}\left(-i\tau X^{\operatorname{NLS}}_{t_k}\left(\varphi^-_{k+\tfrac{1}{2}}\right)_Q\right)\right) \varphi^-_{k+\tfrac{1}{2}} \right)\right\rVert_{L^2} \\
&\le \left\Vert \psi_{k+\tfrac{1}{2}}^- - \varphi^-_{k+\tfrac{1}{2}} \right\rVert_{H^1_h} \left(1+ C_1(\Vert V_{\operatorname{TD}}(t_k) \Vert_{W^{1,\infty}})  \tau  \left(1+\left\Vert  \psi_{k+\tfrac{1}{2}}^-  \right\Vert_{H^1}^{\sigma-1} + \left\Vert  \varphi^-_{k+\tfrac{1}{2}}  \right\Vert_{H^1_h}^{\sigma-1}\right)\right)\\
&\quad + \tau C_2(\Vert \psi_{k+1/2}^-\Vert_{H^1})\left\lVert V_{\operatorname{TD}}(t_k)-V_{\operatorname{TD}}(t_k)_Q \right\rVert_{W_h^{1,\infty}}.
\end{split}
\end{equation}
where we used that for functions $\varphi$, that are constant on cubes, we have $\left\lVert  \varphi \right\rVert_{L^{\infty}} \le C \left\lVert \varphi \right\rVert_{H^1_h}$ for a universal constant $C>0$ independent of $h$ and $\varphi.$ 

Hence, the error propagates at most linearly in this step, as long as the $H^1$ norm of $\psi$ and the $H^1_h$ norm of $\varphi$ remain uniformly bounded, as well. The boundedness of the $H^1$ norm of $\psi$ follows from \eqref{eq:midpoint}.
Since the convergence of the scheme is in $H^1_h$ we can bound the $H^1_h$ norm of $\varphi$ using $\psi$ and the local error. The $H^1_h$ norm of $\psi$ however is controlled by the $H^1$ norm of $\psi$ which is uniformly bounded. This is sufficient for the global convergence of the Strang splitting scheme.

Combining the Splitting scheme (1), the convergence of the exponential function in (2), and (3) the convergence Lemma \ref{dynlem}, the error of the Crank-Nicholson method, cf. Prop. \ref{prop:Schr}, the time-step for the linear Schr\"odinger equation is a recursive function defined as 
\[
\tau(T, \Vert V \Vert_{W^{2,\infty}(\Omega)},\Vert V_{\operatorname{con}} \Vert_{W^{2,\infty}(\Omega)}, \Vert u \Vert_{W^{1,1}_{\operatorname{pcw}}}, \Vert \varphi_0 \Vert_{H^2}).
\]
For the NLS, the time-step is a recursive function
\[
\tau(T, \Vert V \Vert_{W^{3,\infty}(\Omega)},\Vert V_{\operatorname{con}} \Vert_{W^{3,\infty}(\Omega)}, \Vert u \Vert_{W^{1,1}_{\operatorname{pcw}}}, \Vert \varphi_0 \Vert_{H^3}).
\]

\smallsection{(4):}
As a last step, we use the convergence of the full solution to the discretized ones that we established above. This implies that the discretized solutions are uniformly bounded on $(0,T)$. Thus, by Lemma \ref{expK} we can recursively find a $K$ to approximate the exponential function $K$ by $\operatorname{exp}_K$ where $K$ depends on 
\[(T,\tau, \Vert \varphi_0 \Vert_{H^{2+\varepsilon}},\Vert V \Vert_{W^{2,\infty}},\Vert V_{\operatorname{con}} \Vert_{W^{2,\infty}}, h,\Vert u \Vert_{W^{1,1}_{\operatorname{pcw}}})\]
in case of the linear Schr\"odinger equation and 
$
(T,\tau, \Vert \varphi_0 \Vert_{H^{2+\varepsilon}},\Vert V \Vert_{W^{3,\infty}},\Vert V_{\operatorname{con}} \Vert_{W^{3,\infty}}, h,\Vert u \Vert_{W^{1,1}_{\operatorname{pcw}}})
$
for the NLS. 
\end{proof}

We finish with a perturbation result that allows us to take also the numerical integration error for the potentials and the initial state when approximating the cubic discretization into account. 
\begin{prop}
\label{prop:approx_int}
Consider potentials $V,V_{\operatorname{con}}$ and $\widetilde V,\widetilde{V}_{\operatorname{con}}$ that are constant on cubes of fixed(!) size $h$ contained in a bounded cube $\Omega \subset \RR^d$ of side length $R$ and two initial states $\varphi_0, \widetilde{\varphi_0} $ that are constant on the same cubes. We then define the output of the numerical schemes \eqref{eq:Stranglin2} and \eqref{eq:StrangNLS2} for either potential and initial state by $\psi$ and $\widetilde \psi$ respectively. 
For some fixed step size $\tau$ and $n:=\lceil T/\tau \rceil$, there exists a recursively defined function $C(R,h, \tau)>0$ such that for 
\[\operatorname{max}\left\{\left\Vert V-\widetilde{V} \right\rVert_{L^{\infty}}, \left\Vert V_{\operatorname{con}}-\widetilde{V}_{\operatorname{con}} \right\rVert_{L^{\infty}}, \left\lVert  \varphi_0-\widetilde{\varphi_0} \right\Vert_{L^2}\right\} \le 1\] 
we have for a recursive function $C=C(R,h, \tau)$
\begin{equation*}
\begin{split}
\max_{k \in [n]}\left\Vert \psi_{\operatorname{S}}(\tau k)-\widetilde \psi_{\operatorname{S}}(\tau k) \right\Vert_{L^2} \le C\operatorname{max} \Bigg\{  \left\Vert V-\widetilde{V} \right\rVert_{L^{\infty}},\left\Vert V_{\operatorname{con}}-\widetilde{V}_{\operatorname{con}} \right\rVert_{L^{\infty}}, \left\lVert \varphi_0-\widetilde{\varphi_0} \right\Vert_{L^2}\Bigg\}.
\end{split}
\end{equation*}
\end{prop}
\begin{proof}
In the first step, we may use that for two Hamiltonians $H=-\Delta^h+V_Q$ and $\widetilde{H}=-\Delta^h+\widetilde{V}_Q$ we have from the resolvent identity
\begin{equation}
\begin{split}
\label{eq:1}
&\left\lVert \left(1- \tfrac{i\tau}{2} H \right)\left(1+ \tfrac{i\tau}{2} H \right)^{-1} -\left(1- \tfrac{i\tau}{2} \widetilde{H} \right)\left(1+ \tfrac{i\tau}{2} \widetilde{H}\right)^{-1} \right\rVert \\
& \le \tfrac{\tau}{2} \left\lVert V-\widetilde V \right\rVert_{L^{\infty}} \Bigg(\left\Vert \left(1+ \tfrac{i\tau}{2} H \right)^{-1} \right\Vert+ \\
&\qquad \qquad \left\lVert 1- \tfrac{i\tau}{2} \widetilde{H}  \right\rVert \left\lVert \left(1+ \tfrac{i\tau}{2} H \right)^{-1} \right\rVert \left\lVert \left(1- \tfrac{i\tau}{2} H \right)^{-1}\right\rVert \Bigg)\\
&\le \tfrac{\tau}{2} \left\lVert V-\widetilde V \right\rVert_{L^{\infty}}  \times \left(2+\tfrac{\tau}{2} \left\Vert \widetilde H \right\Vert \right).
\end{split}
\end{equation}
This shows that the approximation error in the Crank-Nicholson step (1) and (3) in \eqref{eq:Stranglin2} is controlled by the potential difference. 

The potential difference of the control potential and the error in the state also determines the error in step $(2)$ of the splitting scheme. 
To see this, it suffices to note that since there are only finitely many cubes in the bounded domain $\Omega$, there exists a recursive function $K=K(R,h)>0$ such that for any function $f$ that is constant on cubes we find $\Vert f \Vert_{L^{\infty}} \lesssim K(R,h) \Vert f \Vert_{L^2}.$ 

Since the $L^2$ norm is preserved, up to error of order $\tau^K$ by the scheme \eqref{eq:Stranglin2}, we find for $X_{t_k}^{\operatorname{NLS}}$ that
\begin{equation}
\begin{split}
\label{eq:2}
&\left\lVert \exp_K\left(-i \tau X_{t_k}^{\operatorname{NLS}}(\varphi^{-})\right) - \exp_K\left(-i \tau \widetilde{X}_{t_k}^{\operatorname{NLS}}(\widetilde{\varphi^{-}}) \right) \right\rVert_{L^{\infty}} \\
&\le \tau \left\lVert  X_{t_k}^{\operatorname{NLS}}(\varphi^{-}) - \widetilde{X}_{t_k}^{\operatorname{NLS}}(\widetilde{\varphi^{-}}) \right\rVert_{L^{\infty}}  \\
&\lesssim \tau \Bigg( \left\lVert V_{\operatorname{TD}}(t_k)_Q-\widetilde{V_{\operatorname{TD}}}(t_k)_Q \right\rVert + K(R,h) \left(\left\lVert  \varphi^{-} \right\rVert^{\sigma-2}_{L^2}+\left\lVert \widetilde \varphi^{-} \right\rVert^{\sigma-2}_{L^2} \right)\left\lVert \varphi^{-} - \widetilde \varphi^{-} \right\rVert^{}_{L^2}\Bigg).
\end{split}
\end{equation}
Combining \eqref{eq:1} and \eqref{eq:2} yields the existence of a recursive function $C(R,h,\tau).$ 
\end{proof}

\section{Proof of Theorem \ref{Th:main_thrm_lower_bound}}

\begin{proof}[Proof of Theorem \ref{Th:main_thrm_lower_bound}]
To prove part I of Theorem \ref{Th:main_thrm_lower_bound} we show that $\{\Omega^1_{\mathrm{free}}, \Xi^1_{\mathrm{free}}\} \notin \Delta_1^G$ and argue by contradiction and assume that $\{\Omega^1_{\mathrm{free}}, \Xi^1_{\mathrm{free}}\} \in \Delta_1^G$. 
By a simple translation of variables, it suffices to assume that $O$ is an interval centred at zero. Let $2L$ be the length of this interval.
By the assumption that $\{\Omega^1_{\mathrm{free}}, \Xi^1_{\mathrm{free}}\} \in \Delta_1^G$ we can find a sequence $\{\Gamma_n\}$ of general algorithms such that $\|\Gamma_n(\varphi)- \Xi^1_{\mathrm{free}}(\varphi)\|_{L^2(O)} \leq 2^{-n}$ for all $\varphi \in \Omega^1_{\mathrm{free}}$.  Choose $\varphi_0$ to be the zero function, pick any $\sigma_0 > 0$ and choose $N$ large enough so that 
\begin{equation}\label{eq:lower}
\frac{1}{{(2\pi \sigma_0^2)^{1/4}}}    \int_{-L}^L  \left| e^{-\frac{x^2}{4\sigma_0}}\right|^2 \, dx  > 3 \cdot 2^{-N}.
\end{equation} 
Recall that 
\begin{equation}\label{eq:the_Lambda}
\Lambda = \{f_m\}_{m\in \mathbb{N}}, \quad f_m(\varphi) = \varphi(t_m), \quad \varphi \in \Omega,
\end{equation}
where $\{t_m\}$ is an enumeration of the rational numbers. Let $K = \max|\{t_m \in \mathbb{R} \, \vert \,  f_m \in \Lambda_{\Gamma_N}(\varphi_0)\}|$. Let 
\[
\tilde \varphi(x) = \rho_{K,\epsilon}(x)\psi(x), \quad \psi(x) = \frac{e^{-(x+kT)^2/(4\sigma_0 \nu) + ik\left(x+\frac{kT}{2} \right)}}{\left(2\pi \nu^2 \right)^{1/4}}, 
\]  
where $\nu:=\sigma_0 \left(1 - \frac{iT}{2\sigma_0^2} \right)$,  $\rho_{K,\epsilon}$ is a smooth bump function that is zero outside of $[-K,K]$ and one on $[-(K-\epsilon),K-\epsilon]$ for some $\epsilon > 0$, and $k \in \mathbb{R}$. Clearly, $\psi$ is computable and $\rho_{K,\epsilon}$ can easily be made computable and hence by any appropriate scaling we have that $\tilde \varphi \in \Omega_{\mathrm{free}}$. Now choose $\epsilon$ small enough and $k$ large enough so that $\|\tilde \varphi - \psi\|_{L^2(\mathbb{R})} \leq 2^{-N}$. Note that 
\[
\Xi^1_{\mathrm{free}}(\psi)(x) = \frac{e^{-\frac{x^2}{4\sigma_0} + ikx}}{(2\pi \sigma_0^2)^{1/4}},
\]
 by using the Fourier transform, \cite[Sec.$7.3$]{Te14}. Hence, as time evolution of the Schr\"odinger equation preserves the $L^2$ norm we claim that 
\begin{equation}
\begin{split}
\frac{1}{{(2\pi \sigma_0^2)^{1/4}}}    \int_{-L}^L  \left| e^{-\frac{x^2}{4\sigma_0}}\right|^2 \ dx &\leq \|\Xi_{\mathrm{free}}(\psi) - \Xi_{\mathrm{free}}(\tilde \varphi)\|_{L^2(O)} +  \|\Xi_{\mathrm{free}}(\tilde \varphi) - \Gamma_N(\tilde \varphi)\|_{L^2(O)}\\
&\quad + \|\Gamma_N(\tilde \varphi) - \Xi_{\mathrm{free}}(\varphi_0)\|_{L^2(O)} \leq 3 \cdot 2^{-N},
\end{split}
\end{equation}   
which contradicts \eqref{eq:lower}. Indeed, the first two terms in the right and side are each bounded by $2^{-N}$ by the choices made above. Thus, we are only left with the last term. Note that it suffices to show that $\Gamma_N(\tilde \varphi) = \Gamma_N(\varphi_0)$ since we have that $\|\Gamma_N(\varphi)- \Xi_{\mathrm{free}}(\varphi)\|_{L^2(O)} \leq 2^{-N}$. To see this, note that by the choice of $K$ in the definition of $\tilde \varphi$ it follows that for any $f \in \Lambda_{\Gamma_N}(\varphi_0)$ we have $f(\varphi_0) = f(\tilde \varphi)$. Thus, by assumption (ii) and (iii) in Definition \ref{alg} it follows that $\Gamma_N(\tilde \varphi) = \Gamma_N(\varphi_0)$.

To prove part II of Theorem \ref{Th:main_thrm_lower_bound} we argue by contradiction and assume that $\{\Omega^2_{\mathrm{free}}, \Xi^2_{\mathrm{free}}\} \in \Delta_1^G$, thus we can find a sequence $\{\Gamma_n\}$ of general algorithms such that $\|\Gamma_n(\varphi) - \Xi^2_{\mathrm{free}}(\varphi)\|_{L^2(\mathbb{R})} \leq 2^{-n}$ for all $\varphi \in \Omega^2_{\mathrm{free}}$. By considering shifts and slight variations of 
\[
{\displaystyle \gamma(x)={\begin{cases}\exp \left(-{\frac {1}{1-x^{2}}}\right),&x\in (-1,1)\\0,&{\mbox{otherwise}},\end{cases}}}
\]
it is clear that one can find a sequence $\{\varphi_k\} \subset \Omega^2_{\mathrm{free}}$ of bump functions such that $\mathrm{supp}(\varphi_k) \subset [k,k+1]$, $\|\varphi_k\|_{L^2(\mathbb{R})} = \tilde C$, for some $\tilde C > 0$, where each $\varphi_k$ is computable. Choose any $M \in \mathbb{N}$ such that $2^{-M} < \tilde C/4$.   
Let $\psi$ denote the zero function and consider $\Lambda_{\Gamma_M}(\psi)$. Note that for each $f_m \in \Lambda_{\Gamma_M}(\psi)$ there is a $t_m \in \mathbb{Q}$ such that $f_m(\psi) = \psi(t_m)$. Let $K = \max\{t_m \, \vert \, f_m(\psi) = \psi(t_m), f_m \in \Lambda_{\Gamma_M}(\psi)\}$. Choose $k$ so large that $K < k$. Then, by the choice of the support of $\varphi_k$ we have that $f_m(\psi) = f_m(\phi_k)$ for all $f_m \in \Lambda_{\Gamma_M}(\psi)$. Hence, by assumption (ii) and (iii) in Definition \ref{alg} it follows that $\Gamma_{M}(\psi) = \Gamma_{M}(\varphi_k)$.
Thus, since $\|\Gamma_{M}(\psi)-\Xi^2_{\mathrm{free}}(\psi)\|_{L^2(\mathbb{R})} \leq 2^{-M}$, we have that $\|\Gamma_{M}(\psi)\|_{L^2(\mathbb{R})} = \|\Gamma_{M}(\phi_k)\|_{L^2(\mathbb{R})} < \tilde C/4$. However, the time evolution of the Schr\"odinger equation preserves the $L^2$-norm, so $\|\Xi^2_{\mathrm{free}}(\varphi_k)\|_{L^2(\mathbb{R})} = \tilde C$ and $ \|\Gamma_{M}(\varphi_k)-\Xi^2_{\mathrm{free}}(\varphi_k)\|_{L^2(\mathbb{R})} \leq 2^{-M} < \tilde C/4$ hence $ \|\Gamma_{M}(\phi_k)\|_{L^2(\mathbb{R})}> \tilde C/4$ establishing the contradiction. 

 To prove part III we define $\Omega_{\mathrm{free}} = \{\varphi_s \, \vert \, s \in \mathbb{N}, \varphi_s \text{ is given by \eqref{eq:the_phi}} \},$ 
\begin{equation}\label{eq:the_phi}
 \varphi_s(x) = \frac{e^{-(x+sT)^2/(4\nu) + is\left(x+\frac{sT}{2} \right)}}{\left(2\pi \nu^2 \right)^{1/4}}, \quad \nu=\left(1 - \frac{iT}{2} \right), \quad s \in \mathbb{N}.
\end{equation}
Recall that 
\[
\left\lVert f \right\rVert^2_{H^\rho_{\eta}}:= \|\langle\bullet \rangle^{\rho} \widehat{f} \|_{L^2}^2+  \left\lVert  \langle\bullet \rangle^{\eta} f \right\rVert_{L^2}^2.
\]
Thus, it follows from the basic properties of the Fourier transform that $\|\langle\bullet \rangle^{\rho} \widehat{\varphi_s}\|_{L^2}^2 \rightarrow \infty$ as $s \rightarrow \infty$ for $\rho > 0$, and $\|\langle\bullet \rangle^{\eta}\varphi_s\|_{L^2}^2 \rightarrow \infty$ as $s \rightarrow \infty$, for $\eta > 0$ follows immediately from the definition of $\varphi_s$. Hence, $\sup\{\|\varphi \|_{H^\rho_{\eta}} \, \vert \, \varphi \in \Omega_{\mathrm{free}}\} = \infty$ for all $(\rho, \eta) \in \mathbb{R}^2_+ \setminus \{(0,0)\}$. 
In order to describe the algorithm we first note that $\Lambda$ is the same as in \eqref{eq:the_Lambda}. 
Hence, we may consider $f_m \in \Lambda$, where $t_m = 0$. Define, for $k,l \in \mathbb{N},$ $\mathrm{Log}_{k,l}: (0,\infty) \rightarrow \mathbb{R}$ to be a recursive functions such that $|\mathrm{Log}_{k,l}(x) - \log(x)| \leq 2^{-k}$ for $x \in [1/l,l]$. The existence of such functions are well known (take one of the many series expansion formulas for the $\log$ function for example). Given $\varphi \in \Omega_{\mathrm{free}}$, we can access $f_m(\varphi)$ and thus compute $a = (|f_m(\varphi)|^4 (2\pi |\nu|^2))^{-1}$. Choose $k,l$ such that $2^{-k} \frac{|v|^2}{T^2}\leq 1/2$ and $a \in [1/l,l]$. It is clear that choosing $k,l$ can be done recursively from $f_m(\varphi)$ and $T$. 
Now let $p \in \mathbb{N}$ be such that 
\begin{equation}\label{eq:the_p}
\left|p^2 - \mathrm{Log}_{k,l}(a)\frac{|\nu|^2}{T^2}\right| \leq \frac{1}{2}.
\end{equation}
We claim that $p$ is unique, that it can be computed recursively from $f_m(\varphi)$ and $T$, and that $\Xi_{\mathrm{free}}(\varphi) = \varphi_p$ where $\varphi_p$ is defined in \eqref{eq:the_phi}. Indeed, note that since $\varphi \in \Omega_{\mathrm{free}}$ we have that $\varphi = \varphi_s$ for some $s \in \mathbb{N}$ where $\varphi_s$ is defined in \eqref{eq:the_phi}. Then, 
\[
\log(a) = \log\left(|f_m(\varphi_s)|^4 (2\pi |\nu|^2))^{-1}\right) = \frac{(sT)^2 }{|\nu|^2}, \quad |\log(a) - \mathrm{Log}_{k,l}(a)| \leq 2^{-k},   
\]
thus
\[
\left|s^2 - \mathrm{Log}_{k,l}(a)\frac{|\nu|^2}{T^2}\right| \leq \frac{1}{2}, \text{ since } 2^{-k} \frac{|v|^2}{T^2}\leq 1/2.
\]
Hence, by \eqref{eq:the_p}, since $p \in \mathbb{N}$, we have that $p=s$ (thus it is unique) and 
\[
\Xi_{\mathrm{free}}(\varphi) = \frac{e^{-\frac{x^2}{4} + ipx}}{(2\pi )^{1/4}} =: \psi_p(x).
\]
The latter follows, as above, from basic properties of the Fourier transform (see \cite[Sec.$7.3$]{Te14}).
 The claim about recursiveness follows from the fact that $k,l$ were constructed recursively and that $\mathrm{Log}_{k,l}$ is recursive. As we have been able to recursively compute the $p \in \mathbb{N}$ such that  $\Xi_{\mathrm{free}}(\varphi) = \psi_p$, it is now easy to finally establish the algorithm $\Gamma$, and we will be a bit brief regarding the details, as this part is a routine exercise. Let $\Gamma(\epsilon, \varphi)$ be the vector $\{a_k\}$ of coefficients of a sum of step functions $f = \sum_k a_kg_k$ (the $g_k$s are step functions) where the $a_k$s are approximations to $\psi_p(x)$  for different values of $x \in \mathbb{R}$ such that $\|f - \psi_p\|_{L^2(\mathbb{R})} \leq \epsilon$. Constructing the $a_k$s is clearly recursive as $\psi$ is recursive. Moreover, choosing the gridsize for the step functions can be done in a recursive way from $\epsilon$, $f_m(\varphi_s)$, however, we omit the details. 
\end{proof}

\section{Proof of Theorem \ref{BigTh:main_thrm2}}


\subsection{Proof of Theorem \ref{BigTh:main_thrm2} with Assumption \ref{ass:1}}

\begin{proof}[Proof of Theorem \ref{BigTh:main_thrm2} with Assumption \ref{ass:1}]
 We will construct a sequence of algorithms $\{\Gamma_n\}$ such that 
\begin{equation}\label{eq:total_bound}
\left \| \Gamma_n(\varphi_0,V) - \psi(\bullet,T) \right\|_{L^2} \leq 2^{-n},
\end{equation}
where $\psi$ is the solution to the linear Schr\"odinger equation
\begin{equation}\label{eq:schr}
\begin{split}
i \partial_t \psi(x,t) &= (-\Delta+V +V_{\operatorname{TD}}(t))\psi(t), \quad \psi(0)=\varphi_0.
\end{split}
\end{equation}

{\bf Step I:} (Choosing $R$). We begin by choosing an $R > 0$ to restrict our equation to the cube centred at zero with length $R$. To do that we begin with Theorem \ref{theo2} and its proof which implies that for any $R > 0$ there is a smooth cut-off function $\gamma_R$, based on exponentials which parameters are recursive in $R$, that is supported on $\mathcal B_R(0)$ such that 
\begin{equation}\label{eq:bound_phi}
\Vert \varphi_0\gamma_R \Vert_{H^{2+\varepsilon}_2} \leq f_1(M, R), 
\end{equation}
where $f_1$ is recursive and $\Vert \varphi_0 \Vert_{H^{2+\epsilon}_2} \leq M$.
 Moreover, if $\psi^{D}_R$ is the solution of \eqref{eq:schr} on $\mathcal B_R(0)$, with zero Dirichlet boundary conditions and initial state $\varphi_0\gamma_R$,  and $\xi = \psi - \psi^{D}_R$, then \begin{equation*}
\left\lVert \xi  \right\rVert^2_{L^{\infty}((0,T),L^2(\RR^d))} \lesssim \sup_{t \in (0,T)}\left\lvert \int_{\partial \mathcal B_R(0)} (\nabla_n \xi)(x,t) \overline{\xi(x,t)}  \ dS(x) \right\rvert + R^{-2} C(T) \Vert \varphi_0 \Vert_{H^2_2(\RR^d)},
\end{equation*}
and by Lemmas \ref{eos}, \ref{eos2},\ref{lemm:auxlemmes}, we then have that for some $D_1>0$
\begin{equation}\label{eq:bound_xi}
\begin{split}
\left\lVert \xi  \right\rVert^2_{L^{\infty}((0,T),L^2(\RR^d))} &\leq D_1
\sup_{t \in (0,T)} \frac{\left\lVert  \xi(\bullet, t) \right\rVert^2_{H^2_2}}{R}
+ R^{-2} C(T) \Vert \varphi_0 \Vert_{H^2_2(\RR^d)},\\
&\leq D_1 2 C(T)  \frac{\left\lVert  \varphi_0 \right\rVert^2_{H^2_2}}{R}
+ R^{-2} C(T) \Vert \varphi_0 \Vert_{H^2_2(\RR^d)},
\end{split}
\end{equation}
where 
\[
C(T) = C(T,\Vert u \Vert_{W^{1,1}_{\operatorname{pcw}}((0,T))},\Vert W_{\operatorname{sing}} \Vert_{L^p}, \Vert \langle \bullet \rangle^{-2} W_{\operatorname{reg}} \Vert_{L^{\infty}}, \Vert \langle \bullet \rangle^{-2} V_{\operatorname{con}} \Vert_{L^{\infty}})
\]
is a recursive function. 
Hence, using the fact that for $(\varphi_0,V) \in \Omega_{\operatorname{Lin}}$ as in Assumption \ref{ass:1} we have that $\Vert \varphi_0 \Vert_{H^2_2} \le C_1$ we can, by \eqref{eq:outside}, choose $R > 0$ such that 
\begin{equation}\label{eq:first_bound_R}
D_1 2 C(T)  \frac{C_1^2}{R}
+ R^{-2} C(T) C_1 \leq 2^{-n},
\end{equation}
and thus we can focus on computing an approximation to $\psi^{D}_R$. 
From now on and throughout the argument $\Omega := \mathcal{C}_R(0)$, and it is clear that the bounds above for the ball will apply for the cube.  Note that, since the parameters in $\gamma_R$ are determined recursively and $\gamma_R$ is based on the exponential function, we can evaluate recursively point samples of $\varphi_0\gamma_R$ from point samples of $\varphi_0$. 

{\bf Step II:} (Removing singularities).
We now need to deal with the singular potential $W_{\operatorname{sing}}$ as this potential will have to be approximated by something bounded in order to do a discretisation.
By assumption we have that $W_{\operatorname{sing}} \in L^p \cap W^{2,\infty}_{\mathrm{loc}}(\mathbb{R}^d \setminus \cup_j\{x_j\})$ with singularities $\{x_j\}_{j=1}^{\infty} \subset \mathbb R^d$ has controlled singularity-blowup by $f: (0,\epsilon_0) \times \mathbb{R}_+ \rightarrow \mathbb{Q}_+ \times \mathbb{Q}_+$. Then, let $(\delta, L) = f(2^{-(n+2)}, R)$. Then, $\|W_{\operatorname{sing}}\chi_{\mathcal{A}} - W_{\operatorname{sing}}\chi_{\mathcal{C}_R(0)}\|_{L^p} \leq 2^{-(n+2)},$ where 
\[
\mathcal{A} = \mathcal{A}(\delta,R) = \bigcup_{x_j \in \mathcal{C}_R(0)} \mathcal{C}_\delta(x_j)^c \cap \mathcal{C}_R(0).
\]
 Moreover, $\|W_{\operatorname{sing}}\big\vert_{\mathcal{A}}\|_{W^{q,\infty}(\mathcal A)} \leq L$. 
Hence, given $(\delta, L) = f(2^{-(n+2)}, R)$, and the fact that $f$ is recursive, it is easy to see that one can obtain finitely many cut-off functions based on exponentials with parameters that are recursive in the variables $n$,$R$ and $\{x_j\} \cap \Omega$ such that if $\zeta$ is the sum of these functions then $\zeta(x) = 0$ for $x \in \Omega^c$ and  
\begin{equation}\label{eq:bound_W}
\|W_{\operatorname{sing}}\zeta - W_{\operatorname{sing}}\chi_{\Omega}\|_{L^p} \leq 2^{-(n+1)}, \qquad 
\|W_{\operatorname{sing}}\zeta\|_{W^{q,\infty}} \leq \hat C(L),
\end{equation}
where $\hat C$ can be obtained recursively from $L$. 
Note that, since $\zeta$ is a finite sum of exponentials, it can be recursively evaluated at any rational point to any precision. Hence, by Lemma \ref{potlem}, we can replace $W_{\operatorname{sing}}\chi_{\Omega}$ by $W_{\operatorname{sing}}\zeta$ and the point samples of $W_{\operatorname{sing}}\chi_{\Omega}$ needed later in the construction of $\Gamma_n$ can be determined recursively from point samples of $W_{\operatorname{sing}}$.
We can therefore continue with the problem of computing $\psi^{D}_R$. By using the assumption that $W_{\operatorname{reg}}, V_{\operatorname{con}} \in  W^{2,\infty}_{\mathrm{loc}}$ have controlled local smoothness by $g$ we can now, by using \eqref{eq:bound_W}, assume that we have a potential on $\Omega$ of the form $V + u(t)V_{\operatorname{con}}\big\vert_{\Omega}$, where 
\begin{equation}\label{eq:bound_C_of_n}
\Vert u \Vert_{W^{1,1}_{\operatorname{pcw}}}, \Vert V\Vert_{W^{2,\infty}},\Vert V_{\operatorname{con}}\big\vert_{\Omega} \Vert_{W^{2,\infty}(\Omega)} \leq \tilde C,
\end{equation}   
where the bound $\tilde C$ can be constructed recursively from the integer $n$ determining the accuracy, and where, with slight abuse of notation, $V = W_{\operatorname{sing}}\zeta + W_{\operatorname{reg}}\big\vert_{\Omega}$. To simplify the notation below we will omit the restrictions. 

{\bf Step III:} (Choosing the gridsize $h$ in the discretisation). In order to discretise the initial state $\varphi_0\gamma_R$ we recall Definition \ref{cubdisc} of the function $(\varphi_0\gamma_R)_Q$ that is a sum of characteristic functions according to the lattice depending on a step size $h$ in Definition \ref{cubdisc}. We will simply use the $(\varphi_0\gamma_R)_Q$ notation keeping the dependence of $h$ in mind. The size of this $h$ will be chosen at the very end.   
Now define the Strang splitting scheme with $\mathcal C^{\tau}_h$ being the Crank-Nicholson method defined in \eqref{eq:discrule} and \eqref{eq:the_C}, as follows:
\begin{equation}
\begin{split}
\label{eq:Stranglin22}
&(1) \quad \varphi^-_{k+\tfrac{1}{2}} := \mathcal C^{\tau/2}_h \varphi^{\operatorname{lin}}_{\operatorname{S}}(t_k), \\
&(2) \quad \varphi^+_{k+\tfrac{1}{2}}:=\exp_K\left(-i \tau (X_{t_k}^{\operatorname{lin}})_Q \right) \varphi_{k+\tfrac{1}{2}}^-, \operatorname{ and }\\
&(3) \quad \varphi^{\operatorname{lin}}_{\operatorname{S}}(t_{k+1}) :=\mathcal C^{\tau/2}_h \varphi^+_{k+\tfrac{1}{2}}
\end{split}
\end{equation}
where $\exp_K$ is defined in \eqref{eq:expK}, and 
\begin{equation}\label{eq:initial}
\varphi^{\operatorname{lin}}_{\operatorname{S}}(0):=(\varphi_0 \gamma_R)_Q, \quad (X_{t_k}^{\operatorname{lin}})_Q:=V_{\operatorname{TD}}(t_k)_Q.
\end{equation} 
By Proposition \ref{prop:lintime} there are recursive maps 
\begin{equation}\label{eq:rec1}
\begin{split}
&T, \Vert \varphi_0 \gamma_R\Vert_{H^{2+\varepsilon}},\Vert V \Vert_{W^{2,\infty}},\Vert V_{\operatorname{con}} \Vert_{W^{2,\infty}}, h,\Vert u \Vert_{W^{1,1}_{\operatorname{pcw}}}\\
& \qquad \mapsto \tau(T,\Vert \varphi_0 \gamma_R\Vert_{H^{2+\varepsilon}},\Vert V \Vert_{W^{2,\infty}},\Vert V_{\operatorname{con}} \Vert_{W^{2,\infty}}, h,\Vert u \Vert_{W^{1,1}_{\operatorname{pcw}}})
\end{split}
\end{equation} 
determining the stepsize $\tau$ in \eqref{eq:Stranglin22}, and 
\begin{equation}\label{eq:rec2}
\begin{split}
&T,\tau, \Vert \varphi_0\gamma_R \Vert_{H^{2+\varepsilon}},\Vert V \Vert_{W^{2,\infty}},\Vert V_{\operatorname{con}} \Vert_{W^{2,\infty}},  h,\Vert u \Vert_{W^{1,1}_{\operatorname{pcw}}}  \\
& \qquad \mapsto K(T,\tau, \Vert \varphi_0 \gamma_R\Vert_{H^{2+\varepsilon}},\Vert V \Vert_{W^{2,\infty}},\Vert V_{\operatorname{con}} \Vert_{W^{2,\infty}}, h,\Vert u \Vert_{W^{1,1}_{\operatorname{pcw}}})
\end{split}
\end{equation}
determining $K$ in \eqref{eq:Stranglin22} such that 
\[
 \max_{k} \Vert \psi^{D}_R(\tau k) - \varphi^{\operatorname{lin}}_{\operatorname{S}}(\tau k) \Vert_{L^2} \le C_{T} \Vert \varphi_0 \gamma_R -(\varphi_0 \gamma_R)_Q \Vert_{L^2}.
 \]
Note that \eqref{eq:bound_phi} and \eqref{eq:bound_C_of_n} provide bounds, that are recursively defined in the input $n$, on the input needed for the mappings in \eqref{eq:rec1} and \eqref{eq:rec2}. Also, noting that 
 by Proposition \ref{convrate} we have that for $f \in W^{1,2}(\mathbb R^d)$  
\[
\left\|f_Q-f\right\rVert_{L^2(\mathbb R^d)} \leq C'\Vert f\Vert_{W^{1,2}(\RR^d)} h,
\] 
for some constant $C'$, and by using \eqref{eq:bound_phi}, we can deduce that we can recursively compute
\begin{equation}\label{eq:recursive_h}
h \text{ (from $n$) such that } \max_{k} \Vert \psi^{D}_R(\tau k) - \varphi^{\operatorname{lin}}_{\operatorname{S}}(\tau k) \Vert_{L^2} \le 2^{-(n+1)}.
\end{equation}

{\bf Step IV:} (Choosing $N$ in the integration). To finalise the proof we need to approximate the initial state $(\varphi_0 \gamma_R)_Q$ and the potentials \MA{$V_Q$} and 
$V_{\operatorname{TD}}(t_k)_Q$ in \eqref{eq:Stranglin22} and \MA{\eqref{eq:discrule}} with the numerical integration from Definition \ref{cubdisc}. In particular, we replace (2) in \eqref{eq:Stranglin22} by 
\begin{equation*}
(2') \quad \varphi^+_{k+\tfrac{1}{2}}:=\exp_K\left(-i \tau (X_{t_k}^{\operatorname{lin}})_{Q,N} \right) \varphi_{k+\tfrac{1}{2}}^-, 
\end{equation*}
and \eqref{eq:initial} gets changed to 
\begin{equation*}
\varphi^{\operatorname{lin}}_{\operatorname{S}}(0):=(\varphi_0)_{Q,N} \quad (X_{t_k}^{\operatorname{lin}})_{Q,N}:=V_{\operatorname{TD}}(t_k)_{Q,N},
\end{equation*} 
where we recall that for $\phi \in L^1_{\operatorname{loc}}(\mathbb R^d,\mathbb C)$, $\phi_{Q,N}$ defined in Definition \ref{def:numerical_disc}.  Let $ \tilde  \varphi^{\operatorname{lin}}_{\operatorname{S}}(\tau k) $ denote the outputs of this modified scheme. 
Then, by Proposition \ref{prop:approx_int}, it follows that for $k=\lceil T/\tau \rceil$, $R$ and $h$
\begin{equation}\label{eq:bound1}
\begin{split}
 \left\Vert  \varphi^{\operatorname{lin}}_{\operatorname{S}}(\tau k) - \tilde \varphi^{\operatorname{lin}}_{\operatorname{S}}(\tau k) \right\Vert_{L^2} &\le C(n, R, h) \big( \left\Vert V_Q- V_{Q,N} \right\rVert_{L^{\infty}} \\
& \, \vee \left\Vert V_{\operatorname{con},Q}- V_{\operatorname{con},Q,N} \right\rVert_{L^{\infty}} \vee \left\lVert (\varphi_0)_{Q}-(\varphi_0)_{Q,N} \right\Vert_{L^2} \big),
 \end{split}
 \end{equation}
where the mapping $n, R, h \mapsto C(n, R, h)$ is recursive. 
Note that by Proposition \ref{prop:bound_disc_int} we have that 
\begin{equation*}
\begin{split}
&\| V_{Q} - V_{Q,N} \|_{L^{\infty}} \vee \left\Vert V_{\operatorname{con},Q} - V_{\operatorname{con},Q,N}  \right\rVert_{L^{\infty}}\\
& \qquad \leq \left(\max_{x_j \in \Omega} \left(\mathrm{TV}\left(V \circ \rho_{x_j} \right) \vee \mathrm{TV}\left(V_{\operatorname{con}} \circ \rho_{x_j} \right) \right)\right) C^*(b_1,\hdots,b_d)\frac{\log(N)^d}{N}.
\end{split}
\end{equation*}
as well as 
\begin{equation*}
\begin{split}
&\| (\varphi_0 \gamma_R)_{Q}-(\varphi_0 \gamma_R)_{Q,N}  \|_{L^2(\mathbb{R}^d)} \\
& \quad \leq \left(\sum_{x_j \in \Omega} \left(\mathrm{TV}\left( \varphi_0 \gamma_R\circ \rho_{x_j} \right) \right)^2\right)^{\frac{1}{2}} C^*(b_1,\hdots,b_d)\frac{\log(N)^d}{N}.
\end{split}
\end{equation*}
Thus, we need to bound the total variation of $V, V_{\operatorname{con}}$ and $\varphi_0 \gamma_R$. Note that it is well known that for $a > 0$ and $\lambda = 3^d-2^{d+1} +2$, we have
\begin{equation}\label{eq:TV_bounds}
\begin{split}
\mathrm{TV}_{[-a,a]^d}(\varphi_0 \gamma_R) &\leq \|\varphi_0\|_{L^{\infty}}\|\gamma_R\|_{L^{\infty}}  + \lambda^2\mathrm{TV}_{[-a,a]^d}(\varphi_0)\mathrm{TV}_{[-a,a]^d}( \gamma_R)\\ 
& \quad +  \lambda\left(\mathrm{TV}_{[-a,a]^d}(\varphi_0)\|\gamma_R\|_{L^{\infty}} + \mathrm{TV}_{[-a,a]^d}(\gamma_R)\|\varphi_0\|_{L^{\infty}} \right).
\end{split}
\end{equation}
 
Hence, it is easy to see, by using \eqref{eq:bound_phi}, \eqref{eq:TV_bounds} and (i) in Assumption \ref{ass:1} (asserting that $\varphi_0$ has controlled local bounded variation by $\omega$), as well as the standard bounds of the total variation in terms of the Jacobian, that there is a recursive mapping $ \mathbb{N} \times \mathbb{Q}_+ \times \mathbb{N} \ni (R,h,K) \mapsto \tilde C(R,h,K)$ such that for  $\|V\|_{W^{2,\infty}} \vee \|V_{\operatorname{con}}\|_{W^{2,\infty}} \vee \| \varphi_0\|_{L^{\infty}}    \leq K$ we have
\begin{equation}\label{eq:bound2}
\begin{split}
\| V_{Q} - V_{Q,N} \|_{L^{\infty}} \vee & \left\Vert V_{\operatorname{con},Q} - V_{\operatorname{con},Q,N}  \right\rVert_{L^{\infty}} \vee \| (\varphi_0 \gamma_R)_{Q}-(\varphi_0 \gamma_R)_{Q,N}  \|_{L^2} \\
& \quad \leq \tilde C(R,h,K)C^*(b_1,\hdots,b_d)\frac{\log(N)^d}{N}.
\end{split}
\end{equation}
Thus, by \eqref{eq:bound1} and \eqref{eq:bound2} we have established a recursive way of computing $N$ such that 
\begin{equation}\label{eq:almost_last_bound}
\left\Vert  \varphi_{S}(\tau k) - \tilde  \varphi_{S}(\tau k) \right\Vert_{L^2} \leq 2^{-n}.
\end{equation}

{\bf Step V:} (Showing recursiveness). Let $\Gamma_n(\varphi_0,V) = \tilde  \varphi_{S}(T)$. Then, the desired bound \eqref{eq:total_bound} follows from \eqref{eq:almost_last_bound}, \eqref{eq:recursive_h}, \eqref{eq:first_bound_R} and \eqref{eq:bound_xi}. 
The only thing left to prove is that the mapping $n \mapsto \Gamma_n(\varphi_0,V)$ is recursive. Note that we have already shown how $R$, $h$ and eventually $N$ can be determined in a recursive way from the input. Thus, we only need to show that the execution of $\tilde  \varphi_{S}(T)$ can be done with finitely many arithmetic operations and comparisons. The numerical approximation from Definition \ref{cubdisc} is requires only arithmetic operations. Moreover, so does executing the scheme \eqref{eq:Stranglin22} as the Crank-Nicholson method requires only matrix vector multiplication and solution of linear systems. 

The only thing we have left out is that the numerical integration of from Definition \ref{cubdisc} assumes that we can sample $W_{\operatorname{sing}}\zeta$ and $\varphi_0 \gamma_R$ exactly at the points in the Halton sequence (which consists of rational numbers). However, $\gamma_R$ and $\zeta$, that are cut-off functions based on exponentials can only be evaluated approximately to arbitrary precision. This extra layer of approximation can easily be added using \eqref{eq:bound1}, however, we omit this elementary and obvious exercise, thus finalising the proof. 

{\bf Step VI:} (Uniform runtime).The uniform runtime of the algorithm follows since all steps in the above algorithm only depend on uniform properties of elements in the set $\Omega_{\operatorname{Lin}}.$
\end{proof}

\subsection{Proof of Theorem \ref{BigTh:main_thrm2} with Assumption \ref{ass:2}}

\begin{proof}[Proof of Theorem \ref{BigTh:main_thrm2} with Assumption \ref{ass:2}]
The strategy of the proof is to convert the problem with Assumption \ref{ass:2} to a problem with Assumption \ref{ass:2} and then follow the proof of Theorem \ref{BigTh:main_thrm2} with Assumption \ref{ass:1} almost verbatim. 
 In particular, to obtain $\{\Gamma_m\}$ such that  
 $\| \Gamma_m(\varphi_0,V) - \psi(\bullet,T) \|_{L^2} \leq 2^{-m+1},$
 we begin by preparing the setup so that we can use some of the steps in the proof of Theorem \ref{BigTh:main_thrm2} with Assumption \ref{ass:1} verbatim. To prove the statement that the runtime of $\Gamma_m$ has a uniform upper bound for all inputs $\varphi_0,V$ we will argue as follows. The final definition of $\Gamma_m$ may be viewed as a collection of subroutines defined through several steps. We will argue that the runtime needed in each step is uniformly bounded for all inputs.

{\bf Step I:} (Perturbation theory).
Let $\varphi(t;\varphi_0)$ denote the solution to \eqref{eq:schr} with initial state $\varphi_0$. For the linear Schr\"odinger equation, it follows straight from the variation of constant formula
\[ \varphi(t;\varphi_0) = e^{i \Delta t} \varphi_0 + \int_0^t  e^{i \Delta (t-s)}(V+u(s) V_{\operatorname{control}}) \varphi(s;\varphi_0) \ ds \]
that there exists a recursively defined function $K(C;T)$ such that 
\begin{equation}
\begin{split}
\label{eq:Schroedinger}
&\sup_{t \in (0,T)} \Vert \varphi(t;\varphi_0)  -  \varphi(t;\widetilde \varphi_0) \Vert_{L^2} \le K(C,T) \Vert \varphi_0 - \widetilde \varphi_0\Vert_{L^2},
\end{split}
\end{equation}
where $C$ is the constant from Assumption \ref{ass:2} and $T$ is the final time. 
We therefore start by smoothing out the initial state in order to use the techniques in the proof of Theorem \ref{BigTh:main_thrm2} with Assumption \ref{ass:1} that needs smoothness. 
 
{\bf Step II:} (Smoothing out initial states). 
To simplify the notation we display the technique in one dimension. The multi-dimensional case follows immediately from the one-dimensional setup by using tensor products of one-dimensional functions.  
To approximate the initial state by a smoother state in $L^2(\RR)$, consider the following $L^2(\RR)$ orthonormal system
\begin{equation}\label{eq:eig_func}
 \psi_n(x) = \frac{e^{-x^2}}{\sqrt{2^n n!} \pi^{1/4}}, \quad n \in \mathbb N_0.
 \end{equation}
These Schwartz functions are eigenfunctions to the quantum harmonic oscillator
$
S = -\Delta + \vert x \vert^2 
$
to eigenvalues $\lambda_n:=(2n+1)$\footnote{In higher dimensions $d$, the eigenstates are just the tensor products of the $1d$-eigenfunctions. Correspondingly, the eigenvalues are just $(2n+d)$ for $n \in \mathbb N_0$ where the $n$-th eigenvalue is $\binom{d+n-1}{n}$-fold degenerate.}.

For the linear Schr\"odinger equation, we assume that the initial state $\varphi_0$ has CLBV and is in the space $H^{2+\varepsilon}_{2}.$ The condition on the Sobolev can be relaxed to the condition that $\Vert S^{\varepsilon} \varphi_0\Vert \le C$ for some explicit constant $C>0$ and some $\varepsilon>0.$ 
To see this, we recall that since $(\varphi_n)$ forms an orthonormal system, it follows that 
$
\varphi_0= \sum_{n=0}^{\infty} \langle \psi_n, \varphi_0 \rangle_{L^2} \psi_n.
$
Hence, we find that
\[
\Vert S^{\varepsilon}\varphi_0 \Vert^2_{L^2}= \sum_{n=0}^{\infty} \lambda_n^{2\varepsilon} \lvert \langle \psi_n, \varphi_0 \rangle_{L^2}\rvert^2.
\] 
To compute an $L^2$ approximation of error at most $\delta$ we have to find $K>0$ such that $\sum_{n=K}^{\infty}  \lvert \langle \psi_n, \varphi_0 \rangle_{L^2}\rvert^2<\delta.$ 
Hence,
\[\Vert S^{\varepsilon}\varphi_0 \Vert^2_{L^2} \ge \sum_{n=K}^{\infty} \lambda_n^{2\varepsilon} \lvert \langle \psi_n, \varphi_0 \rangle_{L^2}\rvert^2 \ge \lambda_K^{2\varepsilon} \sum_{n=K}^{\infty}  \lvert \langle \psi_n, \varphi_0 \rangle_{L^2}\rvert^2 \]
which implies that 
$
\sum_{n=K}^{\infty}  \lvert \langle \psi_n, \varphi_0 \rangle_{L^2}\rvert^2 \le \frac{C}{\lambda_K^{2\varepsilon}} <\delta
$
for some explicit $K$ large enough.
Hence, it is clear that one can choose $K$ recursively from $C$ and (rational $\delta$) such that 
\begin{equation}\label{eq:phi_approximation}
\widetilde{\varphi_0}:=\sum_{n=0}^{K-1} \langle \psi_n, \varphi_0 \rangle_{L^2} \psi_n
\end{equation}
approximate the initial state $\varphi_0$ up to an error 
$
\left\Vert \varphi_0- \widetilde{\varphi_0} \right\Vert_{L^2} \le \delta.
$

All inner products $\langle \varphi_0, \psi_n \rangle$ can be computed to arbitrary precision by the assumption that $\varphi_0$ has controlled local bounded variation by $\omega: \mathbb{R}_+ \rightarrow \mathbb{N}$ the explicit decay and regularity estimates on the eigenstates $(\psi_n)$ in \eqref{eq:eig_func}. The computations are taken care of in Step VI. 

It is well-known \cite[Lemm. $24$]{BDLR} that there exist for $k \in [0,\infty)$ universal constants $d_{k}>0$
such that 
$
 \left\lVert \psi \right\rVert_{H^{k}_{k}}\le d_{k}\Vert S^{k} \psi\Vert.
$
Hence, we find a new explicit estimate on the $H^4_4(\RR^d)$ norm of the new initial state
\[ \left\lVert \widetilde{\varphi_0} \right\rVert_{H^{4}_{4}}\le d_{4}\Vert S^2 \widetilde{\varphi_0} \Vert_{L^2} = \sqrt{\sum_{n=0}^{N-1}\lambda_n^{2}\lvert \langle \psi_n, \varphi_0 \rangle_{L^2}\rvert^2} .\] 

To obtain an estimate on the total variation of $\widetilde{\varphi_0}$, it suffices to bound the $L^1$ norm of the gradient of $\widetilde{\varphi_0}$ by 

\begin{equation}\label{eq:bounded_var_bound}
 \Vert \nabla \widetilde{\varphi_0} \Vert_{L^1(B_R(0))} \le C_R \Vert \nabla \widetilde{\varphi_0} \Vert_{L^2(B_R(0))} \le C_R \left\lVert \widetilde{\varphi_0} \right\rVert_{H^{2}_{2}}.
\end{equation}

{\bf Step III:} (Choosing $R$ ). Since the initial state is now smooth, we may choose 
\[R(T, \Vert u \Vert_{W^{1,1}_{\operatorname{pcw}}(0,T)},\Vert \langle \bullet \rangle^{-2}  V \Vert_{L^{\infty}},\Vert \langle \bullet \rangle^{-2}  V_{\operatorname{con}} \Vert_{L^{\infty}})\]
exactly as in Step I in the proofs of Theorem \ref{BigTh:main_thrm2} with Assumption \ref{ass:1}.

 {\bf Step IV:} (Smoothing out potentials).
 We explain how to smoothen out potentials $V,V_{\operatorname{con}}$ without increasing the $L^{\infty}$ norm such that the cut-off radius $R$ remains still applies. Consider the Gaussian distribution 
\[
\chi_{\sigma}(x) = \frac{e^{-\frac{x^2}{2\sigma^2}}}{\sqrt{2\pi} \sigma}
\]
 with variance $\sigma^2>0$. We can then define the smooth potentials
$V_{\sigma} = V*\chi_{\sigma}$ and $(V_{\operatorname{con}})_{\sigma} =  V_{\operatorname{con}}*\chi_{\sigma}.$
By \eqref{eq:citeme} in Lemma \ref{potlem}, it suffices to approximate potentials by smooth ones in $L^p$ where $p$ is as in Remark \ref{rem:bounded}.
In particular, we may choose $\sigma$ such that the error of the following expression becomes as small as required to approximate \eqref{eq:Schroedinger} up to the desired accuracy with the new potentials $V_{\sigma}$ and $ (V_{\operatorname{con}})_{\sigma}$ as follows: 
From Minkowski's integral inequality, we find using Proposition \ref{convrate} for some fixed constant $C>0,$ and $\delta>0$ arbitrary
\begin{equation}\label{eq:conv_bound}
\begin{split}
 \Vert V- V_{\sigma} \Vert_{L^p} 
&\le \int_{\vert y \vert > \delta} \left( \int_{\RR^d} \vert V(x-y)-V(x) \vert^p  \vert \chi_{\sigma}(y) \vert^p \ dx  \right)^{1/p} \ dy\\
&+\int_{\vert y \vert \le \delta} \chi_{\sigma}(y)  \left(\int_{\RR^d} \vert V(x-y)-V(x) \vert^p   \ dx  \right)^{1/p} \ dy \\
& \le 2 \Vert V \Vert_{L^p} \int_{\vert y \vert >\delta}   \chi_{\sigma}(y) \ dy + \int_{\vert y \vert \le \delta}\chi_{\sigma}(y) \sup_{\vert y \vert \le \delta} \Vert V(\bullet-y)-V \Vert_{L^p} \ dy \\
&\le 2 \Vert V \Vert_{L^p} \operatorname{erfc}\left(\frac{\delta}{\sqrt{2}\sigma} \right)+ \sup_{\vert y \vert \le \delta} \Vert V(\bullet-y)-V \Vert_{L^p}\\
&\le 2 \Vert V \Vert_{L^{\infty}} \vert B(0,R+1) \vert^{1/p} \operatorname{erfc}\left(\frac{\delta}{\sqrt{2}\sigma} \right) + C \Phi(R+1) \delta^{\varepsilon}.
\end{split}
\end{equation}
Moreover, for every $n \in \mathbb N_0$ there are explicit estimates on the $n$-th derivative by Young's inequality
\begin{equation}\label{eq:V_bound}
\Vert V_{\sigma}^{(n)} \Vert_{L^{\infty}} \le \Vert V \Vert_{L^{\infty}} \Vert \chi_{\sigma}^{(n)} \Vert_{L^1}.
\end{equation}

By \eqref{eq:citeme} in Lemma \ref{potlem} and \eqref{eq:conv_bound} it follows that we can choose $\sigma$ recursively from $n \in \mathbb{N}$ and $R \in \mathbb{N}$ such that the solution to the Schr\"odinger equation \eqref{eq:schr} with initial state $\widetilde \varphi_0$ does not differ in $L^2$ norm more than $2^{-n}$ from the solution to \eqref{eq:schr} with the potentials $V_{\sigma}$, $(V_{\operatorname{con}})_{\sigma}$. Thus, the same holds if we restrict the problem to the $R$-cube. We will denote the restricted potentials of $V_{\sigma}$, $(V_{\operatorname{con}})_{\sigma}$ to the $R$-cube by $\widetilde V$, $\widetilde V_{\operatorname{con}}$ respectively. 
Given the new $\widetilde \varphi_0$, $\widetilde V$, $\widetilde V_{\operatorname{con}}$ we now have 
\begin{equation}\label{eq:bound_C_of_n5}
\Vert u \Vert_{W^{1,1}_{\operatorname{pcw}}}, \Vert \widetilde V\Vert_{W^{2,\infty}},\Vert \widetilde V_{\operatorname{con}}\big\vert_{\Omega} \Vert_{W^{2,\infty}(\Omega)}, \Vert \widetilde  \varphi_0 \Vert_{H^2_2} \leq \tilde C.
\end{equation}   
Thus, we are now having the same situation as Assumption \ref{ass:1} for our Schr\"odinger problem on the $R$-cube.

{\bf Step V:} (Choosing the gridsize $h$ in the discretisation). This is done exactly as in Step III in the proofs of Theorem \ref{BigTh:main_thrm2} with Assumption \ref{ass:1}.

{\bf Step VI:} (Choosing $N$ and $M$ in the integration). 
This part differs from the proof of Theorem \ref{BigTh:main_thrm2} with Assumption \ref{ass:1}.
To finalise the proof we need to approximate the initial state $(\widetilde \varphi_0 \gamma_R)_Q$ and the potentials $\tilde V_Q$ and 
$\tilde V_{\operatorname{TD}}(t_k)_Q$ in \eqref{eq:Stranglin22} and \eqref{eq:discrule} with the numerical integration from Definition \ref{cubdisc}. However, the numerical integration will now be slightly different. We will replace 
\begin{equation}\label{eq:function_change}
(\widetilde  \varphi_0 \gamma_R)_Q,  \quad \widetilde  V_Q, \quad \widetilde  V_{\operatorname{TD}}(t_k)_Q
\end{equation}
with functions that are, as the functions in \eqref{eq:function_change}, constant on cubes as in Definition \ref{cubdisc}. Note that if we could compute $(\widetilde  \varphi_0 \gamma_R)_{Q,N}$, $\widetilde  V_{Q,N}$, and $\widetilde  V_{\operatorname{TD}}(t_k)_{Q,N}$,  
where we recall that for $\phi \in L^1_{\operatorname{loc}}(\mathbb R^d,\mathbb C)$, $\phi_{Q,N}$ is introduced in Definition \ref{def:numerical_disc}, the rest of the argument would be identical to that of Step IV in the proof of Theorem \ref{BigTh:main_thrm2} with Assumption \ref{ass:1}. However, we can only produce approximations to $(\widetilde  \varphi_0 \gamma_R)_{Q,N}$, $\widetilde  V_{Q,N}$ and $\widetilde  V_{\operatorname{TD}}(t_k)_{Q,N}$ because of the smoothing approximations done in Step II and Step IV. Hence, $(\widetilde  \varphi_0 \gamma_R)_{Q,N}$, $\widetilde  V_{Q,N}$, and $\widetilde  V_{\operatorname{TD}}(t_k)_{Q,N}$ will be replaced by the $M$-approximations 
\[
(\widetilde  \varphi_0 \gamma_R)^{M}_{Q,N}, \quad \widetilde  V^M_{Q,N}, \quad \widetilde V_{\operatorname{TD}}(t_k)^M_{Q,N},
\]
introduced in Definition \ref{def:numerical_disc}. We will specify how these functions are chosen at the very last stage in this step. 

From here on we can now stay close to Step IV of the proof of Theorem \ref{BigTh:main_thrm2} with Assumption \ref{ass:1}.
In particular, we replace (2) in \eqref{eq:Stranglin22} by 
\begin{equation*}
(2') \quad \varphi^+_{k+\tfrac{1}{2}}:=\exp_K\left(-i \tau (X_{t_k}^{\operatorname{lin}})^M_{Q,N} \right) \varphi_{k+\tfrac{1}{2}}^-, 
\end{equation*}
and \eqref{eq:initial} gets changed to 
\begin{equation*}
\widetilde \varphi^{\operatorname{lin}}_{\operatorname{S}}(0):=(\widetilde \varphi_0 \gamma_R)^{M}_{Q,N}, \quad (X_{t_k}^{\operatorname{lin}})^M_{Q,N}:=\widetilde  V_{\operatorname{TD}}(t_k)^M_{Q,N}.
\end{equation*} 
Then, by Proposition \ref{prop:approx_int}, it follows that for $k=\lceil T/\tau \rceil$, $R$ and $h$
\begin{equation}\label{eq:bound1a}
\begin{split}
 \left\Vert  \varphi^{\operatorname{lin}}_{\operatorname{S}}(\tau k) - \widetilde  \varphi^{\operatorname{lin}}_{\operatorname{S}}(\tau k) \right\Vert_{L^2} &\le C(n, R, h) \big( \left\Vert \widetilde V_Q- \widetilde V^M_{Q,N} \right\rVert_{L^{\infty}} \\
& \, \vee \left\Vert \widetilde V_{\operatorname{con},Q}- \widetilde V^M_{\operatorname{con},Q,N} \right\rVert_{L^{\infty}} \vee \left\lVert (\widetilde \varphi_0)_{Q}-(\widetilde \varphi_0)_{Q,N} \right\Vert_{L^2} \big),
 \end{split}
 \end{equation}
where the mapping $n, R, h \mapsto C(n, R, h)$ is recursive. 
Note that by Proposition \ref{prop:bound_disc_int} we have that 
\begin{equation}\label{eq:righthand1}
\begin{split}
&\| \widetilde V_{Q} - \widetilde V_{Q,N} \|_{L^{\infty}} \vee \left\Vert \widetilde V_{\operatorname{con},Q} - \widetilde  V^M_{\operatorname{con},Q,N}  \right\rVert_{L^{\infty}}\\
& \qquad \leq \left(\max_{x_j \in \Omega} \left(\mathrm{TV}\left(\widetilde  V \circ \rho_{x_j} \right) \vee \mathrm{TV}\left(\widetilde  V_{\operatorname{con}} \circ \rho_{x_j} \right) \right)\right) C^*(b_1,\hdots,b_d)\frac{\log(N)^d}{N} + N2^{-M}
\end{split}
\end{equation}
as well as 
\begin{equation}\label{eq:righthand2}
\begin{split}
&\| (\widetilde \varphi_0 \gamma_R)_{Q}-(\widetilde \varphi_0 \gamma_R)^M_{Q,N}  \|_{L^2(\mathbb{R}^d)} \\
& \quad \leq \left(\sum_{x_j \in \Omega} \left(\mathrm{TV}\left(\widetilde \varphi_0 \gamma_R\circ \rho_{x_j} \right) \right)^2\right)^{\frac{1}{2}} C^*(b_1,\hdots,b_d)\frac{\log(N)^d}{N} + N2^{-M}|\{x_j \, \vert \, x_j \in \Omega\}|,
\end{split}
\end{equation}
where we recall $\rho_{x_j}$ and the sequence $\{x_j\}$ from Definition \ref{def:numerical_disc}.
Note that the bounds of the total variation of $\widetilde  V$ and $\widetilde  V_{\operatorname{con}}$, 
needed to bound the right hand side of \eqref{eq:righthand1}, follow directly from \eqref{eq:V_bound}. 
The bound of the total variation of $\widetilde \varphi_0 \gamma_R$, needed to bound the right hand side of  \eqref{eq:righthand2} needs a little more care. Note that  \eqref{eq:bounded_var_bound} immediately implies a bound on the total variation of $\widetilde \varphi_0$. Also, \eqref{eq:phi_approximation} implies a bound on $\|\widetilde \varphi_0\|_{L^{\infty}}$ in terms of $\|\varphi_0\|_{L^2}$ and $K$ in \eqref{eq:phi_approximation}. Thus, by using \eqref{eq:TV_bounds} we get a bound on the total variation of $\widetilde \varphi_0 \gamma_R$. 
Hence, to finish the proof it suffices to show that we can compute point samples of $\widetilde \varphi_0 \gamma_R$, $\widetilde V$ and $\widetilde V_{\operatorname{con}}$ to arbitrary precision. To show this we start with $\widetilde \varphi_0 \gamma_R$. Note that by the choice of $\gamma_R$ it suffices to show that we can compute recursively point samples of $\widetilde \varphi_0$ to arbitrary precision from point samples of $\varphi_0$. Recall from \eqref{eq:phi_approximation} that 
$
\widetilde{\varphi_0}:=\sum_{n=0}^{K-1} \langle \psi_n, \varphi_0 \rangle_{L^2} \psi_n
$
where the $\psi_n$s are eigenfunctions to the quantum harmonic oscillator based on exponential functions. Thus, to compute point samples of $\widetilde{\varphi_0}$ to arbitrary precision one needs to compute the inner products $\langle \psi_n, \varphi_0 \rangle_{L^2}$ for $n \leq K-1$ to arbitrary precision. Indeed, by using Proposition \ref{prop:bound_disc_int} this can be done as the total variation bound on $\psi_n\varphi_0$ follows from Assumption \ref{ass:2} that $\varphi_0$ has controlled local bounded variation by $\omega: \mathbb{R}_+ \rightarrow \mathbb{N}$ and the uniform bounds on the total variation of $\psi_n$ when $n \leq K-1$. Note also that it is immediate that this can be done with a fixed number of arithmetic operations and comparisons as a function of the precision needed. 

Note that, by Assumption \ref{ass:2}, $V_{\operatorname{con}},V$ have controlled local bounded variation by $\omega: \mathbb{R}_+ \rightarrow \mathbb{N}$. Thus, since we have that for $x \in \RR$ and $\xi>0$
\begin{equation}\label{eq:tale}
\begin{split}
\left\vert V_{\sigma}(x) - \int_{\vert y \vert  \le \xi} V(x-y) \chi_{\sigma}(y) \ dy \right\rvert 
&\le \Vert V \Vert_{L^{\infty}} \int_{\vert y \vert  > \xi} \chi_{\sigma}(y) \ dy \\
&= \Vert V \Vert_{L^{\infty}} \operatorname{erfc}\left(\frac{\xi}{\sqrt{2}\sigma} \right).
\end{split}
\end{equation}
Note that the complementary error function $\operatorname{erfc}$ is entire and strictly decreasing for positive inputs. Hence, it follows that, given any rational $\sigma > 0$ and $L \in \mathbb{N}$, one can recursively determine $\xi \in \mathbb{N}$ such that the right hand side of \eqref{eq:tale} is bounded by $2^{-L}$. Hence, we are left with the problem of computing $\int_{\vert y \vert  \le \xi} V(x-y) \chi_{\sigma}(y)$ to arbitrary precision for rational $x$ and given $\sigma$ and $\xi$. However, this can be done by again using Proposition \ref{prop:bound_disc_int} as long as one can bound the total variation of $V(x-\cdot) \chi_{\sigma}(\cdot)$. However, this can be done by using \eqref{eq:TV_bounds} and the local bounds on the total variation of $\chi_{\sigma}(\cdot)$ that follows immediately from bounds on the derivatives of $\chi_{\sigma}(\cdot)$.  
\end{proof}

\section{Proof of Theorem \ref{BigTh:main_thrm2_NLS}}
\subsection{Proof of Theorem \ref{BigTh:main_thrm2_NLS} with Assumption \ref{ass:1}}

\begin{proof}[Proof of Theorem \ref{BigTh:main_thrm2_NLS} with Assumption \ref{ass:1}]
We will construct a sequence of algorithms $\{\Gamma_n\}$ such that 
\begin{equation}
\left \| \Gamma_n(\varphi_0,V) - \psi(\bullet,T) \right\|_{L^2} \leq 2^{-n+1},
\end{equation}
where $\psi$ is the solution to the NLS equation
\begin{equation*}
\begin{split}
i \partial_t \psi(x,t) &= H_0 \psi(x,t) + V_{\operatorname{TD}}(t)\psi(x,t)+  \nu F_{\sigma}(\psi(x,t)), \quad (x,t) \in \mathbb R \times (0,T) \\
\psi(\bullet,0)&=\varphi_0
\end{split}
\end{equation*} 

{\bf Step I:} Almost Identical to Step I of the proof of Theorem \ref{BigTh:main_thrm2} with Assumption \ref{ass:1}. We will use notation from these steps below, and thus the reader is encouraged to read these before continuing with the other steps. The only difference is that we replace \eqref{eq:bound_phi} by 
\begin{equation}\label{eq:bound_phi2}
\Vert \varphi_0\gamma_R \Vert_{H^{3+\varepsilon}_2} \leq f_1(M, R), 
\end{equation}
where $f_1$ is recursive and $\Vert \varphi_0 \Vert_{H^{3+\epsilon}_2} \leq M$.

{\bf Step II:} (Removing singularities).
This step is almost identical to Step II in the proof of Theorem \ref{BigTh:main_thrm2}. The only difference is that \eqref{eq:bound_C_of_n} is replaced by 
\begin{equation}\label{eq:bound_C_of_n_2}
\Vert u \Vert_{W^{1,1}_{\operatorname{pcw}}}, \Vert V\Vert_{W^{3,\infty}},\Vert V_{\operatorname{con}}\big\vert_{\Omega} \Vert_{W^{3,\infty}} \leq \tilde C,
\end{equation}   
where the bound $\tilde C$ can be constructed recursively from the integer $n$ determining the accuracy, and where, with slight abuse of notation, $V = W_{\operatorname{sing}}\zeta + W_{\operatorname{reg}}\big\vert_{\Omega}$. To simplify the notation below we will omit the restrictions.

{\bf Step III:} (Choosing the gridsize $h$). We stay close to Step III in the proof of Theorem \ref{BigTh:main_thrm2}. We recall Definition \ref{cubdisc} in order to discretise the initial state $\varphi_0\gamma_R$ via $(\varphi_0\gamma_R)_Q$ that is a sum of characteristic functions according to the lattice depending on a step size $h$ in Definition \ref{cubdisc}. We will simply use the $(\varphi_0\gamma_R)_Q$ keeping the dependence of $h$ in mind. The size of this $h$ will be chosen at the very end. The difference from the linear case is that our approximation scheme will have to be slightly altered to the following:  
\begin{equation}
\begin{split}
\label{eq:StrangNLS2_alt}
&(1) \quad \varphi^-_{k+\tfrac{1}{2}} := \mathcal C^{\tau/2}_h \varphi^{\operatorname{NLS}}_{\operatorname{S}}(t_k), \\
&(2) \quad \varphi^+_{k+\tfrac{1}{2}}:=\exp_K\left(-i \tau X_{t_k}^{\operatorname{NLS}}(\psi_{k+\tfrac{1}{2}}^-)_Q \right) \varphi_{k+\tfrac{1}{2}}^-, \operatorname{ and }\\
&(3) \quad \varphi^{\operatorname{NLS}}_{\operatorname{S}}(t_{k+1}) :=\mathcal C^{\tau/2}_h \varphi^+_{k+\tfrac{1}{2}},
\end{split}
\end{equation}
where $\exp_K$ is defined in \eqref{eq:expK} and 
\begin{equation}\label{eq:initial2}
\varphi^{\operatorname{NLS}}_{\operatorname{S}}(0):=(\varphi_0 \gamma_R)_Q, \quad X_{t_k}^{\operatorname{NLS}}(\psi)_Q:=V_{\operatorname{TD}}(t_k)_Q + \vert  \psi \vert^{\sigma-1},
\end{equation} 
where we recall the definition of $\mathcal C^{\tau/2}_h$ from the Crank-Nicholson method \eqref{eq:discrule}. 

By Proposition \ref{prop:nontime} there are recursive maps 
\begin{equation}\label{eq:rec3}
\begin{split}
&T, \Vert \varphi_0 \gamma_R\Vert_{H^{3+\varepsilon}},\Vert V \Vert_{W^{3,\infty}},\Vert V_{\operatorname{con}} \Vert_{W^{3,\infty}}, h,\Vert u \Vert_{W^{1,1}_{\operatorname{pcw}}}\\
& \qquad \mapsto \tau(T,\Vert \varphi_0 \gamma_R\Vert_{H^{3+\varepsilon}},\Vert V \Vert_{W^{3,\infty}},\Vert V_{\operatorname{con}} \Vert_{W^{3,\infty}}, h,\Vert u \Vert_{W^{1,1}_{\operatorname{pcw}}})
\end{split}
\end{equation} 
determining the stepsize $\tau$ in \eqref{eq:StrangNLS2_alt}, and 
\begin{equation}\label{eq:rec4}
\begin{split}
&T,\tau, \Vert \varphi_0\gamma_R \Vert_{H^{3+\varepsilon}},\Vert V \Vert_{W^{3,\infty}},\Vert V_{\operatorname{con}} \Vert_{W^{3,\infty}},  h,\Vert u \Vert_{W^{1,1}_{\operatorname{pcw}}}  \\
& \qquad \mapsto K(T,\tau, \Vert \varphi_0 \gamma_R\Vert_{H^{3+\varepsilon}},\Vert V \Vert_{W^{3,\infty}},\Vert V_{\operatorname{con}} \Vert_{W^{3,\infty}}, h,\Vert u \Vert_{W^{1,1}_{\operatorname{pcw}}})
\end{split}
\end{equation}
determining $K$ in \eqref{eq:StrangNLS2_alt} such that 
\begin{equation}\label{eq:H1_h_bound}
 \max_{k} \Vert \psi^{D}_R(\tau k) - \varphi^{\operatorname{NLS}}_{S}(\tau k) \Vert_{H^1_h} \le C_{T} \Vert \varphi_0 -(\varphi_0)_Q \Vert_{H^1_h}.
 \end{equation}
 To translate this into an $L^2$ bound we use Lemma \ref{convrate} and \eqref{eq:H1_h_bound} and get 
\begin{equation}\label{eq:np}
\begin{split}
\max_k  \left\Vert \psi^{D}_R(\tau k) - \varphi^{\operatorname{NLS}}_{S}(\tau k)\right\Vert_{L^2} &\le 
\max_{k} \Vert \psi^{D}_R(\tau k) - \varphi^{\operatorname{NLS}}_{S}(\tau k) \Vert_{H^1_h}\\
& \le C   \left\Vert \varphi_0 \gamma_R -(\varphi_0 \gamma_R)_Q \right\Vert_{H^1_h} \le C  h \Vert \varphi_0 \gamma_R\Vert_{H^2}. 
\end{split}
\end{equation}
Hence, using \eqref{eq:np} and \eqref{eq:bound_phi2} in Step I
we can deduce that we can recursively compute
\begin{equation*}
h \text{ (from $n$) such that } \max_k  \left\Vert \psi^{D}_R(\tau k) - \varphi^{\operatorname{NLS}}_{S}(\tau k)\right\Vert_{L^2} \le 2^{-(n+1)}.
\end{equation*}

{\bf Step IV:} (Choosing $N$ in the integration). To finalise the proof we need to approximate the initial state $(\varphi_0 \gamma_R)_Q$ and the potential 
$V_{\operatorname{TD}}(t_k)_Q$ in \eqref{eq:StrangNLS2_alt} with the numerical integration from Definition \ref{cubdisc}. In particular, we replace (2) in \eqref{eq:StrangNLS2_alt} by 
\[
(2) \quad \varphi^+_{k+\tfrac{1}{2}}:=\exp\left(-i \tau X_{t_k}^{\operatorname{NLS}}(\psi_{k+\tfrac{1}{2}}^-)_{Q,N} \right) \varphi_{k+\tfrac{1}{2}}^-
\]
and \eqref{eq:initial2} gets replaced by 
\[
\varphi^{\operatorname{NLS}}_{\operatorname{S}}(0):=(\varphi_0)_{Q,N}, \quad X_{t_k}^{\operatorname{NLS}}(\psi)_{Q,N}:=V_{\operatorname{TD}}(t_k)_{Q,N} + \vert  \psi \vert^{\sigma-1},
\]
where we recall that for $\phi \in L^1_{\operatorname{loc}}(\mathbb R^d,\mathbb C)$, $\phi_{Q,N}$ has been introduced in Definition \ref{cubdisc}.  Let $ \tilde  \varphi_{S}(\tau k) $ denote the outputs of this modified scheme. 
Then, by Proposition \ref{prop:approx_int}, it follows that for $k=\lceil T/\tau \rceil$, $R$ and $h$
\begin{equation}
\begin{split}
 \left\Vert  \varphi_{S}(\tau k) - \tilde  \varphi_{S}(\tau k) \right\Vert_{L^2} &\le C(n, R, h) \big( \left\Vert V_Q- V_{Q,N} \right\rVert_{L^{\infty}} \\
& \, \vee \left\Vert V_{\operatorname{con},Q}- V_{\operatorname{con},Q,N} \right\rVert_{L^{\infty}} \vee \left\lVert (\varphi_0)_{Q}-(\varphi_0)_{Q,N} \right\Vert_{L^2} \big),
 \end{split}
 \end{equation}
where the mapping $n, R, h \mapsto C(n, R, h)$ is recursive. 

{\bf Step V-VI: } The rest of the proof is identical to the proof of is identical to the proof of Theorem \ref{BigTh:main_thrm2}. 
\end{proof}

\subsection{Proof of Theorem \ref{BigTh:main_thrm2_NLS} with Assumption \ref{ass:2}}

\begin{proof}[Proof of Theorem \ref{BigTh:main_thrm2_NLS} with Assumption \ref{ass:2}]
We will follow the proof of Theorem \ref{BigTh:main_thrm2} with Assumption \ref{ass:2} and the proof of Theorem \ref{BigTh:main_thrm2_NLS} with Assumption \ref{ass:1} closely. In certain cases the passages follow the above mentioned proofs verbatim. 

{\bf Step I:} (Perturbation theory).
This step varies slightly from Step I in the proof of Theorem \ref{BigTh:main_thrm2} with Assumption \ref{ass:2}.
Let $\varphi(t;\varphi_0)$ denote the solution to \eqref{eq:schr} with initial state $\varphi_0$. It follows straight from the variation of constant formula
\[ \varphi(t;\varphi_0) = e^{i \Delta t} \varphi_0 -i \int_0^t  e^{i \Delta (t-s)}\vert \varphi(s;\varphi_0) \vert^{\sigma-1} \varphi(s;\varphi_0) \ ds \]
that there exists a recursively defined function $K(C;T)$ such that 
\begin{equation}
\begin{split}
&\sup_{t \in (0,T)} \Vert \varphi(t;\varphi_0)  -  \varphi(t;\widetilde \varphi_0) \Vert_{L^2} \le K(C,T) \Vert \varphi_0 - \widetilde \varphi_0\Vert_{L^2},
\end{split}
\end{equation}
where $C$ is the constant from Assumption \ref{ass:2} and $T$ is the final time. 
We therefore start by smoothing out the initial state in order to use the techniques in the proof of Theorem \ref{BigTh:main_thrm2} with Assumption \ref{ass:1} that needs smoothness. 
 
{\bf Step II:} (Smoothing out initial states). This is exactly as in Step II of the proof of Theorem \ref{BigTh:main_thrm2} with Assumption \ref{ass:2}.

{\bf Step III:} (Choosing $R$ ). This is exactly as in Step III of the proof of Theorem \ref{BigTh:main_thrm2_NLS} with Assumption \ref{ass:1}.

{\bf Step IV:} (Smoothing out potentials). This step is almost identical to Step IV in the proof of Theorem \ref{BigTh:main_thrm2} with Assumption \ref{ass:2}, except that \eqref{eq:bound_C_of_n5} is replaced by  
 \begin{equation*}
\Vert u \Vert_{W^{1,1}_{\operatorname{pcw}}}, \Vert \widetilde V\Vert_{W^{3,\infty}},\Vert \widetilde V_{\operatorname{con}}\big\vert_{\Omega} \Vert_{W^{3,\infty}(\Omega)}, \Vert \widetilde  \varphi_0 \Vert_{H^3_2} \leq \tilde C.
\end{equation*}   

{\bf Step V:} (Choosing the gridsize $h$ in the discretisation). This is done exactly as in Step III in the proof of Theorem \ref{BigTh:main_thrm2_NLS} with Assumption \ref{ass:1}.

{\bf Step VI:} (Choosing $N$ and $M$ in the integration). 
This is exactly as in Step VI of the proof of Theorem \ref{BigTh:main_thrm2} with Assumption \ref{ass:2}.
\end{proof}

\section{Proof of Theorem \ref{Th:main_thrm_blow}}

 \subsection{Determining if the initial state yields blow up of the NLS}
 \label{subsubsec:blowup}
Let us consider norms $\vert f \vert_{L^2}:= \Vert f \Vert_{L^2},$ $\vert f \vert_{H^1}:= \Vert f \Vert_{L^2}^{\alpha}\Vert f \Vert_{\dot{H^1}}^{1-\alpha}$ for some fixed $\alpha \in (0,1)$, and $\vert f \vert_{\dot{H^1}}= \Vert f \Vert_{\dot{H^1}}$ a non-trivial function $f$. We then let $X \in \left\{L^2(\Omega), H^1(\Omega), \dot{H}^1(\Omega)\right\}$ and $\Omega \subset \mathbb R^d$ a domain. 
Let $C>0$ be given and let $\Omega_{\operatorname{BU}(X)}$ be the set of functions $v \in X \cap C(\Omega)$ with $\vert v \vert_X \le C$ and $v$ has controlled local bounded variation by $h.$
We then consider the condition
\begin{equation}
\label{eq:le2}
\vert u_0 \vert_X \le  \vert f \vert_X.
\end{equation}
 To define the computational problem we define for $f \in \Omega_{\operatorname{BU}(X)}$ the set 
\begin{equation}
\begin{split}
\Omega_{\mathrm{BU(X,f)}} &= \left\{u_0 \in \Omega_{\operatorname{BU}(X)}; \vert u_0 \vert_X \neq \vert f \vert_X \right\}, \mathcal{M} = \{\mathrm{No}, \mathrm{Yes}\} = \{0,1\}\text{ and }\\
 \Xi_{\mathrm{BU(X,f)}}(\varphi_0) &= \text{Does \eqref{eq:le2} hold?}.
 \end{split}
 \end{equation}

\subsubsection{Impossibility of blow-up analysis:}
\label{sec:blow-up}

As we saw in the introduction in Section \ref{subsec:INLS}, the blow-up analysis for focussing NLS can in many cases be reduced to the decision problem stated in Section \ref{subsubsec:blowup}.
 \begin{prop}
 \label{prop:Sigma1G}
Given the setup as in Section \ref{subsubsec:blowup}, we have that 
\[\{\Xi_{\mathrm{BU}(f,X)},\Omega_{\mathrm{BU}(f,X)},\mathcal{M},\Lambda\} \notin \Sigma_1^G.\]
 \end{prop}
 \begin{proof}
To show that $\{\Xi_{\mathrm{BU}(f,X)},\Omega_{\mathrm{BU}(f,X)},\mathcal{M},\Lambda\} \notin \Sigma_1^G$ we argue by contradiction and assume the contrary. Let therefore $\{\Gamma_n\}$ be a sequence of general algorithms such that $\Gamma_n(\varphi_0) \rightarrow \Xi_{\mathrm{BU}(f,X)}(\varphi_0)$ as $n \rightarrow \infty$, and with $\Gamma_n(\varphi_0) = 1 \Rightarrow \Xi_{\mathrm{BU}(f,X)}(\varphi_0) = 1$.
 Let $\varphi_0 \in \Omega_{\mathrm{BU}(f,X)}$ denote a function satisfying \eqref{eq:le2} and note that, by the reasoning above $\Xi_{\mathrm{BU}(f,X)}(\varphi_0) = 1$. Thus, there is an $N \in \mathbb{N}$ such that $\Gamma_n(\varphi_0) = 1$ for all $n \geq N$. Choose any such $n \geq N$ and let $\mathcal{B} \subset \Omega$ be an open ball such that for all $f_j \in \Lambda_{\Gamma_n}(\varphi_0)$ we have $\omega_j \notin \mathcal{B}$.  Choose a $\tilde \varphi_0 \in \Omega$ such that $\mathrm{supp}(\tilde \varphi_0) \subset \mathcal{B}$ and 
 \begin{equation}\label{eq:cond}
\left\lvert \tilde \varphi_0 \right\rvert_{X} > \left\lvert f \right\rvert_{X}.
\end{equation}  
Note that such a choice is easy to justify by using bump functions. Note that, by the choice of $\tilde \varphi_0$ we have that 
$
f_j(\tilde \varphi_0) = f_j(\varphi_0) \quad \forall f_j \in \Lambda_{\Gamma_n}(\varphi_0).
$    
Hence, by assumption (iii) in (ii) in Definition \ref{alg} it follows that $1 = \Gamma_n(\varphi_0) = \Gamma_n(\tilde \varphi_0)$. However, by \eqref{eq:cond}, it follows that $\Xi_{\mathrm{BU}(f,X)}(\tilde \varphi_0) = 0$, which contradicts that $\Gamma_n(\tilde \varphi_0) = 1 \Rightarrow \Xi_{\mathrm{BU}(f,X)}(\tilde \varphi_0) = 1$, and we have reached the desired contradiction. 
\end{proof}

\begin{prop}[Mass critical NLS]
Given the setup as in \eqref{eq:focusing2}, we have that 
\[
\{\Xi_{\mathrm{BU}(2)},\Omega_{\mathrm{BU}(2)},\mathcal{M},\Lambda\} \notin \Sigma_1^G.
\]
\end{prop}
\begin{proof}
The ground state soliton $Q$ satisfying 
\[ -\Delta Q-Q \vert Q \vert^4+Q=0 \]
for the $1d$-quintic NLS is known explicitly $Q(x) = \left( \frac{3}{\cosh^2(2x)} \right)^{1/4}$ and exists for all $d \ge 1.$
For $d\ge 1$ and $\sigma=1+4/d$, it is known \cite{D15} that if $\Vert \varphi_0 \Vert_{L^2} < \Vert Q \Vert_{L^2}$ then the solution to \eqref{eq:focusing2} exists globally and scatters whereas for $\Vert \varphi_0 \Vert_{L^2}> \Vert Q \Vert_{L^2}$ there exist solutions that exist only for finite time. The statement then follows from Proposition \ref{prop:Sigma1G}.
\end{proof}

Showing that $\{\Xi_{\mathrm{BU}(2)},\Omega_{\mathrm{BU}(2)},\mathcal{M},\Lambda\} \notin \Pi_1^G$ is in general more subtle. To see this, observe that by Sobolev's embedding in dimension one, we have $\Vert \varphi_0 \Vert_{L^{\infty}} \le \Vert \varphi_0 \Vert_{H^1}.$
This implies that if an algorithm samples a sufficiently large value of $\varphi_0$ it follows that $\Vert \varphi_0 \Vert_{H^1}$ is large as well. 

For our next proposition we consider a bump function 
\[\chi_{\varepsilon,x_0}(x):=e^{1+\frac{\varepsilon^2}{\Vert x-x_0\Vert^2-\varepsilon^2} }\indic_{B(x_0, \varepsilon)}(x).\]
We then have that 
\begin{equation}
\label{eq:chi}
\Vert \chi_{\varepsilon,x_0} \Vert_{L^2} = \mathcal O(\varepsilon^d)\text{ and }\Vert \chi_{\varepsilon,x_0} \Vert_{\dot{H}^1} = \mathcal O(\varepsilon^{d-2}).
\end{equation}

If we impose stronger conditions on $X$ and the dimension, we obtain the following result:
 \begin{prop}
 \label{prop:Pi1G}
For the setup as in Section \ref{subsubsec:blowup},  it follows that $\{\Xi_{\mathrm{BU}(f,X)},\Omega_{\mathrm{BU}(f,X)},\mathcal{M},\Lambda\} \notin \Pi_1^G$ under the following conditions on the space $X$ and the dimension $d$ with open domain $\Omega \subset  \RR^d$
\begin{itemize}
\item If $d=1$ and $X \in \left\{L^2(\Omega)\cap C(\Omega), H^1(\Omega)\cap C(\Omega) \right\} $ with $\alpha< 1/2.$
\item If $d=2$ and $X \in  \left\{L^2(\Omega)\cap C(\Omega), H^1(\Omega)\cap C(\Omega)\right\}.$
\item $d \ge 3.$
\end{itemize}
 \end{prop}
 \begin{proof}
We argue again by contradiction. Assuming the contrary, let $\{\Gamma_n\}$ be a sequence of general algorithms such that $\Gamma_n(\varphi_0) \rightarrow \Xi_{\mathrm{BU}(f,X)}(\varphi_0)$ as $n \rightarrow \infty$, and with $\Gamma_n(\varphi_0) = 0 \Rightarrow \Xi_{\mathrm{BU}(f,X)}(\varphi_0) = 0$.       
Let $ \varphi_0 \in \Omega_{\mathrm{BU}(f,X)}$ be a function that does not satisfy \eqref{eq:le2}.  In this case $\Xi_{\mathrm{BU}(f,X)}( \varphi_0) = 0$ and hence there is an $N \in \mathbb{N}$ such that $\Gamma_n( \varphi_0) = 0$ for all $n \geq N$. Let $\varepsilon$ be small enough such that $B(\tilde \omega_j,\varepsilon)$ are disjoint.

Choose any such $n$ and choose $\tilde \varphi_0 :=\sum  \varphi_0(\omega_j)  \chi_{\varepsilon,\omega_j}$ such that $\tilde \varphi_0$ interpolates $ \varphi_0$ at the points $\tilde \omega_j$, where $f_j( \varphi_0) =  \varphi_0(\tilde \omega_j)$ and $f_j \in \Lambda_{\Gamma_n}( \varphi_0)$. 
Let $\varepsilon$ be sufficiently small, then by \eqref{eq:chi} it follows that
 $\left\lvert \tilde \varphi_0  \right\rvert_{X} < \left\lvert f  \right\rvert_{X}$. Then, as argued as above, we have 
$
f_j(\tilde \varphi_0) = f_j( \varphi_0) \quad \forall f_j \in \Lambda_{\Gamma_n}(\tilde \varphi_0),
$    
and hence by by assumption (iii) in (ii) in Definition \ref{alg} it follows that $0 = \Gamma_n( \varphi_0) = \Gamma_n(\tilde \varphi_0)$. However, since $\left\lvert \tilde \varphi_0  \right\rvert_{X} < \left\lvert f  \right\rvert_{X}$ we have that $\Xi_{\mathrm{BU}(f,X)}(\tilde \varphi_0) = 1$, which contradicts that $\Gamma_n(\tilde \varphi_0) = 0 \Rightarrow \Xi_{\mathrm{BU}(f,X)}(\tilde \varphi_0) = 0$.
\end{proof}

We continue with our result on the cubic NLS:

\begin{prop}
Given the setup in \eqref{eq:focusing} we have that 
\[
\{\Xi_{\mathrm{BU}(1)},\Omega_{\mathrm{BU}(1)},\mathcal{M},\Lambda\} \notin \Pi_1^G.
\]
\end{prop}
\begin{proof}
For \eqref{eq:focusing} one has the following blow up dichotomy \cite{HR,HPR}:
Let $\varphi_0 \in H^1_1(\mathbb R^3)$ be an initial state to the focusing NLS \eqref{eq:focusing} with ground state soliton $Q$ satisfying 
\[ -\Delta Q-Q \vert Q \vert^2+Q=0. \]
\begin{itemize}
\item If $\left\lVert \varphi_0 \right\rVert_{L^2}\left\lVert \nabla \varphi_0 \right\rVert_{L^2} < \left\lVert Q \right\rVert_{L^2}\left\lVert \nabla Q \right\rVert_{L^2},$ then the solution to \eqref{eq:focusing} exists globally in time in the space $H^1(\mathbb R^3).$
\item If $\left\lVert \varphi_0 \right\rVert_{L^2}\left\lVert \nabla \varphi_0 \right\rVert_{L^2} > \left\lVert Q \right\rVert_{L^2}\left\lVert \nabla Q \right\rVert_{L^2},$ then the solution to \eqref{eq:focusing} blows up in finite time, i.e. the solution to \eqref{eq:focusing} exists only in a maximum time interval $[0,T_{\operatorname{max}})$ in $H^1(\mathbb R^3).$ The result then follows from Proposition \ref{prop:Pi1G}.
\end{itemize}
\end{proof}

\begin{proof}[Proof of Theorem \ref{Th:main_thrm_blow}]
Theorem \ref{Th:main_thrm_blow} follows immediately from the analysis above. 
\end{proof}

The phenomenon of undecidability is, for the blow-up dichotomy, not due to the unboundedness of the domain as the following example shows:

\begin{ex}[Cubic NLS on bounded domain] Let $\Omega \subset \mathbb R^2$ be a bounded and smooth domain:  
Consider the cubic NLS with Dirichlet data $\varphi_0 \in H^2(\Omega) \cap H_0^1(\Omega)$ 
\begin{equation}
\begin{split}
i \partial_t\psi(x,t) + \Delta \psi(x,t)+ \vert \psi(x,t) \vert^2 \psi(x,t) &= 0, \ (x,t) \in \Omega \times (0,T), \\
 					\psi(x,t) &= 0, \  (x,t) \in \partial \Omega \times (0,T), \\
					\psi(x,0) &= \varphi_0(x), \  x \in \Omega. 
\end{split}
\end{equation}
This equation has a unique positive ground state to the equation
\[-Q(x)+ \Delta Q(x) + \left\lvert Q(x) \right\rvert^2 Q(x) = 0, \quad x \in \mathbb R^2. \]
Then, there exists a solution with the same $L^2$ norm as $Q$ that blows up in finite time \cite[Theorem $1$]{BGT03}, whereas \cite[Lemma $2.3$]{BGT03} shows that for Dirichlet initial data $\varphi_0 \in H^2(\Omega) \cap H_0^1(\Omega)$ with $\left\lVert \varphi_0  \right\rVert_{L^2(\Omega)} < \left\lVert Q  \right\rVert_{L^2(\mathbb R^2)}$ the solution exists globally in time, see also \cite{W82}.
\end{ex}

\section{Proof of Theorem \ref{Th:main_thrm_discNLS}}
\label{sec:disc}
In this section we discuss the computability of discrete nonlinear Schr\"odinger equations \eqref{eq:fullNLS} using the Strang splitting problem \eqref{eq:StrangNLS} and prove Theorem \ref{Th:main_thrm_discNLS}. But unlike in the continuous case, we allow for either sign in front of the nonlinearity.

Let the one-dimensional discrete Laplacian with Dirichlet or Neumann boundary conditions on $I_n^1$ be denoted by $\Delta_{\operatorname{BVP}}$, then the multi-dimensional discrete Laplacian on $I_n^d$ is defined by 
\[\Delta_{\operatorname{BVP}}^{d}:= \sum_{i=1}^d \operatorname{id}^{i-1} \otimes \Delta_{\operatorname{BVP}} \otimes \operatorname{id}^{d-i}. \]
Let $I_n^d:=[-n,n]^d \cap \mathbb Z^d,$ we consider now a discrete NLS with $\nu \in \{\pm 1\}$ on the entire space $\ell^2(\mathbb Z^d)$
\begin{equation}
\begin{split}
\label{eq:fullNLS2}
i  \partial_t v(k,t) &=- \Delta^d v(k,t) +\nu F_{\sigma}(v(k,t)) ,\ k \in \mathbb Z^d \\
v(0)&=v_0 \in \ell^2(\mathbb{Z}^d)
\end{split}
\end{equation} 
and associate to it a boundary value problem on the hypercube $I_n^d$
\begin{equation}
\begin{split}
\label{eq:BVP}
i\partial_t v_{\operatorname{BVP}}(k,t) &= -\Delta_{\operatorname{BVP}}^{d} v_{\operatorname{BVP}}(k,t) + F_{\sigma}(v_{\operatorname{BVP}}(k,t) ), \ k \in I_n^d \\
v(0)&=v_0\vert_{I_n} \in \ell^2(I_n).
\end{split}
\end{equation} 
We then have the following discrete analogue of Theorem \ref{theo2} which allows us to estimate the difference between \eqref{eq:fullNLS2} and \eqref{eq:BVP}:
\begin{prop}
\label{prop:disc}
Consider the difference of solutions $\xi=v-v_{\operatorname{BVP}}$ to \eqref{eq:fullNLS2} and \eqref{eq:BVP}, respectively. If the initial datum satisfies a bound $\Vert v_0 \Vert_{\ell^2_s}\le A$ for some $s>0$, then this implies 
\[\sum_{k \in I_{n-1}} \vert \xi(k,t) \vert^2 \lesssim \langle n \rangle^{-s/2}. \]
\end{prop}
\begin{proof}
If the initial datum satisfies a fixed decay bound $\Vert v_0 \Vert_{\ell^2_s}\le A$ for some $s>0$, then this implies  \cite[Lemma $2$]{KPS} that for $t \in (0,T)$ and both \eqref{eq:fullNLS2} and \eqref{eq:BVP} there exists $C_T>0$ such that
\[ \Vert v(t) \Vert_{\ell^2_s} \le C_T \Vert v(0) \Vert_{\ell^2_s}\]
and the same for $v_{\operatorname{BVP}}.$
Thus, this implies that, again for both $v$ and $ v_{\operatorname{BVP}}$, just denoted by $v$, that if
\begin{equation}
\label{eq:estm}
\Vert v(0)\indic_{\mathbb Z^d \backslash I_n^{d}} \Vert_{\ell^2} \le \frac{ A}{\langle n \rangle^{s/2}} \text{ then }\Vert v(t)\indic_{\mathbb Z^d \backslash I_n^{d}} \Vert_{\ell^2} \le \frac{C_T A}{\langle n \rangle^{s/2}}
\end{equation}
for all $t \in (0,T).$

The variation of constant formula immediately implies that the solution is given as 
\[ v(k,t)= e^{-i t\Delta^d}v_0 + \int_0^t e^{-i (t-s)\Delta^d}F_{\sigma}(v(k,s)) \ ds\]
which implies by Gronwall's inequality that the solution is Lipschitz continuous with respect to initial data with a recursively computable bound. Hence, it suffices by \eqref{eq:estm} to assume that the initial state to \eqref{eq:fullNLS2} has compact support in $I_n$ up to an error $\mathcal O(\langle n \rangle{-s/2}).$

We also have that for $\xi(k,t):=v(k,t)- v_{\operatorname{BVP}}(k,t)$ since in the discrete case $\vert v(k,t) \vert, \vert v_{\operatorname{BVP}}(k,t) \vert \lesssim \mathcal O(1)$ and $\Delta_d$ is a bounded operator,
\begin{equation}
\begin{split}
\partial_t  \sum_{k \in I_n^d} \frac{\vert \xi(k,t) \vert^2}{2}&=\Re \sum_{k \in I_n^d} \partial_t \xi(k,t) \ \overline{\xi(k,t)} \\
&\lesssim \left\lvert \Im \sum_{k \in I_n^d} (\Delta_d \xi)(k,t) \overline{\xi(k,t)}  \right\rvert +\sum_{k \in I_n^d}  \vert \xi(k,t) \vert^2\\
&\lesssim  \sum_{k \in I_{n+2}^d \backslash I_{n-2}^d} \left(\vert v(k,t)\vert^2+\vert v_{\operatorname{BVP}}(k,t)\vert^2\right) +\sum_{k \in I_n^d}  \vert \xi(k,t) \vert^2.
\end{split}
\end{equation}
By Gronwall's inequality, we then have that $\sum_{k \in I_n^d} \vert \xi(k,t) \vert^2 \lesssim  \sum_{k \in I_{n+2}^d \backslash I_{n-2}^d} \vert v(k,t)\vert^2.$
This implies by \eqref{eq:estm} that uniformly on bounded sets in time
$
\sum_{k \in I_n^d} \vert \xi(k,t) \vert^2 \lesssim \langle n \rangle^{-s/2}. 
$
\end{proof}

We can now give the proof to Theorem \ref{BigTh:main_thrm2_NLS} and show that the Strang splitting scheme \eqref{eq:StrangNLS} provides a convergent algorithm for the discrete NLS. Since many steps are similar and simpler in the discrete setting to the continuous setting, we only comment on the difference to the proof of Theorem \ref{BigTh:main_thrm2_NLS}

\begin{proof}[Proof of Theo. \ref{Th:main_thrm_discNLS}]
{\bf Step I:} (Choosing $n$). We can restrict our equation to the cube $I_n^d$ centred at zero with length $n$ by Prop. \ref{prop:disc}.

{\bf Step II-IV:} Due to the discreteness of the equation, there is no singular potential, no gridsize parameter $h$, and no need for numerical integration to obtain a cubic discretization.

{\bf Step V:} The recursiveness of the choice of $n$ follows from Prop.\ref{prop:disc} whereas the recursiveness of the Strang splitting and Crank-Nicholson method follows as in the proof of Theorem \ref{BigTh:main_thrm2_NLS}.
\end{proof}

\section{Numerical examples}

In this section, we aim to illustrate two phenomena.

\begin{enumerate}
\item The solution to a Schr\"odinger equation on an unbounded domain is well-approximated by a BVP on a sufficiently large bounded domain. 
\item Blow-up of solutions to NLS can- in general- numerically not be computed. 
\end{enumerate}
To address the first point, we compare the explicit solution to a linear Schr\"odinger equation \eqref{eq:F} to the solution of a numerically computed BVP with the same potential.  This is illustrated in Figure \ref{fig:numericsone} with details provided in Subsection \ref{subsec:1}.

To address the second point, we compare the explicit solution to a focussing NLS that blows up in finite time \eqref{eq:quintic} to the output of a standard finite difference scheme on a bounded domain for that equation which suggests a singularity formation but does not capture the exact point breaking time of the solution well.  This is illustrated in Figure \ref{fig:blowup} with details provided in Subsection \ref{subsec:2}.

\subsection{Linear Schr\"odinger equation}
\label{subsec:1}
We consider the linear Schr\"odinger equation 
\begin{equation}
\label{eq:F}
i \partial_t \psi(x,t) =-\psi_{xx}(x,t) - F(t)x\psi(x,t)
\end{equation}
with time-dependent electric potential. 

The solution to this equation is explicitly given by
\[\psi(x,t) = \frac{1}{\sqrt{A(t)/A_0}} \operatorname{exp} \left( -i \int_0^t p_c^2(\tau)\ d\tau \right) \  \operatorname{exp} \left( -i\frac{\mathcal B_0(x-x_c(t))^2}{2A(t)} + i p_c(t)x \right) \]
with average momentum $p_c(t) =  \int_0^t F(s) \ ds$, average position $x_c(t) = \int_0^t \int_0^s F(\tau) \ d\tau \ ds,$ and constants $A_0,\mathcal B_0$ such that $A(t) = A_0 - \mathcal B_0 t.$
\begin{figure}
  \centering
  \begin{subfigure}{0.45\textwidth}
 \includegraphics[width=7cm,height=5cm]{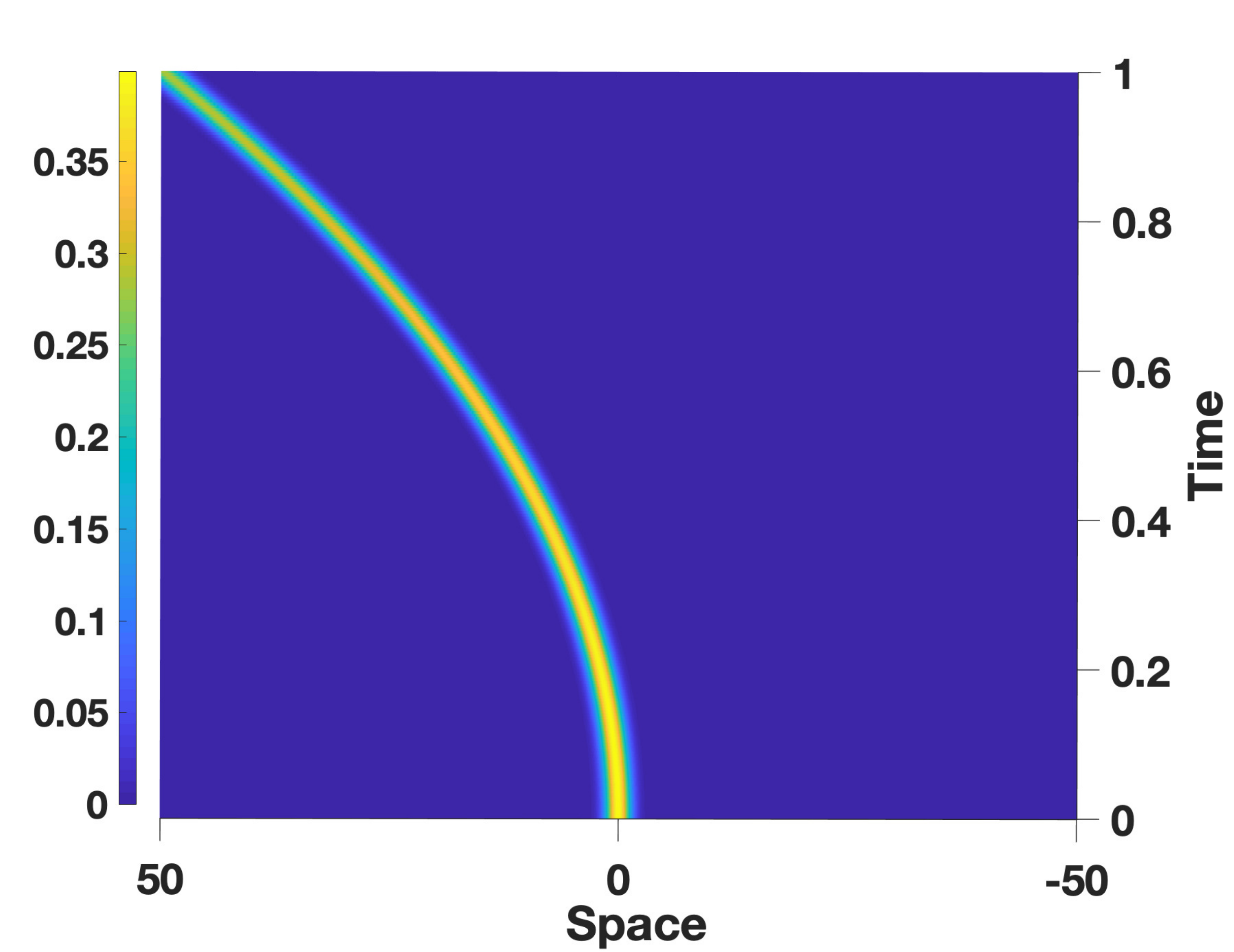}
    \caption{Numerical solution with $F(t)=50$ in \eqref{eq:F}.}
  \end{subfigure}
\qquad \
  \begin{subfigure}{0.45\textwidth}
    \includegraphics[width=7cm,height=5cm]{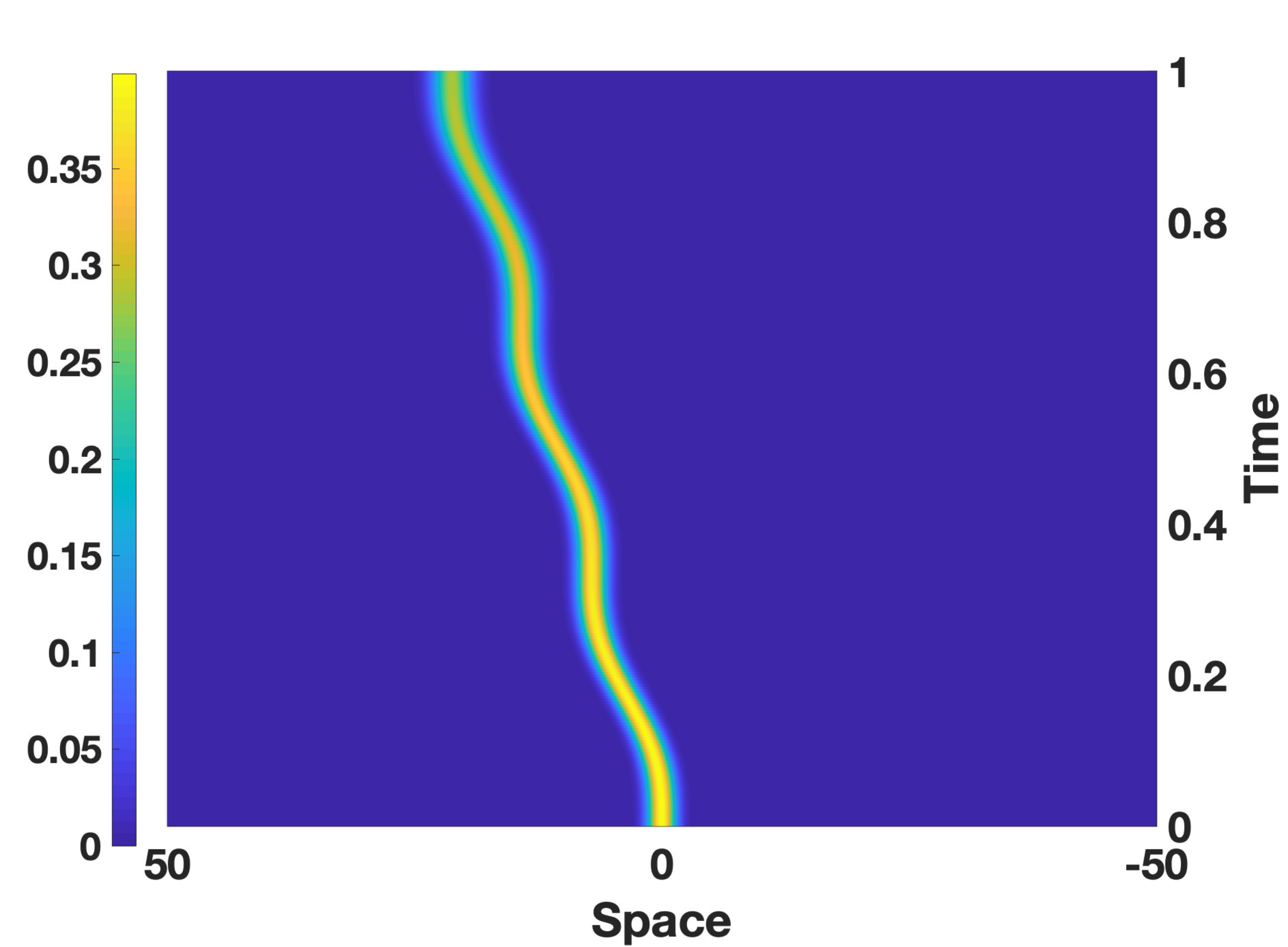}
    \caption{Numerical solution $F(t)=200\sin(6\pi t)$ in \eqref{eq:F}.}
  \end{subfigure}
  \caption{Numerical solution to the linear Schr\"odinger equation. The dynamics \eqref{eq:F} preserves spatial localization of the coherent state and can thus be studied on numerically on a bounded domain \eqref{eq:F}. The color highlights the density of the state.}
  \label{fig:numericsone}
\end{figure}
The density function is then 
\[ \left\vert \psi(x,t) \right\vert^2=\frac{e^{\Im(\mathcal B_0/A_0) \frac{\vert x-x_c(t) \vert^2 }{ \vert A(t)/A_0 \vert^2}}}{\vert A(t)/A_0 \vert}. \]
Thus, we see that although the state disperses over time, it remains exponentially localized. This property allows us to study the global solution in a small finite window, cf. Fig. \ref{fig:numericsone}. 
\subsection{Focusing NLS}
\label{subsec:2}
\begin{figure}
  \centering
  \begin{subfigure}{0.45\textwidth}
 \includegraphics[width=8cm]{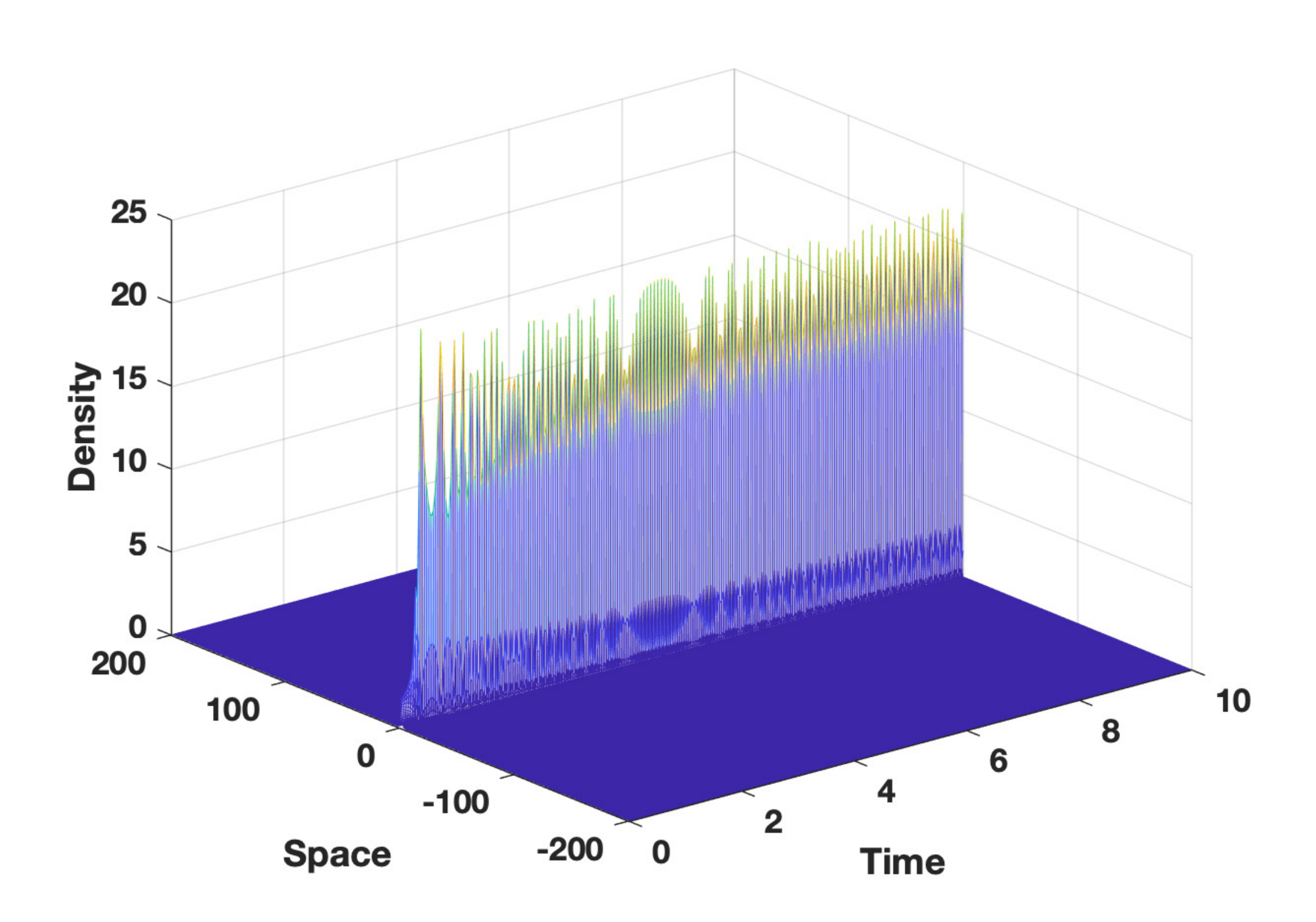}
    \caption{Numerical solution past blow-up time $T^*=10$.}
  \end{subfigure}
\qquad \
  \begin{subfigure}{0.45\textwidth}
    \includegraphics[width=8cm]{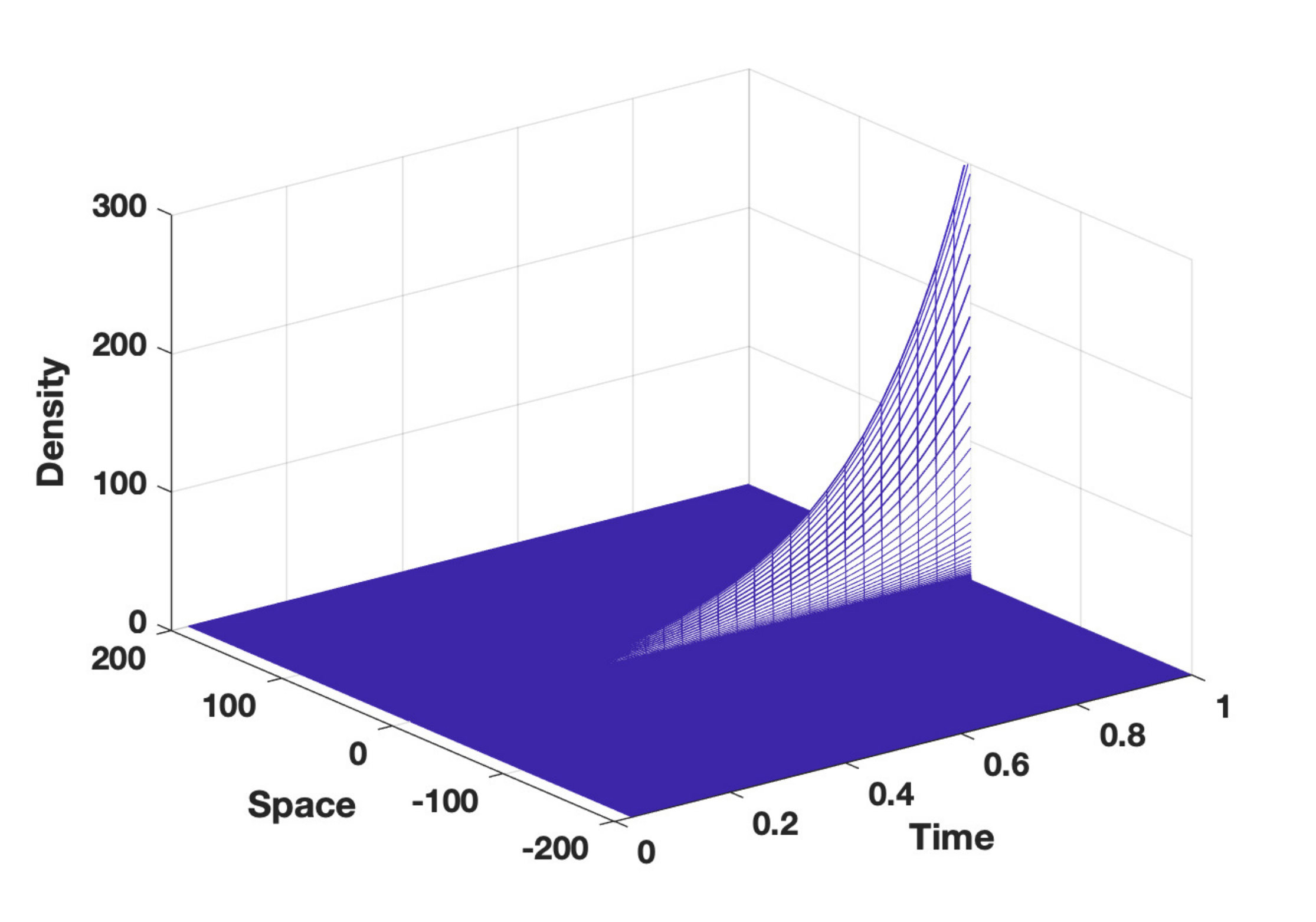}
    \caption{Exact solution on time interval $[0,1]$.}
  \end{subfigure}
  \caption{The numerical and exact solution to the focusing quintic NLS \eqref{eq:quintic} with initial state \eqref{eq:blow-upsol} and blow up time $T^*=10$.}
 \label{fig:blowup}
\end{figure}

Consider for illustrative purposes the quintic focusing 1d NLS 
\begin{equation}
\begin{split}
\label{eq:quintic}
i \partial_t \psi(x,t) &= - \psi_{xx}(x,t)- \left\lvert \psi \right\rvert^4 \psi (x,t), \quad (x,t) \in \mathbb R \in (0,T), \\
\psi(\bullet,0) &= \varphi_0 \in H^1(\mathbb R).
\end{split}
\end{equation}
Then, there exists a family \cite[(5)]{Pe} of explicit blow up solutions $u_{T^*}: (0,T^*) \rightarrow C^{\infty}(\mathbb R)$
\begin{equation}
\label{eq:blow-upsol}
 u_{T^*}(x,t) = \left(\tfrac{T^*}{T^*-t} \right)^{1/2} e^{i \frac{x^2}{4(t-T^*)} + \frac{tT^*}{T^*-t}} \varphi \left( \tfrac{T^*x}{T^*-t} \right)\operatorname{ with }\varphi(x) = \sqrt{\frac{3^{1/2}}{\cosh(2x)}}. 
 \end{equation}
 The singularity of the blow up solution \eqref{eq:blow-upsol} is also numerically visible but stops to increase after some time $t$ and thus exists for all times.

\begin{appendix}
\section{Cubic discretization}
\label{sec:CD}
In this section, we prove the rest of Proposition \ref{convrate}.
\begin{proof}
Let $f \in W^{n,p}(\mathbb R^d) \cap C^{\infty}(\mathbb R^d)$ and a cubic discretization $Q$ of side length $h$ be given, then we can define  
\begin{equation}
\begin{split}
 (\delta_h^1f)(x) &=  \fint_0^h  (\partial_1 f)(x+s \widehat{e}_1 ) \ ds \\
  ((\delta_h^1)^nf)(x) &= \fint_0^h  (\partial_1 (\delta^1_h)^{n-1} f)(x+s \widehat{e}_1 ) \ ds \\
  &=  \fint_0^h... \fint_0^h  (\partial_1^2 f)(x+(s_1+..+s_n) \widehat{e}_1 ) \ ds_n...\ ds_1.
\end{split}
\end{equation}
where $\widehat{e}_1=(1,0,..,0)$ is the unit vector.
 
We then have that 
\[ ((\delta_h^1)^nf)(x)-(\partial_1^nf)(x) = \fint_0^h...\fint_0^h \int_0^1  (\partial_1^{n+1} f)(x+u(s_1+..+s_n) \widehat{e}_1 ) (s_1+...+s_n) \ du \ ds_n... \ ds_1 \]
which implies that 
\[ \int_{\mathbb R^d} \left\lvert ((\delta_h^1)^nf)(x)-(\partial_1^nf)(x) \right\rvert^p \ dx \lesssim h^p \int_{\mathbb R^d} \left\lvert (\partial_1^{n+1}f)(x) \right\rvert^p \ dx \]
and $\left\lVert (\delta_h^1)^nf-\partial_1^nf \right\rVert_{L^{\infty}}  \lesssim h   \left\lVert \partial_1^{n+1}f \right\rVert_{L^{\infty}}.$

If we then define the function $g_h(y):= \fint_0^h...\fint_0^h \partial_1^{n} f (y+ (s_1+..+s_n) \widehat{e}_1) \ ds_n ... ds_1$, then we see that analyzing the convergence of $(\delta_h^1)^n f$ to $(\delta_h^1)^n f_Q$ is equivalent to analyzing the convergence of $g_h$ to $(g_h)_Q.$

To see this, it suffices to observe that 
\begin{equation}
\begin{split}
(\delta_h^1 f_Q )(x)
&= \sum_{j \in \mathbb Z^d} \frac{ \fint_{Q_j} f(u+\widehat{e}_1 h) \ du- \fint_{Q_j} f(u-\widehat{e}_1  h) \ du}{2h} \indic_{Q_j}(x)   \\
&= \sum_{j \in \mathbb Z^d} \fint_{Q_j} \left(\frac{ f(u+\widehat{e}_1 h) -  f(u-\widehat{e}_1  h) }{2h} \right) \ du\indic_{Q_j}(x) =( \delta_h^1f)_Q(x)
\end{split}
\end{equation}  
and similarly for higher derivatives.

Let $f (x)= \sum_{j \in \mathbb Z^d} f(x) \indic_{Q_{x_i}}(x)$ and $f_Q(x) =\sum_{j \in \mathbb Z^d}\fint_{Q_{x_j}} f(u) du \indic_{Q_{x_j}}(x).$
Thus 
\[  \int_{\RR^d} \left\lvert f(x) - f_Q(x) \right \rvert^p  \ dx  =\sum_{j \in \mathbb Z^d} \int_{Q_{x_j}} \left\lvert f(x) - \fint_{Q_{x_j}} f(y) dy \right\rvert^p \ dx. \]

We can then use Poincar\'e's inequality
\[ \int_{Q_{x_j}} \left\lvert f(x) - \fint_{Q_{x_j}} f(y) dy \right\rvert^p \ dx \le d^p h^p \int_{Q_{x_j}} \vert \nabla f(x) \vert^p \ dx\]
or conclude directly when $p=\infty$ that since for the curve $\gamma(s)=sx + (1-s)y$ 
we have that 
$
f(x) -f(y) = \int_{0}^1 \langle Df(\gamma(s)),x-y \rangle \ ds
$
 this implies
\[\sup_{x \in Q_{x_j}} \left\lvert  \fint_{Q_{x_j}} f(x) - f(y) dy \right\rvert  \lesssim h  \Vert D f \Vert_{L^{\infty}}.\]
Hence, we have that 
\begin{equation}
\begin{split}
 &\int_{\RR^d} \left\lvert f(x) - f_Q(x)  \right \rvert^p  \ dx  =d^p h^p \int_{\RR^d} \vert \nabla f(x) \vert^p \ dx \text{ and } \\
 &\sup_{j \in \mathbb{Z}^d, \ x \in Q_{x_j}} \left\lvert  \fint_{Q_{x_j}} f(x) - f(y) dy \right\rvert  \lesssim h  \Vert D f \Vert_{L^{\infty}}. 
\end{split}
\end{equation} 
The statement for general $\varepsilon$ follows then for example from interpolation:
Consider the complex interpolation spaces $(W^{s_0,p}(\RR^d),W^{s_1,p}(\RR^d))_{\varepsilon} = W^{(1-\varepsilon)s_0+\varepsilon s_1,p}(\RR^d),$ then interpolation gives that the operator $T_{\varepsilon}: W^{(1-\varepsilon)s_0+\varepsilon s_1,p}(\RR^d)\rightarrow L^2(\RR^d)$ 
\[ T_{\varepsilon}(f) :=\Delta^{h} f_Q-\Delta f \]
is bounded with operator norm $\Vert T_{\varepsilon} \Vert \le \Vert T_1 \Vert^{\varepsilon} \Vert T_0 \Vert^{1-\varepsilon} = \mathcal O(h^{\varepsilon}).$

\end{proof}

\end{appendix}

\end{document}